\documentclass[a4paper]{amsart}
\pdfoutput=1

\usepackage[utf8]{inputenc}
\usepackage[english]{babel}
\usepackage{amssymb}
\usepackage{amsmath}
\usepackage{amsthm}
\usepackage{array}
\usepackage{eucal}
\usepackage{verbatim}
\usepackage{subfigure}
\usepackage[pdftex]{color, graphicx}
\usepackage{fancyhdr}
\usepackage{tikz}
\usepackage{tikz-cd}
\usepackage{amsfonts}
\usepackage{mathrsfs}
\usepackage{mathtools}
\usepackage{graphicx}
\usepackage[final]{pdfpages}
\usepackage{algorithm}
\usepackage{algorithmic}
\usepackage{enumitem}
\usepackage{pdflscape}
\usepackage{url}

\usepackage{hyperref}
\usepackage[all]{hypcap}        % Convenience: Let hyperref jump to the figure instead to the caption only.

\usetikzlibrary{arrows,shapes,decorations.pathmorphing,patterns,backgrounds,fit,positioning,matrix}
\tikzset{dot/.style={circle,fill=black,thick,inner sep=0pt,minimum size=1mm,draw}}
\tikzset{arrow/.style={semithick,>=stealth',shorten >=1pt,shorten <=1pt}}
\tikzset{equal/.style={kant,double distance=2pt}}

\numberwithin{equation}{section}
\numberwithin{figure}{section}
\numberwithin{table}{section}
\newcommand{\inputhack}{\input}                                         % Rename \input, so that a very simple awk-skript can merge all tex-files.
\graphicspath{{pictures/}}                                              % Set path of the pictures. This is used by the pdf_tex-files generated by inkscape.
\newcommand{\inkpic}[2][\columnwidth]{
  \def\svgwidth{#1}
  \inputhack{#2.pdf_tex}
}

\theoremstyle{plain}
\newtheorem{theorem}{Theorem}[section]
\newtheorem{theorema}{Theorem}

\newtheorem{proposition}[theorem]{Proposition}
\newtheorem{propositiona}[theorema]{Proposition}

\newtheorem{lemma}[theorem]{Lemma}
\newtheorem*{lemma*}{Lemma}

\theoremstyle{definition}
\newtheorem{definition}[theorem]{Definition}
\newtheorem{example}[theorem]{Example}
\newtheorem{discussion}[theorem]{Discussion}
\theoremstyle{remark}

\newtheorem{remark}[theorem]{Remark}
\newtheorem{notation}[theorem]{Notation}

\newenvironment{cond}
  {\par\vspace{\abovedisplayskip}\noindent\begin{tabular}{l @{ \hspace{3ex}} l}}
  {\end{tabular}\par\vspace{\belowdisplayskip}}
\theoremstyle{plain}
\newtheorem{innerlemma}{Lemma}
\newenvironment{lemmawithfixednumber}[1]
  {\renewcommand\theinnerlemma{#1}\innerlemma}
  {\endinnerlemma}

\definecolor{ForestGreen}{rgb}{0.134,0.545,0.134}

\DeclareMathAlphabet{\mathpzc}{OT1}{pzc}{m}{it}

\newcommand{\N}{\mathbb N}
\newcommand{\Z}{\mathbb Z}

\newcommand{\Q}{\mathbb Q}

\DeclareMathOperator{\degen}{degen}

\DeclareMathOperator{\id}{id}
\DeclareMathOperator{\imag}{im}

\DeclareMathOperator{\Symm}{Symm}

\DeclareMathOperator{\Tymm}{T}

\newcommand{\cof}{\hookrightarrow}
\newcommand{\fib}{\twoheadrightarrow}

\newcommand{\xr}[1]{ \xrightarrow{\hspace{6pt}{#1}\hspace{6pt}} }
\newcommand{\xhr}[1]{ \xhookrightarrow{\hspace{6pt}{#1}\hspace{6pt}} }

\newcommand{\xlr}[1]{ \xleftrightarrow {\hspace{6pt}{#1}\hspace{6pt}} }

\newcommand{\mspc}[2]{ \hspace{#2pt}\text{#1}\hspace{#2pt} }

\newcommand{\Fat}{\mathfrak{Fat}}

\newcommand{\Mod}{\mathrm{Mod}}
\newcommand{\Diff}{\mathit{Diff}}

\newcommand{\SD}{\mathcal{SD}}
\newcommand{\SDs}{\mathscr{SD}}
\newcommand{\cM}{\mathcal{M}} %Moduli space
\newcommand{\Rad}{\mathfrak{Rad}} %Rad
\newcommand{\bRad}{\overline{\mathfrak{Rad}}}%Harmonic
\newcommand{\uRad}{\mathfrak{U}\overline{\mathfrak{Rad}}}%Harmonic
\newcommand{\ad}{\ensuremath{\,\protect\scalebox{.7}{$\searrow$}\,}}
\newcommand{\au}{\ensuremath{\,\protect\scalebox{.7}{$\nearrow$}\,}}
\newcommand{\betterwidetilde}[1]{\smash{\widetilde{#1}}{\vphantom{#1}}} % A wide tilder that looks better.
\newcommand{\del}{\partial}

\newcommand{\finv}{\ensuremath{C[F^{-1}]}}

\newcommand{\ov}{\overline}

\newcommand{\SDgm}{\ensuremath{\SD(g,m)}}

\mathchardef\mhyphen="2D
\newcommand{\xh}[1]{\ensuremath{#1\hspace{1pt}\mhyphen}}
\newcommand{\ph}{\ensuremath{p\hspace{1pt}\mhyphen}}

\DeclareRobustCommand\bbid{{\ensuremath{%
  \tikz[baseline={([yshift=-3.5pt]current bounding box.center)}, x=1ex, y=1ex]{%
    \draw[line width=.8pt] (0,0) -- (2,0) node[shape=circle, fill=black, inner sep=1pt] {};
    \draw node at (-1,0) {$1$};
    \draw node at (4,0) {$v_0$};
  }%
}}}

\DeclareRobustCommand\bbunit{{\ensuremath{%
  \tikz[baseline={([yshift=-3.5pt]current bounding box.center)}, x=1ex, y=1ex]{%
    \draw node[shape=circle, fill=black, inner sep=1pt] at (0,0) {};
    \draw node at (2,0) {$v_0$};
  }%
}}}

\DeclareRobustCommand\bbmult{{\ensuremath{%
  \tikz[baseline={([yshift=-3.5pt]current bounding box.center)}, x=1ex, y=1ex]{%
    \draw[line width=.8pt] (0,1)  -- (2,0) node[shape=circle, fill=black, inner sep=1pt] {} -- (4,0) node[shape=circle, fill=black, inner sep=1pt] {};
    \draw[line width=.8pt] (0,-1) -- (2,0);
    \draw node at (-1,1) {$1$};
    \draw node at (-1,-1) {$2$};
    \draw node at (6,0) {$v_0$};
  }%
}}}

\DeclareRobustCommand\bbcomult{{\ensuremath{%
  \tikz[baseline={([yshift=-3.5pt]current bounding box.center)}, x=1ex, y=1ex]{%
    \draw[line width=.8pt] (0,0) -- (2,0)  node[shape=circle, fill=black, inner sep=1pt] {} -- (4,1) node[shape=circle, fill=black, inner sep=1pt] {};
    \draw[line width=.8pt] (2,0) -- (4,-1) node[shape=circle, fill=black, inner sep=1pt] {};
    \draw node at (-1,0) {$1$};
    \draw node at (6,1) {$v_0$};
    \draw node at (6,-1) {$v_1$};
  }%
}}}

\DeclareRobustCommand\bbcomultunit{{\ensuremath{%
  \tikz[baseline={([yshift=-3.5pt]current bounding box.center)}, x=1ex, y=1ex]{%
    \draw[line width=.8pt] (2,1) node[shape=circle, fill=black, inner sep=1pt] {} to[out=180, in=90] (0,0) to[out=270, in=180] (2,-1) node[shape=circle, fill=black, inner sep=1pt] {};
    \draw node at (4,1) {$v_0$};
    \draw node at (4,-1) {$v_1$};
  }%
}}}

\DeclareRobustCommand\bbcounitmult{{\ensuremath{%
  \tikz[baseline={([yshift=-3.5pt]current bounding box.center)}, x=1ex, y=1ex]{%
    \draw[line width=.8pt] (0,1) to[out=0, in=90] (2,0) to[out=270, in=0] (0,-1);
    \draw node at (-1,1) {$1$};
    \draw node at (-1,-1) {$2$};
  }%
}}}

\DeclareRobustCommand\bbtwist{{\ensuremath{%
  \tikz[baseline={([yshift=-3.5pt]current bounding box.center)}, x=1ex, y=1ex]{%
    \draw[line width=.8pt] (0,-1) to (2,1)  node[shape=circle, fill=black, inner sep=1pt] {};
    \draw[line width=2pt, color=white] (0,1)  to (2,-1);
    \draw[line width=.8pt] (0,1)  to (2,-1) node[shape=circle, fill=black, inner sep=1pt] {};
    \draw node at (-1,1) {$1$};
    \draw node at (-1,-1) {$2$};
    \draw node at (4,1) {$v_0$};
    \draw node at (4,-1) {$v_1$};
  }%
}}}

\author{Felix Jonathan Boes and Daniela Egas Santander}
\title{On the homotopy type of the space of Sullivan Diagrams}

\begin{document}

\begin{abstract}
We study the homotopy type of the harmonic compactification of the moduli space of a $2$-cobordism S with one outgoing boundary component, or equivalently of the space of Sullivan diagrams of type $S$ on one circle.
Our results are of two types: vanishing and non-vanishing.
In our vanishing results we are able to show that the connectivity of the harmonic compactification increases with the number of incoming boundary components.
Moreover, we extend the genus stabilization maps of moduli spaces to the harmonic compactification and show that the connectivity of these maps
increases with the genus and number of incoming boundary components.
In our non-vanishing results we compute the non-trivial fundamental group of the harmonic compactification of the cobordism $S$ of any genus with two unenumerated punctures and empty incoming boundary.
Moreover, we construct five infinite families of non-trivial homology classes of the harmonic compactification, two of which correspond to non-trivial higher string topology operations.
\end{abstract}

\maketitle

\section{Introduction}

Let $S(p,g,m)$ denote an oriented $2$-cobordism with $m \ge 1$ parametrized, enumerated incoming boundary components,
genus $g$ and $p \ge 1$ parametrized, enumerated outgoing boundary components.
We denote by $\cM(p,g,m)$ the moduli space of the surface $S$ underlying the cobordism $S(p,g,m)$.
Intuitively $\cM(p,g,m)$ is a space that parametrizes all conformal structures the underlying surface $S$ can have up to a notion of equivalence.
Note in particular that the homotopy type of $\cM(p,g,m)$ depends only on the underlying surface $S$ and not on its cobordism structure.

We denote by $\Diff^+(S, \partial S)$ the space of orientation preserving self-diffeo\-morphisms of $S$ which fix the boundary point-wise.
Composition of diffeomorphisms give $\Diff^+(S, \partial S)$ the structure of a topological group.
The mapping class group of $S$, denoted $\Mod(S)$, is the group of components of the diffeomorphism group i.e.,
\[\Mod(S):= \pi_0(\Diff^+(S, \partial S)).\]
Since the underlying surface has at least one boundary component, its moduli space is a model of its corresponding mapping class group i.e., 
\[\cM(p,g,m)\simeq B\Mod(S).\]

In \cite{bodigheimer}, B\"{o}digheimer constructs a model for $\cM(p,g,m)$ which carries a natural notion of compac\-tification called the harmonic compactification.
In contrast to moduli space, the homotopy type of this compactification does not only depend on the underlying surface but also on its cobordism structure.
In this paper we study the homotopy type of the harmonic compactification for the case $p=1$.
Our results are of two types: vanishing and non-vanishing.

In our vanishing results we  are able to show that the connectivity of the harmonic compactification increases with $m$.
Moreover, we extend the genus stabilization maps of moduli spaces to the harmonic compactification and show that the connectivity of these maps 
increase with genus and number of incoming boundary components.
The proofs of our vanishing results are not based on standard methods used to prove homological stability.
Instead, we make use of B\"{o}digheimer's combinatorial description of the moduli space of surfaces introduced in \cite{bodigheimerold, bodigheimer} together with methods from discrete Morse theory.
Finally, we get similar results for the cases when the cobordism $S$ has either empty incoming boundary and $m$ punctures or $m$ unenumerated incoming boundary components.

In our non-vanishing results we compute the non-trivial fundamental group of the harmonic compactification of the cobordism $S$ of any genus, with one outgoing boundary component, no incoming boundary components and two unenumerated punctures.
Moreover, combining the work of Wahl and Westerland in \cite{wahlwesterland, wahluniversal} with new methods, we construct five infinite families of non-trivial homology classes of the harmonic compactification, two of which correspond to non-trivial higher string topology operations.

\medskip

Before stating our results more concretely, we make precise the relationship between the harmonic compactification, the moduli space of Riemann surfaces,  Sullivan diagrams and string topology.
Fix an oriented cobordism $S = S(p,g,m)$ as above.
In \cite{bodigheimer}, B\"{o}digheimer describes how any surface with a conformal structure of underlying topological type $S$ can be constructed from $p$ annuli in $p$ disjoint complex planes 
by first cutting slits into these annuli in a radial direction
and then gluing the borders of these slits together in such a way that one obtains an orientable surface with boundary.
With this in mind, he constructs the space of radial slit domains $\Rad(p,g,m)$ which is a combinatorial model for the moduli space $\cM(p,g,m)$.
This model carries a natural notion of compactification called the harmonic compactification which we denote by $\ov\cM(p,g,m)$.
Geometrically, this compactification allows handles to degenerate to intervals and finitely many points to become identified on the boundary as long as there is a path going from the incoming boundary to the outgoing boundary that does not go through a degeneration.
It is called the harmonic compactification since certain uniquely determined harmonic functions used to construct the model for $\cM(p,g,m)$ also exist on the degenerate surfaces in $\ov\cM(p,g,m)$.
In contrast to moduli space, for two fixed cobordisms $S$ and $S'$ with the same underlying surface but with different distribution of the boundaries into incoming and outgoing boundary components,
their associated harmonic compactifications need not be homotopy equivalent.

On the other hand, Godin uses the ideas of Penner and Igusa to construct a space of (metric) fat graphs $\Fat(p,g,m)$ which is also a model of the moduli space $\cM(p,g,m)$ \cite{Godinunstable, godin}.
Informally, a fat graph is a graph together with a cyclic ordering of the edges incident at each vertex; 
and a metric fat graph is a fat graph with lengths on its edges.
This space of fat graphs has a homotopy equivalent subspace, called the space of $p$-admissible fat graphs which we denote by $\Fat^{ad}(p,g,m)$.
Intuitively, a $p$-admissible fat graph is a fat graph on $p$ disjoint circles.
See Definition \ref{definition:admissible_graph} for a precise description.
The space of $p$-Sullivan diagrams of type $S$, here denoted by $\betterwidetilde{\SDs}(p,g,m)$, is a quotient of $\Fat^{ad}(p,g,m)$.
See Definition \ref{definition:space_of_SD} for a precise description of this space.
As for the case of the harmonic compactification, given two fixed cobordisms $S$ and $S'$ with the same underlying surface but with different distribution of the boundaries into incoming and outgoing boundary components,
their associated spaces of Sullivan diagrams need not be homotopy equivalent.

The space of $p$-Sullivan diagrams has a canonical CW-structure and its associated cellular chain complex is the chain complex of $p$-Sullivan diagrams.
There is a cellular homotopy equivalence between the spaces $\betterwidetilde{\SDs}(p,g,m)$ and the harmonic compactification $\bRad(p,g,m)$, see \cite{egaskupers}.
Indeed, there is a homotopy commutative diagram.
\[
  \begin{tikzcd}[row sep=2em, column sep=1em]
    \Fat^{ad}(p,g,m) \ar[twoheadrightarrow]{d} \ar[phantom]{r}{\simeq} & \cM(p,g,m) \ar[phantom]{r}{\simeq} & \Rad(p,g,m) \ar[hook]{d} \\
    \betterwidetilde{\SDs}(p,g,m) \ar{rr}{\simeq} & & \bRad(p,g,m)
  \end{tikzcd}
\]

Therefore, studying the homotopy type and in particular the homology of Sullivan diagrams is of interest in the study of the homology of the moduli space of Riemann surfaces.
We make this idea precise.
By gluing a genus one surface with two boundary components to the outgoing boundary of $S(1,g,m)$ we obtain stabilization maps
\begin{equation*}
  \cM(1,0,m) \to \cM(1,1,m) \to \ldots \to \cM(1,g,m) \to \cM(1,g+1,m) \to \ldots
\end{equation*}
In \cite{Harer_stable}, Harer showed that these maps induce isomorphisms in a range of dimensions increasing with the genus
and the stable homology is well understood by the work of Tillmann, Madsen-Weiss and Galatius in \cite{MadsenWeiss, tillmann, galatius_modp}.
However, little is known about the unstable homology of $\cM(p,g,m)$.
Explicit computations for small genus and small number of boundary components and punctures are given in \cite{ehrenfried, Godinunstable, abhau, wang, mehner, BoesHermann}.
The map from the space of $p$-admissible fat graphs onto Sullivan diagrams induces a map in homology
\[
  H_\ast(\cM(p,g,m)) \cong H_\ast(\Fat^{ad}(p,g,m)) \to H_\ast(\betterwidetilde{\SDs}(p,g,m)).
\]
In \cite{wahlwesterland}, Wahl and Westerland show that for $m\geq 2$ parametrized, enumerated incoming boundaries this map is trivial on the stable classes of the moduli space of surfaces.
In contrast, one of our non-vanishing results shows that the map is non-trivial if one relaxes the conditions on the $m$ incoming boundaries,
see Proposition \ref{intro:prop_pi1} in this introduction.
Studying the homotopy type, in particular the homology, of the space of Sullivan diagrams is a much more tractable problem than studying the homology of the moduli space of surfaces.
The hope is that understanding the homology of the former would give further insight towards understanding the unstable homology of the latter.

The study of the homotopy type of Sullivan diagrams is also of interest in the field of string topology, which studies algebraic structures on the homology of free loop spaces.
Let $LM$ be the free loop space of a manifold $M$.
In string topology one constructs operations
\begin{equation*}
  H_{*}(LM)^{\otimes m} \longrightarrow H_{\bullet}(LM)^{\otimes p}
\end{equation*}
parametrized by a \emph{space of operations} and subject to compatibility conditions such that they assemble into some sort of field theory.
Our chain complex of Sullivan diagrams agrees with the one defined by Tradler and Zeinalian in \cite{TradlerZeinalian} and by Wahl and Westerland in \cite{wahlwesterland}
to parametrize operations on the Hochschild homology of algebras with a given structure.

On the other hand, there are other versions of spaces of (Sullivan) chord diagrams appearing in the string topology literature as spaces of operations.
We briefly compare those spaces to the space of Sullivan diagrams which we study in this paper.
Whenever the cobordism $S$ has no punctures,
our space of Sullivan diagrams is homeomorphic to the underlying space of the Sullivan PROP described by Kaufmann in \cite{kaufmann_sullivan}.
Restricting further to the case of genus zero and one incoming boundary component (i.e., $m=1$) our chain complex of $p$-Sullivan diagrams is isomorphic to the chain complex of trees with spines on $p$-white vertices constructed by Ward in \cite{Ward}.
Furthemore in \cite{poirierrounds}, Poirier and Rounds construct string operations using a space of chord diagrams $\overline{SD}$ and they describe a quotient of this space $\overline{SD}/_\sim$ through which their operations factor.
This quotient space is homeomorphic to our spaces of Sullivan diagrams.
In a sequel \cite{drummond_poirier_rounds}, Drummond-Cole, Porier, and Rounds use yet another chain complex of chord diagrams.
Although there is a function from their space of chord diagrams to our version of Sullivan diagrams, it is currently unclear whether this assignment is a homotopy equivalence or even continuous.
In \cite{cohengodin}, Cohen and Godin construct string operations using yet another space of chords diagrams.
Although the concepts are very closely related, these spaces are neither homotopy equivalent nor are their corresponding chain complexes quasi-isomorphic.
Finally, a related space of chord diagrams also occur in upcoming work of Hingston and Wahl on the study of geometric string topology operations.
The concrete relationship between their space of operations and the one we study in this paper is yet to be understood.

\medskip

We now present our results more concretely.
Consider $S(1,g,m)$, the oriented cobordism with
$m \ge 1$ parametrized, enumerated incoming boundary components, genus $g$ and
$1$ parametrized outgoing boundary component.
We denote by $\betterwidetilde{\SDs}_{g,m}:=\betterwidetilde{\SDs}(1,g,m)$ its space of 1-Sullivan diagrams.
Inspired by B\"{o}digheimer's combinatorial description of the moduli space of surfaces described in \cite{bodigheimerold, bodigheimer, abhau, BoesHermann} we give an alternative presentation of Sullivan diagrams in terms of permutations with weights
which resembles the concept of stable graphs of \cite{kon92, Loo95}.
See Proposition \ref{proposition:combinatorial_1_SD}.
Exploiting this description, we prove the following two theorems.

\begin{theorema}
\label{intro:thma}\hspace{2mm}
  Let $g \ge 0$ and $m \ge 2$.
  The space $\betterwidetilde{\SDs}_{g,m}$ is $(m-2)$-connected.
\end{theorema}

We extend the stabilization maps of moduli spaces to their corresponding spaces of Sullivan diagrams, i.e.\, we give maps $\varphi \colon \betterwidetilde{\SDs}_{g,m} \to \betterwidetilde{\SDs}_{g+1,m}$ making the following diagram commute.
\[
  \begin{tikzcd}
    \cM(1,g,m) \ar{d} \ar{r} & \cM(1,g+1,m) \ar{d} \\
    \betterwidetilde{\SDs}_{g,m} \ar{r}{\varphi} &\betterwidetilde{\SDs}_{g+1,m}
  \end{tikzcd}
\]

\begin{theorema}
\label{intro:thmb}\hspace{2mm}
  Let $g \ge 0$ and $m > 2$.
  The stabilization map $\varphi$ is $(g + m - 2)$-connected.
\end{theorema}

Additionally, we consider the space of 1-Sullivan diagrams of the cobordism $S$ of genus $g$, with one parametrized outgoing  boundary component and either:
\begin{itemize}
\item $m$ parametrized and unenumerated incoming boundary components, which we denote by $\SDs_{g,m}$;
\item empty incoming boundary and $m$ enumerated punctures, which we denote by $\betterwidetilde{\SDs}_g^m$;
\item empty incoming boundary and $m$ unenumerated punctures,  which we denote by $\SDs_g^m$.
\end{itemize}
For each of these models of spaces of 1-Sullivan diagrams, there is a version of Theorem \ref{intro:thma} and Theorem \ref{intro:thmb} (see Theorems \ref{theorem_a} and \ref{theorem_b}).
Moreover, the structure of the proofs show that a huge amount of cells do not contribute to the homology.
See Proposition \ref{proposition:support_of_homology} for details.

There are forgetful maps of spaces
\begin{align*}
  \begin{tikzcd}[ampersand replacement=\&]
                                                                                      \& \betterwidetilde{\SDs}_g^m \arrow{dr}{\omega} \& \\
    \betterwidetilde{\SDs}_{g,m} \arrow{ur}{\tilde\vartheta} \arrow{dr}{\hat\omega}   \&                                               \& \SDs_g^m \\
                                                                                      \& \SDs_{g,m} \arrow{ur}{\vartheta}              \&
  \end{tikzcd}
\end{align*}
where the maps $\tilde\vartheta$ and $\vartheta$  collapse each incoming boundary component to a puncture, and the maps $\tilde\omega$ and $\omega$  forget the enumeration of the incoming boundary components or the enumeration of the punctures respectively.
In contrast to the forgetful maps of moduli spaces, none of these maps is a fibration or a covering since the homotopy type of the fibers is not constant.
However, by passing to homology, the maps $\omega$ and $\hat\omega$ that forget the enumeration admit a transfer map.
We get the result below, see Proposition \ref{proposition:transfer_maps} for details.
\begin{propositiona}\hspace{2mm}
  In homology there are maps
  \begin{align*}
    tr \colon H_\ast(\SDs_{g,m}; \Z ) \to H_\ast( \betterwidetilde{\SDs}_{g,m}; \Z ) \;\;\;\;\text{ and }\;\;\;\; tr \colon H_\ast( \SDs_g^m; \Z ) \to H_\ast( \betterwidetilde{\SDs}_g^m; \Z )
  \end{align*}
  where $\hat\omega_\ast \circ tr$ respectively $\omega_\ast \circ tr$ are the multiplication by $m!$.
\end{propositiona}

By Theorem \ref{intro:thma}, the spaces of 1-Sullivan diagrams are simply connected provided that we have at least $m \ge 3$ incoming boundaries or $m \geq 3$ punctures.
In Proposition \ref{proposition:fundamental_group_two_punctures} we show that this range is strict.
More precisely, we show the following:

\begin{propositiona}
\label{intro:prop_pi1}
\hspace{2mm}
Let $S_{g,1}^2$ denote the cobordism with no incoming boundary, genus $g$, one parametrized outgoing boundary component and two punctures and denote by $\cM(S_{g,1}^2)$ its moduli space.
The fundamental group of its corresponding space of 1-Sullivan diagrams $\SDs_{g}^2$ is
  \begin{align*}
    \pi_1( \SDs_g^2 ) \cong 
    \begin{cases}
      \mathbb Z\langle \alpha_0 \rangle & g=0 \\
      \mathbb Z / 2\mathbb Z \langle \alpha_g \rangle  & g > 0
    \end{cases}
  \end{align*}
  The homomorphism on fundamental groups induced by the stabilization map 
  \[
    \varphi \colon \pi_1(\SDs_{g}^2) \to \pi_1(\SDs_{g+1}^2)
  \] 
  sends $\alpha_g$ to $\alpha_{g+1}$.
  Furthermore,the class $\alpha_g$ is in the image of the homomorphism induced by the map of spaces.
  \[
    B\Mod(S_{g,1}^2) \to \SDs_g^2
  \]
  The source is the classifying space of the mapping class group of the surface $S_{g,1}^2$ with one boundary component, genus $g$ and two punctures;
  and the target is the space of Sullivan diagrams $\SDs_g^2$.
  After taking fundamental groups, the generator $\alpha_g$ is in the image of the Dehn (half-)twist that exchanges the two punctures inside a small disk.
\end{propositiona}

Using the methods of \cite{wahluniversal}, we extend her results and construct two infinite families of non-trivial classes of increasing dimension.
See Proposition \ref{proposition:non_trivial_op}.
The forgetful maps and transfers yield three additional infinite families of non-trivial classes.
See Proposition \ref{proposition:generators_by_operations}.
We summarize these results below.

\begin{propositiona}\hspace{2mm}
    Let $m>0$, $1\leq i \leq m$, $c_i>1$ and $c=\sum_i c_i$.
  \begin{enumerate}
    \item[$(i)$]  There are classes of infinite order $\widetilde{\Gamma}_m\in H_{4m-1}(\betterwidetilde{\SDs}_{m,m}; \mathbb Z)$ and $\widetilde{\Omega}_{(c_1, \ldots, c_m)}\in H_{2c-1}(\betterwidetilde{\SDs}_{0,c}; \mathbb Z)$.
      All these classes correspond to non-trivial higher string topology operations.
    \item[$(ii)$] There are classes of infinite order $\Gamma_m\in H_{4m-1}({\SDs}_{m,m}; \mathbb Z)$ and $\Omega_{(c_1, \ldots, c_m)}\in H_{2c-1}({\SDs}_{0,c}; \mathbb Z)$ and $\zeta_{2m} \in H_{2m-1}({\SDs}_{0,2m}; \mathbb Z)$ .
  \end{enumerate}
\end{propositiona}

Furthermore, we compute the homology of $\SDs_{g}^m$ and $\SDs_{g,1+m}$ for small $m$ and $g$.
In particular the homology groups of $\SDs_{0}^m$, the space of 1-Sullivan diagrams corresponding to the disk with $m$ unenumerated punctures, seem to indicate a general pattern.
For $1\leq m \leq 7$ and $m\neq 6$ we have:  
\[
  H_i(\SDs_{0}^m;\mathbb{Z})=\left\lbrace
    \begin{array}{l l}
      \mathbb{Z} & \text{for } i=0,m'+1\\
      0          & \text{else}
    \end{array}
  \right.
\]
where $m'$ is the biggest even number strictly smaller than $m$ and for $m=6,8$ we have:
\[
  H_i(\SDs_{0}^m;\mathbb{Z})=\left\lbrace
    \begin{array}{l l}
      \mathbb{Z} & \text{for } i=0,m-1, m+1, m+2\\
      0          & \text{else}.
    \end{array}
  \right.
\]
In fact, we believe that $\SDs_{0}^m$ is in general homotopy equivalent to a wedge of spheres of increasing dimension.
Further results for other genera can be found in Appendix \ref{appendix:computations}.

\medskip
Our results leave several perspectives and open questions for future work, which we briefly describe now.
Theorem \ref{intro:thma} implies in particular that 1-Sullivan diagrams have homological stability with respect to the number of incoming boundary components and that the stable homology is trivial.
Moreover, the stabilization maps with respect to genus and incoming boundary components commute.
Thus, the stable homology with respect to both maps is trivial as well.
However, for a fixed number of incoming boundary components $m$ the stable homology with respect to genus is still unknown.

On the other hand, while in this paper we focus on the case of 1-Sullivan diagrams, we expect there are versions of Theorems \ref{intro:thma} and \ref{intro:thmb} for the spaces of $p$-Sullivan diagrams for any $p$.
More precisely, in future work we intend to extend our current methods in order to show that the spaces of $p$-Sullivan diagrams have homological stability with respect to the number of incoming boundary components.
Moreover, we want to address if, as in the case $p=1$, the connectivity of these spaces increases with the number of incoming boundary components.
Similarly, we believe that extensions of these tools can be used to show that the spaces of $p$-Sullivan diagrams have homological stability with respect to the genus and
that both stabilization maps commute in homology.

Finally, turning back to moduli spaces, we hope that Sullivan diagrams might be used as a ``detector" of non-trivial homology classes of moduli spaces.
More precisely, we expect the map
\[
  \cM(p,g,m) \simeq \Rad(p,g,m) \to \ov{\Rad}(p,g,m) \simeq \betterwidetilde\SDs(p,g,m)
\]
to be non-trivial in homology.
Our work shows that the homology of the harmonic compactification is much easier to compute and
understanding generators of its homology could help us detect non-trivial classes of the moduli space of surfaces.
Indeed, Proposition \ref{intro:prop_pi1} shows that this is possible,
although the non-trivial homology classes of the moduli space of surfaces which we can currently detect were already known to be non-trivial, see for example \cite{bodigheimertillmann}.

Moreover, in contrast to the case of moduli spaces, the homotopy type of the space of $p$-Sullivan diagrams does not only depend on the topological type of its underlying surface $S$ but it is also sensitive to the cobordism structure of $S$.
See Subsection \ref{subsec:asymmetry}.
By changing the cobordism structure of the underlying surface  
we would like to construct maps of spaces
\[\cM(p,g,1) \to \betterwidetilde{\SDs}(p,g,1)\to 
\betterwidetilde{\SDs}(p-1,g,2) \to
\betterwidetilde{\SDs}(p-2,g,3)\to \ldots \to
\betterwidetilde{\SDs}(1,g,p).
\]
The complexity of the space of Sullivan diagrams increases with the number of outgoing boundary components of the cobordism $S$.
Moreover, the results of \cite{kaufmann_sullivan, wahlwesterland} in genus zero suggest that this sequence of spaces could give an approximation of the homology of the moduli space $\cM(p,g,1)$ whose accuracy increases as one increases the integer $p$.
We hope that such constructions together with the homological computations described above might give further insight into the homology of the moduli space of surfaces.

\medskip
The paper is organized as follows.
In Section \ref{section:p_sullivan_diags} we define the space of $p$-Sullivan diagrams as a quotient of admissible fat graphs and define the chain complex of $p$-Sullivan diagrams.
In Section \ref{section:sd_one_circle} we give a way of describing unique representatives for the generators of the chain complex of Sullivan diagrams and describe the differential under this presentation for the case $p=1$.
In Section \ref{section:vanishing}, we present our vanishing results and give a brief review of discrete Morse theory, which is our main technical tool.
In Section \ref{section:detecting_homology}, we present our non-vanishing results.
In Section \ref{section:proof_AB}, we prove two technical lemmas used in Section \ref{section:vanishing}.
In Appendix \ref{appendix:computations} we present complete computations for the homology of Sullivan diagrams in the parametrized and unparametrized, unenumerated case for small genus and number of punctures or boundary components.

\medskip
\medskip

\emph{Acknowledgments.}
We would like to thank Carl-Friedrich B\"{o}digheimer and Nathalie Wahl for many extensive and illuminating conversations on different models of moduli spaces, homology operations and string operations as well as for reading earlier versions of this paper.
We would like to thank Oscar Randall-Williams for useful conversations on methods to prove homological stability.
We are thankful to Frank Lutz for early on conversations on how discrete Morse theory could be used to address this problem.
We would also like to thank Andrea Bianchi,  Gabriel Drummond-Cole and Alexander Kupers for insightful conversations as well as proofreading this paper.
Furthermore, we thank Elmar Vogt and Daniel L\"{u}tgehetmann for proof reading and helpful comments on the accessibility of this paper.
Finally, we are thankful to Rune Johansen, and Kasper Andersen for insights in methods on experimental mathematics.

\tableofcontents

\section{Sullivan diagrams basic definitions}
\label{section:p_sullivan_diags}

In this section we first define Sullivan diagrams as an equivalence class of certain types of fat graphs or ribbon graphs, which are graphs with additional combinatorial structure.
Then we give brief descriptions of alternative definitions of Sullivan diagrams appearing in the literature.
Thereafter, we briefly describe composition of Sullivan diagrams, which assembles them into a dg-PROP.
Finally, we describe a property of these diagrams, that is that they are ``asymmetric" with respect to their incoming and outgoing boundary components.

\subsection{Fat graphs and \texorpdfstring{$p$}{p}-Sullivan diagrams}
\label{section:p_sullivan_diags:definitions}
\begin{definition}
  A \emph{combinatorial graph} $G$ is a tuple $G=(V,H,s,i)$, with a finite set of vertices $V$,
  a finite set of \emph{half-edges} $H$, a source map $s\colon H\to V$ and an involution $i\colon H\to H$.
  The source map $s$ ties each half-edge to its source vertex and the involution $i$ attaches half-edges together.
  The \emph{valence} of a vertex $v\in V$ is the cardinality of the set $s^{-1}(v)$ and we denote it by $|v|$.
  A \emph{leaf} of a graph is a fixed point of $i$.
  An \emph{edge} of the graph is an orbit of $i$.
  Note that leaves are also edges.
  We call both leaves and edges connected to a vertex of valence one \emph{outer-edges} all other edges are called \emph{inner-edges}.
  See Figure \ref{figure:fattening_graph_ex}.
\end{definition}

\begin{definition}
  The \emph{geometric realization} of a combinatorial graph $G$ denoted $|G|$ is the space obtained by taking a copy of the half interval $[0,\frac{1}{2}]$ for each half-edge of $G$ and identifying: 
  \begin{itemize}
    \item $0\in [0,\frac{1}{2}]$ of all half edges with the same source;
    \item $\frac{1}{2}\in[0,\frac{1}{2}]$ of the half edges which are matched by the involution.
  \end{itemize}
\end{definition}

\begin{definition}
  A \emph{fat graph} $\Gamma=(G,\sigma)$ is a combinatorial graph together with a cyclic ordering $\sigma_v$ of the half edges incident at each vertex $v$.
  The cyclic orderings $\sigma_v$ define a permutation $\sigma \colon H \to H, \sigma(h) := \sigma_{s(h)}(h)$ which satisfies $s\sigma = s$.
  This permutation is called the \emph{fat structure} of the graph.
  Figure \ref{Fat_example} shows some examples of fat graphs.
\end{definition}
\begin{figure}[ht]
  \centering
  \inkpic[.6\columnwidth]{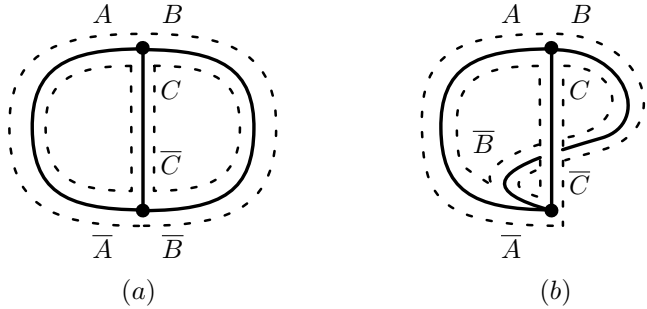}
  \caption{
    The thick lines show two different fat graphs with the same underlying combinatorial graph.
    The  fat structure is given by the clockwise orientation of the plane.
    The dotted lines indicate the boundary cycles.
    We make this graphical description precise using the names of the half edges indicated in the picture.
    (a) Fat structure: $\sigma = (ABC)(\ov A \ov C \ov B)$.
    Boundary cycles: $\omega = (A\ov C)(B\ov A)(C \ov B)$.
    (b) Fat structure: $\sigma = (ABC)(\ov A \ov B \ov C)$.
    Boundary cycles: $\omega: = (A \ov B C \ov A B \ov C)$.
  }
  \label{Fat_example}
\end{figure}

\begin{definition}
  The \emph{boundary cycles} of a fat graph $\Gamma = (G,\sigma)$ are the orbits of the permutation $\omega := \sigma\circ i \colon H \to H$.
  Let $c = (h_1, \ldots, h_k)$ be a boundary cycle of $\Gamma$.
  The \emph{boundary cycle sub-graph} of $c$ the sub-graph of $\Gamma$ uniquely defined by the set of half-edges $\{ h_1, i(h_1), h_2, i(h_2), \ldots \mid h_j \neq i(h_j) \}$.
  See Figure \ref{Fat_example} for an example.
  When clear from the context, we will refer to a boundary cycle sub-graph simply as boundary cycle.
\end{definition}

\begin{definition}[Surface with decorations]
  \label{definition:fattening}
  Given a fat graph $\Gamma = (V,H,s,i,\sigma)$, we construct a surface with marked points at the boundary by ``fattening" the geometric realization $|\Gamma|$.
  More precisely, we start with a collection of oriented disks for every vertex 
  \[\{D_v \mid v \in V\}\]
  and a collection of oriented strips for every half edge
  \[\{I_h := [0,{\textstyle\frac{1}{2}}] \times [-1,1]\mid h\in H\}.\]
  Preserving the orientations of the disks and the strips, we glue the boundary $\{0\} \times [-1,1]$ of every strip $I_h$ to the disk $D_{s(h)}$ in the cyclic order given by the fat structure of $\Gamma$.
  For each edge which is not a leaf, say $e=\{h, \overline{h}\}$, we glue $\{\frac1 2\}\times[-1,1]$ of $I_h$ to $\{\frac1 2\} \times[-1,1]$ of $I_{\overline{h}}$ via $-\id$ in the second factor (see Figure \ref{figure:fattening_graph_ex} (b)).
  
  We end our construction by collapsing each boundary component, which is not connected to a leaf or to a vertex of valence one, to a puncture.
  This procedure gives a surface with boundary, where each boundary component is connected to at least one leaf or one vertex of valence one.
  We interpret these as marked points on the boundary (see Figure \ref{figure:fattening_graph_ex} (c)).
  
  We call this surface together with the combinatorial data of the marked points at the boundary a \emph{surface with decorations} and denote it by $S_\Gamma$.
  The \emph{topological type} of a surface with decorations $S_\Gamma$ is its genus, number of boundary components and punctures 
  together with the cyclic ordering of the marked points at the boundary.
\end{definition} 
  
\begin{remark}
  \label{remark:fattening_graph}
  Note that there is a strong deformation retraction of a surface with decorations ${S_{\Gamma}}$ onto its fat graph $|\Gamma|$ so $\chi({S_{\Gamma}})=\chi(|\Gamma|)$.
  One can think of $|\Gamma|$ as the ``skeleton" of the surface $S_\Gamma$.
  Furthermore, the number of boundary components and punctures are completely determined by $\Gamma$.
  Indeed, the number of punctures of $S_\Gamma$ is equal to the number of cycles of $\omega$ which do not contain a half-edge that belongs to an outer-edge.
  The number of boundary components of $S_\Gamma$ is equal to the number of cycles of $\omega$ which contain at least one half-edge that belongs to an outer-edge.
  
  Therefore, the topological type of $S_\Gamma$ is completely determined by $\Gamma$.
  Moreover, given an inner-edge $e$ of $\Gamma$ which is not a loop, 
  collapsing $e$ gives a homotopy equivalence $|\Gamma|\stackrel{\simeq}{\to}|\Gamma/e|$ and does not change the number of boundary cycles or their decorations.
  Thus, the surfaces ${S_{\Gamma}}$ and ${S_{\Gamma/e}}$ have the same topological type.
\end{remark}

\begin{figure}[ht]
  \inkpic[.85\columnwidth]{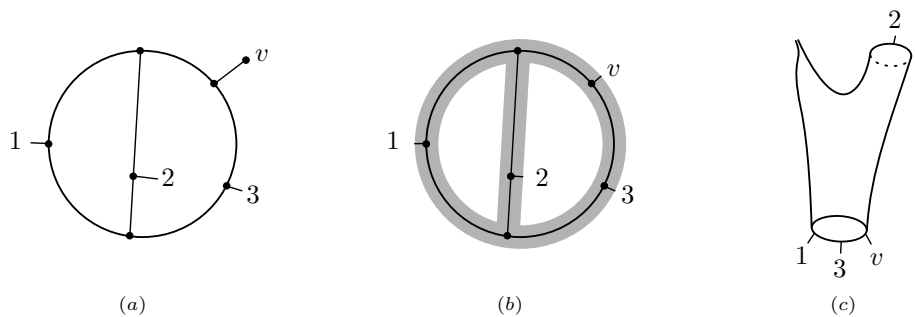}
  \caption{\label{figure:fattening_graph_ex} The picture shows (a) a fat graph with four outer-edges and six inner-edges; (b) the surface obtained from the fattening procedure;
    (c) the surface with decorations obtained from (b) by labeling points at the boundary and collapsing unmarked boundary components to a puncture.}
\end{figure}

\begin{definition}
  \label{definition:admissible_graph}
  A \emph{p-admissible fat graph} $\Gamma$ consists of:
  \begin{enumerate}
    \item an (isomorphism class) of fat graphs such that all vertices have valence at least $3$ and
    \item an enumeration of a subset of the set of leaves by $\{1,2,\ldots, k\}$ for some $k \ge p$
  \end{enumerate}
  such that:
  \begin{enumerate}
   \item each of the first $p$ leaves $l_1, \ldots, l_p$ is the only leaf on its boundary cycle $c_i$ for $1\leq i \leq p$,
   \item their corresponding boundary cycle sub-graphs 
   \[\Gamma_{c_1}, \Gamma_{c_2}, \ldots, \Gamma_{c_p}\]
   are disjointly embedded circles in $|\Gamma|$.
  \end{enumerate}
  We will refer to the first $p$ leaves as \emph{admissible leaves} and to their corresponding boundary cycles as \emph{admissible cycles}.
  Figure \ref{figure:fattening_graph_ex} (a) shows a fat graph that is not 1-admissible because the leaf with number 1 is not the only leaf in its boundary cycle.
  Figure \ref{figure:admissible_example} shows an example of a 3-admissible fat graph.
\end{definition}
\begin{figure}[ht]
  \centering
  \inkpic[.5\columnwidth]{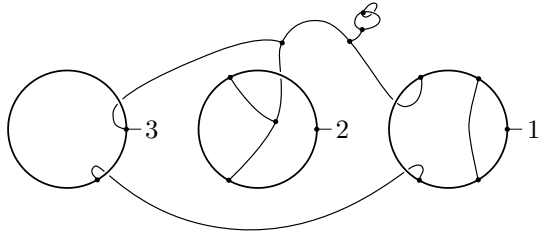}
  \caption{An example of a 3-admissible fat graph.}
  \label{figure:admissible_example}
\end{figure}

\begin{definition}
  Let  $\Gamma_1$ and $\Gamma_2$ be $p$-admissible fat graphs.
  We say $\Gamma_1 \sim_{\scriptscriptstyle{SD}} \Gamma_2$ if
  $\Gamma_1$ and $\Gamma_2$ are connected by a zigzag of edge collapses where we only collapse inner-edges that do not belong to the admissible cycles and are not loops.
  Equivalently, if $\Gamma_2$ can be obtained from $\Gamma_1$ by sliding vertices along edges that do not belong to the admissible cycles.
  Figure \ref{figure:sullivan_example} shows some examples of equivalent 1-admissible fat graphs.
\end{definition}
\begin{figure}[ht]
  \centering
  \inkpic[\columnwidth]{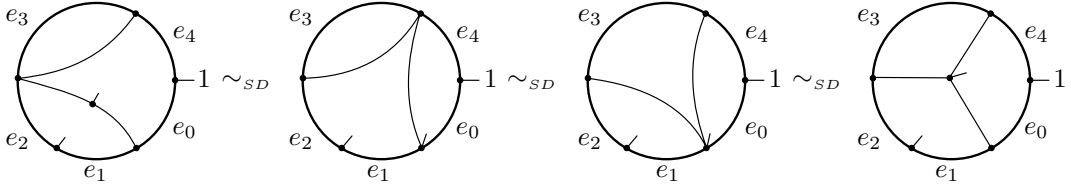}
  \caption{Four equivalent 1-admissible fat graphs.
  Since there is only one admissible cycle, 
  instead of labeling edges by $e^1_j$
  (as described in Definition \ref{definition:space_of_SD})
  they are just labeled by $e_j$.}
  \label{figure:sullivan_example}
\end{figure}

It follows directly that $\sim_{\scriptscriptstyle{SD}}$ is an equivalence relation.

\begin{definition}
  \label{definition:SD_eq_relation}
  A \emph{$p$-Sullivan diagram} $\Sigma$ is an equivalence class of $p$-admissible fat graphs under the equivalence relation $\sim_{\scriptscriptstyle{SD}}$.
\end{definition}

\begin{definition}
  \label{definition:space_of_SD}
  The \emph{multi-degree of a $p$-Sullivan diagram} $\Sigma$ is the tuple
  \[(\vert E_1 \vert -1,\vert E_2 \vert -1,\ldots, \vert E_p \vert -1)\]
  where $E_i$ is the set of inner-edges that belong to the $i$-th admissible cycle.
  The fat structure together with the admissible leaves give a natural ordering of the edges that belong to each admissible cycle.
  Let $e^i_0, e^i_1, \ldots, e^i_{\vert E_i\vert-1}$ denote the edges on the $i$-th admissible cycle in this induced order, see Figure \ref{figure:sullivan_example} for an example.
  For $1\leq i \leq p$, $0\leq j < |E_i|$ and $1 < |E_i|$, the \emph{faces} of a Sullivan diagram $\Sigma$ are given by
  \begin{equation}
    \label{faces}
    d^i_j(\Sigma):=\Sigma/e^i_j
  \end{equation}
  where $\Sigma/e^i_j$ is the Sullivan diagram obtained by collapsing the edge $e^i_j$.
  Note that $\Sigma/e^i_j$ is well defined since we are only collapsing inner-edges on the admissible cycles and if $|E_i|=1$ we do not collapse any edges on the $i$-th admissible cycle.
  The $d_j^i$'s fulfill the multi-semisimplicial identities.

  The \emph{space of p-Sullivan diagrams} is the multi-semisimplicial set denoted by $\ph\SDs$ with $(k_1,k_2,\ldots, k_p)$ multi-simplices given by (isomorphism classes) of Sullivan diagrams of that multi-degree and faces given as described in (\ref{faces}).
  In particular, the space of 1-Sullivan diagrams is a semi-simplicial set.
  The \emph{space of Sullivan diagrams $\SDs$} is
  \[
    \SDs:=\coprod_p \ph\SDs \,.
  \]
\end{definition}

\begin{remark}
  Notice that a $p$-Sullivan diagram $\Sigma$ can be fattened as in Remark \ref{remark:fattening_graph} to give a topological type of a surface with decorations which we will denote by $S_{\Sigma}$ and
  we call this \emph{the topological type} of $\Sigma$.
  The surface with decorations $S_\Sigma$ obtained by this fattening procedure must have either $p$ boundary components and at least one puncture or at least $p+1$ boundary components.
  Moreover, $p$ of its boundary components must have exactly one marked point which are enumerated by the set $\{1,2,\ldots, p\}$.
  
  Since the face maps collapse an inner-edge which is not a loop $\Sigma$ and $d^i_j(\Sigma)$ have the same topological type for any $i$ and $j$.
  Furthermore, and any two $p$-Sullivan diagrams of the same topological type are connected by a zigzag of face maps.
  Therefore, the space of $p$-Sullivan diagrams splits into connected components given by the topological type of their diagrams.
  We denote these components by by $\ph\SDs(S)$ i.e.,
  \[
    \ph\SDs=\coprod_S \ph\SDs(S)
  \]
  where the disjoint union is taken over all topological types of surfaces with decorations that can be obtained by the fattening procedure as indicated above.
\end{remark}

One can interpret the barycentric coordinates of a point in the space of a Sullivan diagrams as the lengths of the edges in the admissible cycles.
Then one can intuitively think of $\SDs$ as a space of metric admissible graphs were the topology is given by the lengths of the edges in the admissible cycles while all other edges are of length zero.
We use this to motivate the following definition.

\begin{definition}
  \label{definition:degenerate_diagram}
  We call a boundary cycle of a Sullivan diagram \emph{non-degenerate} if its boundary cycle sub-graph has at least one inner-edge which belongs to an admissible cycle and \emph{degenerate} otherwise.
  In other words a boundary cycle is non-degenerate if it has at least one edge of ``positive length".
  We call a Sullivan diagram \emph{non-degenerate} if all its boundary cycles are non-degenerate and \emph{degenerate} otherwise.
\end{definition}

\begin{remark}
  \label{remark:not_a_fibration}
  Let $\tilde{S}$ be an oriented surface with $p+m$ boundary components each of which has exactly one marked point enumerated by the set $\{1,2,\ldots, p+m\}$.
  Let $S$ be the oriented surface obtained by collapsing each of the last $m$ boundary components to a marked point.
  Then we have a central extension of the mapping class group $\Mod(S)$ by $\mathbb Z^m$
  \[\mathbb{Z}^m \cof \Mod(\tilde{S}) \fib \Mod(S)\]
  which induces a fibration sequence of spaces
  \[B\mathbb{Z}^m \to B\Mod(\tilde{S}) \fib B\Mod(S).\]
  Similarly, we have a forgetful map 
  \[\pi\colon\ph\SDs(\tilde{S})\to \ph\SDs(S)\]
  given by forgetting the parametrization of the last $m$ boundaries i.e.\ by forgetting the last $m$ leaves of a Sullivan diagram in $\ph\SDs(\tilde{S})$.
  Note that this is not a simplicial map.
  Moreover, due to the equivalence relation $\sim_{\scriptscriptstyle{SD}}$ the map induced on the level of spaces is not a fibration.
  Indeed, its fibers do not even have a constant homotopy type.
  In fact, one can endow $\ph\SDs(S)$ with the structure of a filtered space, in the sense of Quinn or Miller in \cite{Quinn, Miller},
  such that the homotopy type of the fibers of $\pi$ remain constant over each stratum but changes when moving from one stratum to another.

  To see this, we construct a filtration of $\ph\SDs(S)$
  \[\ph\SDs(S)=:X^m \supseteq X^{m-1} \supseteq \ldots \supseteq X^0\supseteq X^{-1}:=\emptyset\]
  where $X^i$ is the subspace of $\ph\SDs(S)$ consisting of Sullivan diagrams with at most $i$ degenerate boundary components.
  The \emph{$i$-th strata} are the path connected components of $X^{i}-X^{i-1}$.
  Each $i$-stratum consists of a path connected component of the subspace of $\ph\SDs(S)$ consisting of diagrams which have exactly $i$ degenerate boundary cycles.
  One can readily see that the fibers over each $i$-stratum are homeomorphic to $(S^1)^{m-i}$ which correspond to all the possible positions of the the leaves corresponding to the $(m-i)$ non-degenerate boundary cycles.
  See Figure \ref{not_fibration} for an example.

  For the same reason, the maps that forget the labels of the boundary components or punctures are not fibrations either.
  Nevertheless, in Subsection \ref{subsection:detecting_homology_recipe} we show that in the case $p=1$ these maps do behave like coverings on the level of homology.
\end{remark}

\begin{figure}[ht]
  \centering
  \inkpic[\columnwidth]{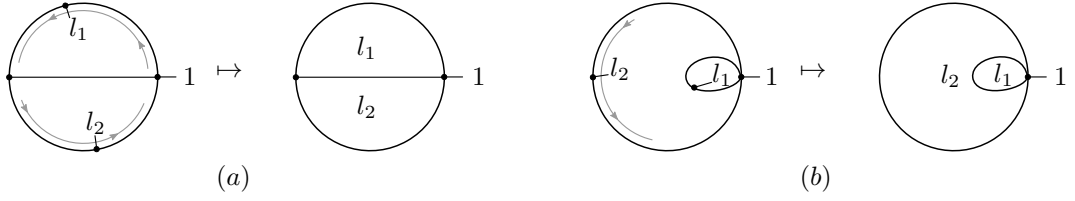}
  \caption{
    \label{not_fibration}
    The fibers of the map that forgets the leaves do not have a constant homotopy type.
    $(a)$ The fiber over a point in $X^0$ is a $S^1\times S^1$.
    $(b)$ The fiber over a point in $X^1-X^0$ is $S^1$.
  }
\end{figure}

We now define the chain complex of Sullivan diagrams.
Our definition is equivalent to the one given by Tradler and Zeinalian in \cite{TradlerZeinalian} and by Wahl and Westerland in \cite{wahlwesterland}.

\begin{definition}
  The \emph{chain complex of Sullivan diagrams $\SD$} is the chain complex freely generated as a $\Z$-module by all Sullivan diagrams.
  The total degree of a $p$-Sullivan diagram $\Sigma$ is 
  \[
    \mathrm{deg}(\Sigma):=\vert E_a \vert -p
  \]
  where $E_a$ is the set of edges that belong to the admissible cycles.
  The differential of a Sullivan diagram $\Sigma$ of multi-degree $(k_1, \ldots, k_p)$ is
  \[
    d(\Sigma):=\sum_{i=1}^p (-1)^{k_1 + \ldots + k_{i-1}} \sum_{j=0}^{k_i}
    (-1)^j d^i_j(\Sigma) \,.
  \]
  Figure \ref{figure:sullivan_differential} gives and example of the differential.

  The \emph{chain complex of $p$-Sullivan diagrams}, denoted $\ph\SD$, is the sub-complex of $\SD$ generated by all $p$-Sullivan diagrams.
  Note that the differential $d$ is the total differential of the multi-simplicial set $\ph\SDs$.
  Thus, $\ph\SD$ is the cellular complex of the multi-semisimplicial space $\ph\SDs$.
\end{definition}

\begin{figure}[ht]
  \centering
  \inkpic[.8\columnwidth]{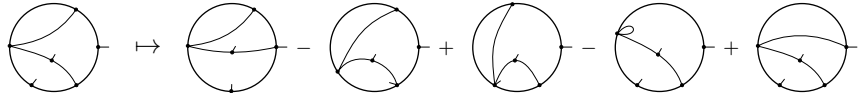}
  \caption{
    The differential of a $1$-Sullivan diagram of degree $4$.
    The labeling of the first leaf is as in Figure \ref{figure:sullivan_example}.
    Here, it is omitted for better readability.
  }
  \label{figure:sullivan_differential}
\end{figure}

\begin{remark}
  The complex of Sullivan diagrams splits into components
  \[
    \SD = \bigoplus_{p\geq 1} \ph\SD 
    \quad\quad\text{and}\quad\quad 
    \ph\SD = \bigoplus_S \ph\SD(S)
  \]
  where $\ph\SD(S)$ is the chain complex of $p$-Sullivan diagrams of topological type $S$.
  The second direct sum is indexed over all topological types $S$ that a $p$-admissible Sullivan diagram can have.
  More precisely, $S$ must have either $p$ boundary components and at least one puncture or at least $p+1$ boundary components.
  Moreover, $p$ of its boundary components have exactly one marked point and these are enumerated by the set $\{1,2,\ldots, p\}$.
\end{remark}

\begin{proposition}
  Let $S$ be a connected surface with at least one boundary component.
  If $S$ is not homeomorphic to the disk with one puncture then
  \[
    \chi(\ph\SDs(S)) = 0 \,.
  \]
  In particular, $\ph\SDs(S)$ is not contractible.
  Otherwise, if $S$ is homeomorphic to the disk with one puncture, then $p=1$ and
  \[\xh{1}\SDs(S)=\ast\, .\]
\end{proposition}
\begin{proof}
Let $S$ be a surface which is not homeomorphic to the disk with one puncture and let $\Sigma$ be a $p$-Sullivan diagram of topological type $S$.
Let $v_0$ be the vertex of the first admissible cycle of $\Sigma$ which is connected to its admissible leaf.
We call $\Sigma$ \emph{suspended} if the $v_0$ has valence $|v_0| = 3$ and \emph{unsuspended} if $|v_0|>3$.
Since $S$ is connected and not homeomorphic to the disk with one puncture, there is at least one $p$-Sullivan diagram of topological type $S$ which is unsuspended.
Consider the map
\[
  \Psi\colon\{\Sigma\in\ph\SD(S) \mid \Sigma \text{ is unsuspended}\}
  \to
  \{\Sigma\in\ph\SD(S) \mid \Sigma \text{ is suspended}\}
\]
where $\Psi(\Sigma)$ is the Sullivan diagram obtained by rotating the first admissible cycle of $\Sigma$ in clockwise direction while fixing its admissible leaf.
See Figure \ref{suspension_standard} for an example.
\begin{figure}[ht]
  \centering
  \inkpic[\columnwidth]{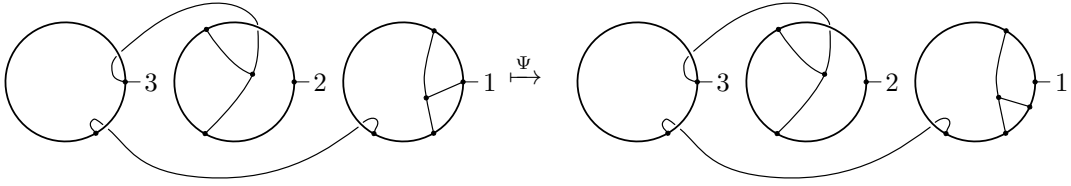}
  \caption{On the left an unsuspended Sullivan diagram and on the right its image under $\Psi$
    (which is obtained by rotating the first admissible cycle of $\Sigma$ in clockwise direction while fixing its admissible leaf).
  }
  \label{suspension_standard}
\end{figure}
Note that $\Psi$ is a bijection with inverse the face map $d_0^1$.
To finish the proof in this case, notice that $\mathrm{deg}(\Sigma) = \mathrm{deg}(\Phi(\Sigma))-1$.

On the other hand, if $S$ is homeomorphic to the disk with one puncture then 
$\ensuremath{1\hspace{1pt}\mhyphen}\SDs(S)$
has a single zero simplex and no higher simplices.
\end{proof}

\subsection{Alternative definitions}
In this paper, we have described the space of Sullivan diagrams as a quotient of the space of admissible fat graphs, although we have not explicitly described the topology of the latter.
However, there are other, equivalent definitions of this space, which are useful to have in mind.

In this subsection, we include a brief description of three such definitions.
It is important to note that not all constructions have the same conventions regarding incoming and outgoing boundary components.
Our convention coincide with the one of \cite{costellorg, wahlwesterland, wahluniversal}, but are opposite to the ones used by \cite{bodigheimer, kaufmann_sullivan}.
In our exposition, we have reversed the conventions of the latter to fit our setup.

\subsubsection{Spaces of arcs}
In \cite{kaufmann_sullivan}, Kaufmann describes a \emph{space of open-closed Sullivan diagrams} $\text{Sull}^{\text{c/o}}_p$ in terms of arcs embedded in a surface.
We briefly describe the closed part, $\text{Sull}^{\text{c}}_p$.
This space has components 
\[
  \text{Sull}^{\text{c}}_p=\bigsqcup_S \text{Sull}^{\text{c}}_p(S)
\]
where the disjoint union is taken over topological types of cobordisms $S$ with $p$ outgoing boundary components (all of which have exactly one marked point) and at least one incoming boundary component.
Each component is defined to be the quotient
\[
  \text{Sull}^{\text{c}}_p(S):=
  \betterwidetilde{\text{Sull}}^{\text{c}}_p(S)/\Mod(S)
\]
where $\betterwidetilde{\text{Sull}}^{\text{c}}_p(S)$ is a CW-complex obtained as the realization of a multi-semisimplicial set and $\Mod(S)$ is the mapping class group of $S$.
Each $(k_1, k_2, \ldots, k_p)$ multisimplex is given by a family $\alpha$ of ambient isotopy classes of arcs embedded in $S$ such that the arcs of $\alpha$ are all:
disjoint, non pairwise isotopic and flow from the outgoing boundary components to the incoming boundary components;
there are exactly $k_i+1$ arcs starting on the $i$-th outgoing boundary component for $1\leq i \leq p$.
Note that in particular, $\alpha$ has at least one arc intersecting each outgoing boundary component.
In other words $\betterwidetilde{\text{Sull}}^{\text{c}}_p(S)$ is a space of weighted isotopy classes of embedded arcs on the cobordism which flow from out to in, such that
there is at least one arc of positive weight on each outgoing boundary component.
The topology is given by the weights of the isotopy classes of arcs.
The mapping class group $\Mod(S)$ acts on the embedding and this action is cellular.
One can use a similar argument to Hatcher's surgery argument in \cite{hatcher_arc} to show that $\betterwidetilde{\text{Sull}}^{\text{c}}_p(S)$ is contractible.
However, the action of the mapping class group is not free.
The stabilizers are mapping class groups of surfaces of lower complexity.
One can give an explicit cellular isomorphism 
\[
  \text{Sull}^{\text{c}}_p(S) \stackrel{\cong}{\to} \ph\SDs(S).
\]
See \cite[Remark 2.15]{wahlwesterland} for a detailed description of this isomorphism.

\subsubsection{The harmonic compactification}
\label{subsection:harmonic_compactification}
Let $S$ denote an oriented cobordism of genus $g$ with $p \ge 1$ parametrized, enumerated outgoing boundary components and $m \ge 1 $  parametrized, enumerated incoming boundary components.
In \cite{bodigheimer}, B\"{o}digheimer constructs a space $\Rad_p(S)$ which is a model for $\cM(S)$ the moduli space of $S$.
The idea behind his construction is that any surface with a conformal structure of underlying topological type $S$ can be constructed from $p$ annuli in $p$ disjoint complex planes by a cut and glue procedure.
First, we cut slits into these annuli in a radial direction.
Then, the Riemann surface is obtained by gluing the border of these slits as shown in the example of Figure \ref{figure:radial_slit_configuration_glues_to_surface}.
\begin{figure}[ht]
  \centering
  \inkpic[.6\columnwidth]{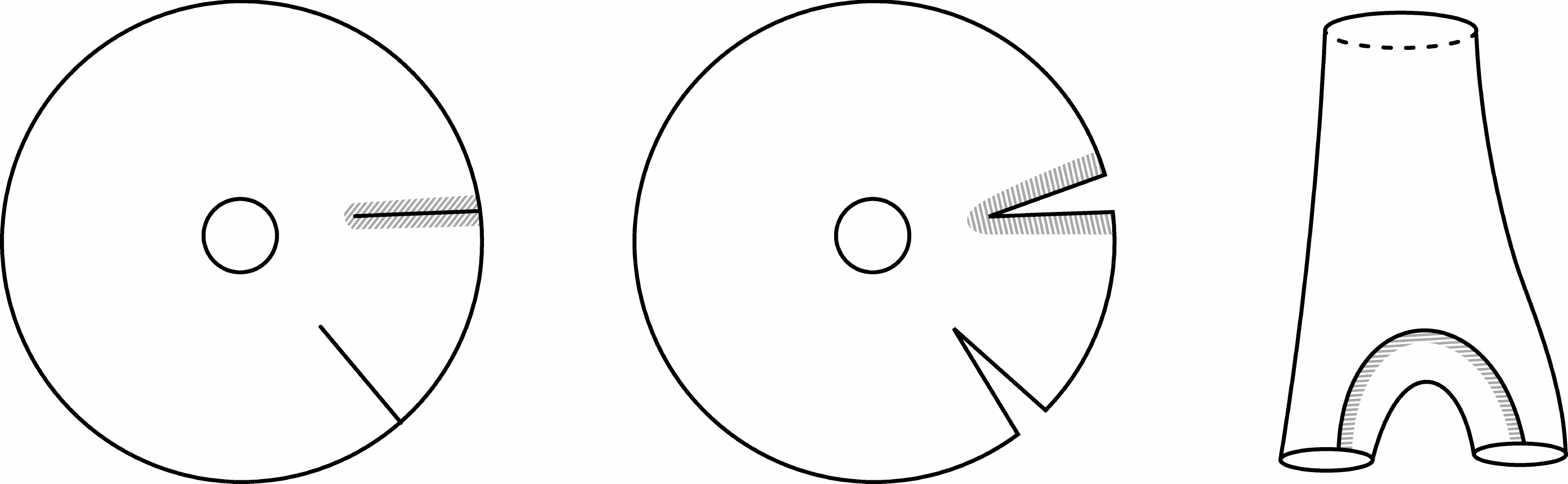}
  \caption{
    In this example, we begin with an annulus with two distinguished radial slits of the same length, starting at the outer boundary.
    After cutting these open, we glue the gray shaded border of one of the cut slits to the unshaded border of the second cut slit.
    This results in a surface with three boundary components.
    One of them corresponds to the inner radius of the annulus and the other two come from the cut and glue procedure.
    }
  \label{figure:radial_slit_configuration_glues_to_surface}
\end{figure}

The space $\Rad_p(S)$ is called \emph{the space of radial slit configurations of $S$}.
Its points are configurations of slits on the annuli together with gluing data,
subject to certain combinatorial conditions, such that the cut and glue procedure gives a surface of topological type $S$.

The main theorem of \cite{bodigheimer} states that there is a flat, affine bundle $\mathcal H_p(S) \xr{\simeq} \cM(S)$ together with a homeomorphism 
\[
  \mathcal H_p(S)\stackrel{\cong}{\longrightarrow} \Rad_p(S).
\]

Moreover, the space $\Rad_p(S)$ has a natural notion of compactification.
Namely, it is an open and dense subspace of a compact space $\bRad_p(S)$ which is the \emph{harmonic compactification of the moduli space of surfaces}.
Geometrically, this compactification allows handles to degenerate to intervals and finitely many points to become identified on the boundary.
The harmonic compactification has a subspace $\uRad_p(S)$  consisting of annuli of the same inner and outer radius in which all slits have the same length.
There is a strong deformation retraction of $\bRad_p(S)$ onto this subspace.
Furthermore, $\uRad_p(S)$ has a natural cell structure and admits a cellular homeomorphism to $\ph\SDs(S)$.
See \cite{egaskupers} for a detailed construction of this homeomorphism.
We summarize this in the following diagram: 
\begin{align*}
  \begin{tikzpicture}[baseline={([yshift=-.5ex]current bounding box.center)}] % This sets center of the tikzpicture to the middle of the baseline. This implies that the label of the equation is set correctly.
    \node (a) {$\cM(S)$};
    \node [right=1.5cm of a]  (b) {$\Rad_p(S)$};
    \node [right=1.5cm of b]  (c) {$\bRad_p(S)$};
    \node [right=1.5cm of c]  (d) {$\uRad_p(S)$};
    \node [right=1.5cm of d]  (e) {$\ph\SDs(S)$};
    \path[->]
      (b) edge node [above] {$\simeq$} (a)
      (d) edge node [above] {$\cong$} (e)
      (c) edge [bend left] node [above] {$\simeq$} (d);
    \path [right hook->]
      (b) edge (c);
    \path [left hook->]
      (d) edge node [below] {$\simeq$}(c);
  \end{tikzpicture}
\end{align*}
where $\simeq$ denotes homotopy equivalences, $\cong$ cellular {homeomorphisms} and hooked arrows denote inclusions.
Therefore, we can think of $\ph\SDs$ as a compactification of the moduli space and $\ph\SD$ as the chain complex that computes the homology of this compactification.

\subsubsection{BW graphs and the PROP-structure}
\label{subsection:BW_graphs}
  Let $S$ be the topological type of a cobordism with $p \ge 1$ parametrized, enumerated outgoing boundary components and $m \ge 1$ parametrized, enumerated incoming boundary components.
  In \cite{costellorg}, Costello constructs a chain complex generated by fat graphs with additional structure which computes the homology of the moduli space.
  Following the notation of \cite{wahlwesterland}, we call this complex the chain complex of black and white graphs and denote it  $\ph BW_{\mathrm{graphs}}(S)$.
  
  The chain complex of black an white graphs can be endowed with a notion of composition which models gluing surfaces along their boundary components \cite{wahlwesterland, egas}.
  In order words, black and white graphs assemble into a dg-PROP which models the 2d-cobordism category.
  That is, it assembles into a symmetric monoidal dg-category with objects the natural numbers and morphism spaces from $m$ to $p$ the chain complex of the disjoint union of the moduli spaces of all cobordisms with $m$ incoming boundary components and $p$ outgoing boundary components.
  
  Additionally, Wahl and Westerland show that the chain complex of Sullivan diagrams is a quotient of the chain complex of black and white graphs.
  That is, they give a quotient map
  \[
    \ph BW_{\mathrm{graphs}}(S)\fib \ph\SD(S).
  \]
  See \cite[Theorem 2.9]{wahlwesterland}.
  Furthermore, Sullivan diagrams inherit a composition structure from black and white graphs and also assemble into a dg-PROP.
  Here we briefly describe the PROP composition of Sullivan diagrams along closed boundary components.

\begin{definition}
  \label{def:composition}
  Let $S$ and $S'$ be cobordisms  with $n$ respectively $p$ incoming boundary components and $p$ respectively $q$ outgoing boundary components.
  The cobordisms $S$ and $S'$ are allowed to have punctures and their boundary components are parametrized (this is equivalent to having exactly one marked point on each boundary component) and
  enumerated by the sets $\{1, \ldots, n\}$, $\{1, \ldots, p\}$ and $\{1, \ldots, q\}$.
  Using the parametrizations and labels we glue the $i$-th outgoing boundary component of $S$ to the $i$-th incoming component of $S'$ and
  obtain a cobordism with $n$ incoming boundary components and $q$ outgoing boundary components which we denote $S'\circ S$.
  Then, by considering the admissible boundary cycles as outgoing boundary cycles and all other boundary cycles with a leaf as incoming boundary cycles, we get a degree zero chain map
  \[
    \circ \colon q\hspace{1pt}\mhyphen\SD(S') \otimes \ph \SD(S)  
    \longrightarrow q\hspace{1pt}\mhyphen \SD(S'\circ S), \quad
    \Sigma' \otimes \Sigma \mapsto \Sigma '\circ \Sigma
  \]
  where $\Sigma '\circ \Sigma$ is a sum of Sullivan diagrams obtained by the following procedure:
  \begin{enumerate}
    \item Choose fat graph representatives of $\Sigma$ and $\Sigma '$ and call them $\Gamma$ and $\Gamma'$.
    \item Cut open the $p$ admissible cycles of $\Gamma$ at the admissible leaves (which correspond to the marked points on the outgoing boundaries of $S$).
      This gives $p$ ``admissible intervals" enumerated by $\{1, \ldots, p\}$ to which a fat graph is attached.
    \item For each $1\leq i \leq p$ glue the $i$-th interval of the cut open graph along the $i$-th incoming boundary component of $\Gamma'$ such that the endpoints of the interval coincide with the $i$-th leaf of $\Gamma'$.
      Delete the $i$-th incoming leaf of $\Gamma'$.
    \item Note that there are several ways of gluing the intervals to their corresponding boundary cycles.
    Set $\Gamma'\circ\Gamma$ to be the sum over all such choices where the degree of the glued graph is $\deg(\Sigma)+\deg(\Sigma ')$.
    Let $\Sigma'\circ \Sigma$ be the sum of their corresponding Sullivan diagrams.
  \end{enumerate}
\end{definition}
Wahl and Westerland show that this is indeed a chain map and it is associative.
Therefore, we can consider a dg-category with objects natural numbers and morphisms from $n$ to $p$ to be  $\bigoplus_{S}\ph\SD(S)$, where the sum is taken over all cobordisms $S$ from $n$ to $p$ boundary components.
The composition map is the one given above.
The sum of natural numbers and disjoint union of cobordisms makes this into a symmetric monoidal dg-category.
That is, Sullivan diagrams assemble into a dg-PROP.

\subsection{Asymmetry of Sullivan diagrams}
\label{subsec:asymmetry}
The homotopy type of the space of $p$-Sullivan diagrams depends on the distribution of the (parametrized, enumerated) boundaries into incoming and outgoing.
We make this precise.
For $1\leq i\leq p$, let $S(i,g, m+p-i)$ denote a cobordism of genus $g$ with
$(m+p-i)$ parametrized, enumerated incoming boundary components and 
$i$ parametrized, enumerated outgoing boundary components.
Let $S_{g,p+m}$ denote its underlying surface.
The spaces $\Rad_i(S(i,g, m+p-i))$ and $\Rad_j(S(j,g, m+p-j))$ for $i\neq j$ are on the nose different.
However, they are homotopy equivalent since they are both models for the moduli space $\cM(S_{g, p+m})$ of $S_{g, p+m}$ which is independent of the parameters $i$ and $j$ i.e.,
\[
  \Rad_i(S(i,g, m+p-i)) \simeq \cM(S_{g,p+m}) \simeq \Rad_j(S(j,g, m+p-j)).
\]
This is no longer true for the space of Sullivan diagrams or equivalently the harmonic compactification.
Namely, $\xh{i}\SDs(S(i,g, m+p-i))$ and $\xh{j}\SDs(S(j,g, m+p-j))$  for $i\neq j$ are in general not homotopy equivalent.
They are not even homologically equivalent.
Indeed, in \cite{kaufmann_sullivan, wahlwesterland}, Kaufmann and independently Wahl and Westerland show that $\ph\SD(S(p,0,1))$ computes the homology of $\cM(S_{0, p+1})$.
The homology of $\cM(S_{0, p+1})$ has been computed by Salvatore and Wahl in \cite{Salvatore_Wahl}.
Contrasting their result with Theorem \ref{theorem_a} in Section \ref{section:vanishing} shows that $\ph\SD(S(p,0,1)$ and 
$\xh{1}\SD(S(1,0,p)$ are not quasi-isomorphic, i.e.\
\[
  H_*(\ph\SD(S(p,0,1)) \cong H_*(\cM(S_{0, p+1})) \ncong H_*(\xh{1}\SD(S(1,0,p)).
\]
Therefore the homotopy type of Sullivan diagrams is not symmetric with respect to which boundary components we consider as incoming and which we consider outgoing.

\section{Sullivan diagrams on a single circle}
\label{section:sd_one_circle}
In this paper we study the homotopy type of the following four families of spaces of 1-Sullivan diagrams.
\begin{definition}
  \label{sd_components}
  Let ${S_{\Sigma}}$ be the topological type of the surface with decorations obtained by thickening a Sullivan diagram $\Sigma$ (see Definition \ref{definition:fattening}).
  Let $g(\Sigma)$, $\partial(\Sigma)$ and $m(\Sigma)$ denote the genus, number of boundary components and number of punctures of ${S_{\Sigma}}$ respectively.
  We define the following spaces.
  \begin{description}
    \item[Unparametrized unenumerated diagrams]
      \[
        \xh{1}\SDs\supset\SDs_{g}^m:= \left\lbrace
          \text{cells }\Sigma \ \left|  \
          \begin{array}{l}
            \text{the only leaf of } \Sigma \text{ is the admissible leaf;}\\
             g(\Sigma)=g, \partial(\Sigma)=1, m(\Sigma)=m.
          \end{array}
          \right.
        \right\rbrace
      \]

    \item[Parametrized unenumerated diagrams]
      \[
        \xh{1}\SDs\supset\SDs_{g,m}:= \left\lbrace
          \text{cells }\Sigma \ \left|  \
          \begin{array}{l}
            \Sigma \text{ has exactly one leaf per boundary cycle;}\\
            \text{all except the admissible leave are unenumerated;}\\
            g(\Sigma)=g, \partial(\Sigma)=m+1, m(\Sigma)=0.
          \end{array}
          \right.
        \right\rbrace
      \]

    \item[Parametrized enumerated diagrams]
      \[
        \xh{1}\SDs\supset\betterwidetilde{\SDs}_{g,m}:=\left\lbrace
          \text{cells }\Sigma \ \left|  \
          \begin{array}{l}
            \Sigma \text{ has exactly one leaf per boundary cycle;}\\
            \text{all leaves are enumerated;}\\
            g(\Sigma)=g, \partial(\Sigma)=m+1, m(\Sigma)=0.
          \end{array}
          \right.
        \right\rbrace
      \]

    \item[Unparametrized enumerated diagrams]
      \[
        \betterwidetilde{\SDs}_{g}^m:=\left\lbrace
          \text{cells }\Sigma \ \left|  \
          \begin{array}{l}
            \text{the only leaf of } \Sigma \text{ is the admissible leaf;}\\
            g(\Sigma)=g, \partial(\Sigma)=1, m(\Sigma)=m;\\
            \text{all boundary cycles are enumerated.}
          \end{array}
          \right.
        \right\rbrace
      \]
  \end{description}
  The first three types of spaces of Sullivan diagrams are connected components of $\SDs$.
  However, the space $\betterwidetilde{\SDs}_{g}^m$ is not a subspace of $\SDs$ as we have defined it, since we only enumerate boundary cycles which are connected to a leaf and use the leaf to do so.
  Nevertheless, one could extended the definition of a fat graph to include this case by including additional data.
  Namely enumerating the boundary cycles of the fat graph and therefore of the Sullivan diagram.
  The topological type of such Sullivan diagrams would give surfaces with enumerated punctures.
  Our results extended naturally to this context.
\end{definition}

From now onward, we will remove $p$ from the notation and we will refer only to ``the space of Sullivan diagrams" or a ``Sullivan diagram" and it will be understood that $p=1$ unless stated otherwise.
By abuse of notation we write $\SDs$ for the space of 1-Sullivan diagrams.
Moreover, we exclude the enumeration of the admissible leaf in our drawings as it is the unique leaf on the admissible circle that ``points outwards".

In Section \ref{section:sd_one_circle:combinatorial_1_SD} we establish a combinatorial definition of Sullivan diagrams for the components described above.
We treat the face maps in this presentation in Section \ref{section:sd_one_circle:face_map}.
Finally, the stabilization map is discussed in Section \ref{section:sd_one_circle:stabilization_map}.

\subsection{Combinatorial 1-Sullivan diagrams}
\label{section:sd_one_circle:combinatorial_1_SD}

Recall that the cells of the space of Sullivan diagrams are given by equivalence classes of fat graphs as described in Section \ref{section:p_sullivan_diags:definitions}.
In this section we give unique representatives for the cells of all of the spaces above, except for the parametrized unenumerated case.
In this special case, we still represent the cells using an equivalence relation.
However, it is a much simpler equivalence relation then the one given in terms of graphs.

In order to do this we give a combinatorial definition of Sullivan diagrams.
The informal idea is that any Sullivan diagram is uniquely described by attaching onto a ``ground circle" topological types of surfaces with decorations.
These surfaces are determined by their genus, number of punctures, boundary components and the combinatorial data of marked points at the boundary.

Inspired by \cite{bodigheimer, abhau}, we describe the combinatorial data of the marked points at the boundary and the way in which these surfaces are attached to the ground circle by permutations.
Thereafter, we encode the genus and number of punctures as weights.
Therefore, we give a presentation of Sullivan diagrams in terms of bi-weighted permutations subject to certain conditions, which resemble the concept of stable graphs of \cite{kon92} and \cite{Loo95}.

For better readability we give separately the definition of combinatorial diagrams for each of the four cases in Definitions 
\ref{definition:combinatorial_1_SD_unpar_unen},
\ref{definition:combinatorial_1_SD_unpar_en},
\ref{definition:combinatorial_1_SD_par_en} and
\ref{definition:combinatorial_1_SD_par_unen}.
However, since they all have similar features, in the end we denote them in the same way.
See Notation \ref{notation:combianatorial_SD}.
We recommend the reader to study the unparametrized unenumerated case first.
This is the base case since most of the geometric intuition follows from it.

The other three cases are obtained by adding decorations to this basic case in the form of enumeration data or leaves.
However, to ease future reference and for compactness, we collect all the definitions of combinatorial diagrams together at the start of this subsection.
Therefore, in a first reading we recommend to start with Definition 
\ref{definition:combinatorial_1_SD_unpar_unen}, 
which describes the base case and then jump to the statement of Proposition \ref{proposition:combinatorial_1_SD}.
Continue with Definition \ref{definition:essential_trivalent} and read until Remark \ref{remark:non-degenerate_boundary}.  
In between these two the idea of the proof is presented, which gives the geometric intuition.
Thereafter go back to study the other three remaining cases and the proof of the proposition.

At the end of this subsection we make two remarks.
Remark \ref{remark:compare_combinatorial_1_SD_with_Rad} where we expose the tight relation between the combinatorics of B\"{o}digheimer's model for the moduli spaces and our description of 1-Sullivan diagrams.
Remark \ref{remark:top_degree_of_a_SD} relates the top degree of the spaces of Sullivan diagrams of topological type $S$ with the Euler characteristic of $S$.

\medskip

We start with the combinatorial description for the unparametrized cases.

\begin{definition}
  \label{definition:combinatorial_1_SD_unpar_unen}
  Let $n \ge 0$ and  $k \ge 1$.
  An \emph{unparametrized unenumerated combinatorial $1$-Sullivan diagram} $\Sigma$ of degree $n$ is a pair $(\lambda, \{ S_1, \ldots, S_k \})$ consisting of:
  \begin{enumerate}
    \item the \emph{fat structure} $\lambda\in \Symm([n])$ where $[n] = \{0, \ldots, n\}$,
    \item the set of cycles of $\lambda$ is denoted by $\Lambda$,
    \item the \emph{non-degenerate boundary} given by the permutation
      $\rho = \lambda^{-1}(0\ 1\ \ldots\ n),$
    \item the \emph{ghost surfaces} given by triples $S_i = (g_i, m_i, A_i)$ where $A_i\subset
      \Lambda$ and $g_i, m_i\in \N_{\geq 0}$.
  \end{enumerate}
  Subject to the following conditions:
  \begin{enumerate}[label={(\roman*)}]
    \item \label{data_partition_unpar_unen}      The $A_i$'s form a partition of $\Lambda$, i.e., $A_i\cap A_j = \emptyset$ if $i\neq j$ and $\cup_i A_i = \Lambda$.
    \item \label{data_suspension_discunpar_unen} If $S_i=(0,0,A_i=\{(r)\})$, that is $A_i$ is a set with one cycle on one element, then $r=0$.
                                                 In this case, we call $S_i$ the \emph{suspension disk}.
                                            
  \end{enumerate}
  We denote by $C(n)$ the set of unparametrized unenumerated combinatorial 1-Sullivan diagrams of degree $n$.
\end{definition}

 The unparametrized enumerated case is obtained from the above by adding data corresponding to the enumeration of the non-degenerate boundary components and punctures.
 
 \begin{definition}
  \label{definition:combinatorial_1_SD_unpar_en}
  Let $n \ge 0$ and $k \ge 1$.
  An \emph{unparametrized enumerated combinatorial $1$-Sullivan diagram} $\Sigma$ of degree $n$ is a quadruple $(\lambda, \{ S_1, \ldots, S_k \}, \beta_1, \beta_2)$ consisting of:
  \begin{enumerate}
    \item the \emph{fat structure} $\lambda\in \Symm([n])$ where $[n] = \{0, \ldots, n\}$,
    \item the set of cycles of $\lambda$ is denoted by $\Lambda$,
    \item the \emph{non-degenerate boundary} given by the permutation
      $\rho = \lambda^{-1}(0\ 1\ \ldots\ n),$
    \item the \emph{ghost surfaces} given by triples $S_i = (g_i, m_i, A_i)$ where $A_i\subset
      \Lambda$ and $g_i, m_i\in \N_{\geq 0}$,
    \item the \emph{enumerating data} given by injections
      \[
        \beta_1 \colon \{\text{cycles of } \rho\} \cof \{1,2,\ldots, m\}
      \]
      \[
        \beta_2 \colon \{1,2,\ldots, m\}-\mathbf{Im}(\beta_1)\cof \{S_1, S_2,\ldots, S_k\}.
      \]
  \end{enumerate}
  Subject to the following conditions:
  \begin{enumerate}[label={(\roman*)}]
    \item \label{data_partition_unpar_en}       The $A_i$'s form a partition of $\Lambda$, i.e., $A_i\cap A_j = \emptyset$ if $i\neq j$ and $\cup_i A_i = \Lambda$.
    \item \label{data_suspension_disc_unpar_en}   If $S_i=(0,0,A_i=\{(r)\})$, that is $A_i$ is a set with one cycle on one element, then $r=0$.
      In this case, we call $S_i$ the \emph{suspension disk}.
    \item[(v)] \label{data_enumerattion_unpar_en} For every $1\leq i\leq k$ it holds that $|\beta_2^{-1}(i)|=m_i$.
  \end{enumerate}
  We denote by $\betterwidetilde{C}(n)$ the set of unparametrized enumerated combinatorial 1-Sullivan diagrams of degree $n$.
\end{definition}
 
We extend this further to the parametrized case by adding a set of leaves to this combinatorial data.
 
\begin{definition}
  \label{definition:combinatorial_1_SD_par_en}
  Let $n \ge 0, k \ge 1$ and let $L = \{ l_1, \ldots, l_m \}$ be a non-empty finite set.
  A \emph{parametrized enumerated combinatorial $1$-Sullivan diagram} $\Sigma$ of degree $n$ with leaves $L$ is a pair $(\lambda, \{ S_1, \ldots, S_k \})$ consisting of the following data:
  \begin{enumerate}
    \item the \emph{fat structure} $\lambda\in \Symm([n]\sqcup L)$ where $[n] = \{0, \ldots, n\}$,
    \item the set of cycles of $\lambda$ is denoted by $\Lambda$,
    \item the \emph{non-degenerate boundary} given by the permutation
      $\rho = \lambda^{-1}(0\ 1\ \ldots\ n),$
    \item the \emph{ghost surfaces} given by triples $S_i = (g_i, 0, A_i)$ where $A_i\subset \Lambda$ and $g_i\in \N_{\geq 0}$.
  \end{enumerate}
  Subject to the following conditions:
  \begin{enumerate}[label={(\roman*)}]
    \item \label{data_partition_par_en}            The $A_i$'s form a partition of $\Lambda$, i.e., $A_i\cap A_j = \emptyset$ if $i\neq j$ and $\cup_i A_i = \Lambda$.
    \item \label{data_suspension_disc_par_en}      If $S_i=(0,0,A_i=\{(r)\})$, that is $A_i$ is a set with one cycle on one element, then $r=0$.
      In this case, we call $S_i$ the \emph{suspension disk}.
    \item \label{data_attach_all_surfaces_par_en}  For each $i$ the set $A_i$ contains at least one cycle with an element in $[n]$.
    \item \label{data_conditions_leaves_par_en}    Each cycle of $\rho$ permutes exactly one leaf non-trivially or it is a fixed point $(l_i)$ with $l_i \in L$.
  \end{enumerate}
    We denote by  $\betterwidetilde{C}(n,L)$, the set of parametrized enumerated combinatorial 1-Sullivan diagrams of degree $n$ with leaves $L$.
\end{definition}

We now define an equivalence relation on this data and use it to define the parametrized unenumerated case.

\begin{definition}
  \label{definition:combinatorial_1_SD_par_unen}
  Let $n \ge 0$ and $L = \{ l_1, \ldots, l_m \}$ be a finite, non-empty set.
  Given permutations $\lambda\in \Symm([n]\sqcup L)$ and $\sigma \in \Symm(L)$, we view $\sigma \in \Symm([n] \cup L)$ by extending $\sigma(i) = i$ for all $i \in [n]$.
  The conjugation of $\lambda$ by $\sigma$ is $c_\sigma(\lambda) = \sigma^{-1}\lambda\sigma$.
  Consequently, $\Symm(L)$ acts from right on the set of parametrized enumerated combinatorial 1-Sullivan diagrams of degree $n$ with labels in $L$ by
  \[
    (\lambda, S_1, \ldots, S_k). \sigma = (c_\sigma(\lambda), c_\sigma(S_1), \ldots, c_\sigma(S_k))
  \]
  where for $S_i = ( g_i, m_i, \{ s_{i,1}, \ldots, s_{i,r_i} \})$ we have $c_\sigma(S_i) = (  g_i, m_i, \{ c_\sigma(s_{i,1}), \ldots, c_\sigma(s_{i,r_i}) \} )$.
  
  This action defines an equivalence relation on the set of parametrized enumerated combinatorial 1-Sullivan diagrams and
  a \emph{parametrized unenumerated combinatorial 1-Sullivan diagram} is an equivalence class.
  Informally such a diagram is given by ``forgetting" the enumeration of the leaves i.e.\ the elements of the set $L$.
  We denote by 
  \[C(n,L) := \betterwidetilde{C}(n,L)/_{\Symm(L)}\]
  the set of parametrized unenumerated combinatorial 1-Sullivan diagrams of degree $n$ with $m$ leaves.
\end{definition}

\begin{remark}
  Note that there is also an action of $\Symm({1,2,\ldots, m})$ on the set of  unparametrized enumerated combinatorial 1-Sullivan diagrams of degree $n$ given by acting by conjugation on the enumerating data.
  This is similar to the action described in Definition \ref{definition:combinatorial_1_SD_par_unen}.
  This action gives an equivalence relation on the set of unparametrized enumerated combinatorial 1-Sullivan diagrams.
  One can readily see that an unparametrized unenumerated combinatorial 1-Sullivan diagram as we have defined it is an equivalence class under this relation.
\end{remark}  
  
We collect the four versions of combinatorial Sullivan diagrams into one.
  
\begin{notation}
\label{notation:combianatorial_SD}
  We refer by a \emph{combinatorial 1-Sullivan diagram} to any of the four different versions of given in Definitions \ref{definition:combinatorial_1_SD_unpar_unen},
  \ref{definition:combinatorial_1_SD_unpar_en},
  \ref{definition:combinatorial_1_SD_par_en} and 
  \ref{definition:combinatorial_1_SD_par_unen}.
  Such a diagram is determined by the data $(1)-(5)$ and conditions $(i)-(v)$ of their respective definitions.
  In different versions some of the data entries or conditions are empty.
  We denote any combinatorial 1-Sullivan by a tuple $\Sigma = (\lambda, S_1, \ldots, S_k)$ and omit writing the enumerating data unless it is specifically necessary.

  Consider two combinatorial 1-Sullivan diagrams $(\lambda, S_1, \ldots, S_k)$ and $(\lambda', S'_1, \ldots, S'_{k'})$ of degree $n$.
  In the unparametrized unenumerated case and in the parametrized enumerated case,
  the two Sullivan diagrams are equal if and only if $\lambda = \lambda'$, $k=k'$ and the ghost surfaces agree up to reordering.
  In the unparametrized enumerated case they must also have the same enumerating data.
  In the parametrized unenumerated case we require that the orbits of the Sullivan diagrams under the action of $\Symm(L)$ agree.
  By convention we order the ghost surfaces in the order in which they are attached to the ground circle.
\end{notation}

We now show that combinatorial Sullivan diagrams indeed index the cells of their corresponding spaces of Sullivan diagrams.

\begin{proposition}
  \label{proposition:combinatorial_1_SD}
  Given $n \ge 0$ and $m\geq 1$, let $L = \{ l_1, \ldots, l_m\}$.
  There exist bijections.
  \begin{align*}
    C(n) & \xlr{1:1} \{ \text{$n$-cells of } \coprod_{g,r} \SDs_g^r \}\\
    \betterwidetilde{C}(n) & \xlr{1:1} \{ \text{$n$-cells of } \coprod_{g,r} \betterwidetilde\SDs_{g}^r \}\\
    \betterwidetilde{C}(n,L) & \xlr{1:1} \{ \text{$n$-cells of } \coprod_{g} \betterwidetilde{\SDs}_{g,m} \}\\    
    C(n,L) := \betterwidetilde{C}(n,L)/_{\Symm(L)} & \xlr{1:1} \{ \text{$n$-cells of } \coprod_{g} \SDs_{g,m} \}
  \end{align*}
\end{proposition}

In order to prove Proposition \ref{proposition:combinatorial_1_SD}
we need a particular graph-representative of a Sullivan diagram, which we can think of as a representative in ``general form".

\begin{definition}
  \label{definition:essential_trivalent}
  Let $\Gamma$ be a 1-admissible fat graph, and let $\Gamma_{in}\subset \Gamma$ be the sub-graph of $\Gamma$ obtained by deleting the edges of the admissible cycle and the admissible leaf.
  We say $\Gamma$ is \emph{maximally contracted away from the boundary} if each connected component of $\Gamma_{in}$ has at most one vertex which is not on the admissible cycle of $\Gamma$.
  Furthermore, we call an admissible fat graph \emph{essentially trivalent at the boundary} if all the vertices on the admissible cycle have valence three,
  except possibly the vertex that is connected to the admissible leaf which can have valence four.
\end{definition}

\begin{remark}
  \label{remark:essential_trivalent}
  Any Sullivan diagram has a representative which is essentially trivalent at the boundary and maximally contracted away from the boundary.
  This is obtained by first sliding higher valence vertices away from the admissible cycle and then contracting each connected component of $\Gamma_{in}$ onto a graph with at most one vertex not on the admissible cycle.
\end{remark}

The general idea of the proof of Proposition \ref{proposition:combinatorial_1_SD} is that any Sullivan diagram of degree $n$ is given by attaching onto a ``ground circle" topological types of surfaces $S_1, S_2, \ldots, S_k$.
The data above determines the topological type of the $S_i$'s and how they are attached to the ground circle.

More precisely, for each $S_i$ the natural numbers $g_i$ and $m_i$ are the genus and the number of punctures of the surface.
The cycles of $\lambda$, that is the set $\Lambda$, represent the (set of) boundary components of all the surfaces $S_i$.
The decorations of each boundary component are encoded in the cycles of $\lambda$.
The partition of $\Lambda$ into the subsets $A_i$ determine which boundary components belong to the surface $S_i$.
Finally, the permutation $\lambda$ also determines how the ghost surfaces are attached to the ground circle.
See Figures \ref{figure:parametrized_and_enumerated} and \ref{figure:unparametrized} for examples.
This idea is clarified further by the following remark.

\begin{remark}
  \label{remark:non-degenerate_boundary}
  The names \emph{fat structure} and \emph{non-degenerate boundary} come from the geometric interpretation of this description.
  The cycles of the fat structure $\lambda$ describe the decorations of the boundary of the ghost surfaces and the cyclic ordering in which they occur.
  The cycles of $\rho$ determine the number of boundary cycles of $\Sigma$ which have at least one edge on the admissible cycle, to which we refer to as non-degenerate (see Definition \ref{definition:degenerate_diagram}).
  
  More precisely, in the unparametrized case
  the cycles of $\rho$ are in bijection with the cycles of $\Sigma$ which are not the admissible cycle and have at least one edge on the admissible cycle.
  Thus, a Sullivan diagram $\Sigma\in\betterwidetilde\SDs_{g}^m$ or $\SDs_{g}^m$ is non-degenerate if and only if $\rho$ has $m$ cycles.
  See Figure \ref{figure:unparametrized}.
  
  In the parametrized case, the cycles of $\rho$ are in bijection with the cycles of $\Sigma$ which are not the admissible cycle and
  which have either at least one edge on the admissible cycle or consist of a single leaf $\rho(l_i) = l_i$.
  Thus, a Sullivan diagram $\Sigma\in\betterwidetilde\SDs_{g,m}$ or $\SDs_{g,m}$ is non-degenerate if $\rho$ has $m$ cycles none of which consist of a single leave.
  See Figure \ref{figure:parametrized_and_enumerated}.

  In the case where all the ghost surfaces attached are disks with one boundary component, the reason for these names becomes evident.
  Indeed, in this case there is a unique fat graph representative of $\Sigma$ which is essentially trivalent at the boundary and maximally contracted.
  This is a fat graph where we attach corollas to the ground circle.
  Call this fat graph $\Gamma_{\Sigma}$.
  Then $\lambda$ is the fat structure at the vertices of $\Gamma_{\Sigma}$ which are not in the admissible cycle and
  the cycles of $\rho$ are the cycles of $\Gamma_{\Sigma}$ which are not the admissible cycle, all of which are non-degenerate.

  Regarding the enumerating data, in the unparametrized enumerated case i.e., $\betterwidetilde{\SDs}_{g}^m$, the maps $\beta_1$ and $\beta_2$ enumerate the non-degenerate boundary cycles and punctures respectively.
  In the case of $\betterwidetilde{\SDs}_{g,m}$, the role of $\beta_1$ and $\beta_2$, enumerating the non-degenerate and degenerate boundary, is implicit in $\lambda$ and the $A_i$'s.
\end{remark}

\begin{figure}[ht]
    \centering
    \inkpic[.4\columnwidth]{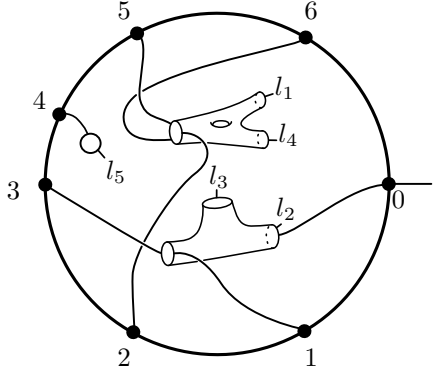}
    \caption{
      \label{figure:parametrized_and_enumerated}
      We interpret a combinatorial 1-Sullivan diagram $\Sigma = (\lambda, S_1, S_2, S_3)$ as a cell of $\widetilde{\SDs}_{g,m}$.
      The fat structure is $\lambda = (0\ l_2)(1\ 3)(4\ l_5)(2\ 6\ 5)$ with fixed points $l_1, l_3, l_4$ and $\rho = (0\ 3\ l_5\ 4\ 6\ l_2)(1\ 5\ 2)$ with the same fixed points.
      The ghost surfaces are $S_1 = (0,0,\{(0\ l_2),(l_3), (1\ 3)\})$, $S_2=(0,0,\{(4), (l_5)\})$ and
      $S_3= (1,0, \{(2\ 6 \ 5), (l_1), (l_4)\})$.
      Consequently, we consider a surface of genus $0$, no punctures and three boundary components for $S_1$,
      a surface of genus $1$, no punctures and three boundary components for $S_2$ and
      a surface of genus $0$, no punctures and one boundary component for $S_3$.
      The distribution of the leaves and the gluing of the surfaces to the ground cycle is prescribed by the fat structure $\lambda$.
    }
\end{figure}

\begin{proof}[Proof of Proposition \ref{proposition:combinatorial_1_SD}]
  We construct three pairs of maps that are mutually inverse to each other.
  The unparametrized enumerated case follows similarly and we leave it to the reader.
  
  Let $\Sigma\in\betterwidetilde{\SDs}_{g,m}$ and let $\Gamma_{\Sigma}$ denote a fat graph representative of $\Sigma$ which is essentially trivalent at the boundary (see Remark \ref{remark:essential_trivalent}).
  Note that $\Gamma_{\Sigma}$ is a fat graph with one admissible leaf and $m$ other leaves $L = \{ l_1, \ldots, l_m \}$.
  Enumerate the vertices on the admissible cycle of $\Gamma_{\Sigma}$ by the set $[n]$ using the order in which they occur on the admissible cycle starting at the vertex connected to the admissible leaf.
  See Figure \ref{figure:sd_construction}.
  \begin{figure}[ht]
    \centering
    \inkpic[.25\columnwidth]{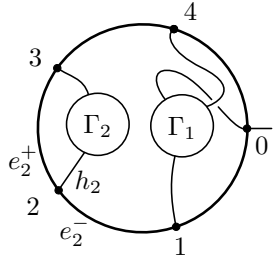}
    \caption{
      Following our convention, the orientation of the vertices is inherited by the (clockwise) orientation of the plane.
      The cyclic order of the half edges attached at the vertex $2$ is $(e_2^-, e_2^+, h_2)$.
    }
    \label{figure:sd_construction}
  \end{figure}
  Let $\Gamma_{in}\subset \Gamma_{\Sigma}$ be the sub-graph of $\Gamma_{\Sigma}$ obtained by deleting the edges of the admissible cycle and the admissible leaf.
  Then $\Gamma_{in}$  is a disjoint union of fat graphs
  \[
    \Gamma_{in}=\Gamma_1\sqcup \Gamma_2 \sqcup \ldots \sqcup \Gamma_k \,.
  \]
  We order the $\Gamma_i$'s by the smallest label of the vertex with which they are attached to on the admissible cycle, we call these the attaching vertices.
  In particular, if the vertex $0$ in $\Gamma_\Sigma$ is of valence three, then $\Gamma_1$ is a single vertex.
  In this case, we call $\Sigma$ suspended and $\Gamma_1$ the suspension disk.
  In every other case, each $\Gamma_i$ is a fat graph with leaves $L_i \subset L$.
  Moreover, each $\Gamma_i$ has least one vertex of valence one with labels in $[n]$.
  Observe that each boundary cycle of $\Gamma_i$ which is not connected to an attaching vertex has exactly one leaf.
  Any two fat graphs representing $\Sigma$ have the same topological type because they are connected by a zigzag of collapses of inner-edges which are not loops.
  Therefore, the topological type of $\Gamma_i$ is independent of the choice of representative of $\Gamma_{\Sigma}$.

  Then set $\lambda\in \Symm([n]\cup L)$ to be the permutation with cycles given by the cyclic order in which the leaves and attaching vertices occur in the boundary cycles of the $\Gamma_i$'s.
  Note in particular that, if $\Sigma$ is suspended, then $(0)$ is also a cycle of $\lambda$.
  Finally, set $S_i=(g_i, m_i, A_i)$, where $g_i$ is the genus of $\Gamma_i$, $m_i$ is the number of boundary cycles of $\Gamma_i$ which are not connected to a leaf or an attaching vertex and
  $A_i$ is the set of cycles of $\lambda$ which correspond to the leaves and attaching vertices of $\Gamma_i$.
  In particular, if $\Gamma_i$ is just a leaf $l_j$ with attaching vertex $r$, then $S_i = (g_i, m_i, A_i)$ is a disk with two marked points on the boundary, i.e., $A_i =\{ (r, l_j) \}$, $g_i = 0$ and $m_i = 0$.
  
  Conditions \ref{data_partition_par_en} and \ref{data_attach_all_surfaces_par_en} and \ref{data_suspension_disc_par_en} hold by construction.
  Moreover, we have $m_i=0$ for every $i$ by Definition \ref{definition:fattening} and Remark \ref{remark:fattening_graph}.
  Then, condition \ref{data_conditions_leaves_par_en} holds because the cycles of $\rho = \lambda^{-1}(0\ 1\ \ldots\ n)$ are in one-to-one correspondence with the boundary cycles of the Sullivan diagram and in each boundary cycle there is exactly one leaf.

  On the other hand, if $\Sigma\in \SDs_{g,m}$, then first choose an order of the leaves $L = \{ l_1, \ldots, l_m \}$ of $\Sigma$ and then do the construction above to obtain a combinatorial 1-Sullivan diagram $(\lambda, S_1, \ldots, S_k)$.
  Observe that a different choice of an ordering of $L$ and doing the above construction afterwards results in a combinatorial 1-Sullivan diagram which is equivalent to $(\lambda, S_1, \ldots, S_k)$.
  Finally, if $\Sigma\in \SDs_g^m$, then the construction above restricts to the case where the $\Gamma_i$'s have no leaves and $m_i$ is possibly non-zero.

  \medskip
  To go the other way around consider a combinatorial 1-Sullivan diagram $(\lambda, S_1, \ldots, S_k)$ of degree $n$ and let $C$ be an embedded circle on the plane with $n+1$ marked points enumerated by $[n]$ in clockwise order.
  We will use the ghost surfaces to construct fat graphs $\Gamma_1, \Gamma_2, \ldots, \Gamma_k$ with attaching vertices enumerated by $[n]$.
  The attaching vertices are of valence one, except possibly in the case where exactly one $\Gamma_i$ is a single vertex (which happens when $S_i = (0,0, \{ (0) \})$).
  We attach the fat graphs to the circle using the label of the attaching vertices.
  Notice that this gives the circle the structure of a graph by considering the attaching points as vertices and the intervals between them as edges.
  We need to give a fat structure at the attaching points and to add the admissible leaf.
  Let $x$ be an attaching point on the circle, the embedding of the circle gives the notion of incoming and outgoing half edges on $x$ in clockwise direction, say $e_{x}^-$ and $e_{x}^+$ respectively.
  The cyclic ordering at $x$ is given by $(e_x^+,  h_x, e_x^-)$, where $h_x$ is the half edge attached to the vertex $x$.
  Informally, this is to say all graphs are attached on the ``inside of the circles".
  Following the same convention, we attach the admissible leaf at the marked points $0$ from the outside, see Figure \ref{figure:sd_construction}.

  We now describe how to construct the $\Gamma_i$'s.
  The ghost surfaces and the fat structure will give topological types of surfaces with marked points at the boundary which we denote by $[S_i]$.
  We choose $\Gamma_i$ to be a fat graph of topological type $[S_i]$.
  Since two different choices of fat graphs are connected by a zigzag of collapses of inner-edges which are not loops, they all give the same Sullivan diagram.
  It is only left to describe how to obtain the $[S_i]$'s.
  If $S_i$ is the suspension disk, then $\Gamma_i$ is a single vertex.
  In all other cases, let $b_i:=|A_i|\in\mathbb{N}_{>0}$ and let $M_i\subset [n]\cup L$ to be the subset given by all the elements of the cycles of $A_i$.
  Then $[S_i]$ is the surface of genus $g_i$ with $m_i$ punctures and $b_i$ boundary components.
  The set $M_i$ is the set of marked points at the boundary of $S_i$ and the cyclic ordering on which they occur is given by $\lambda$.
  See Figure \ref{figure:unparametrized} for an example in the unparametrized case.
  \begin{figure}[ht]
    \centering
    \inkpic[.4\columnwidth]{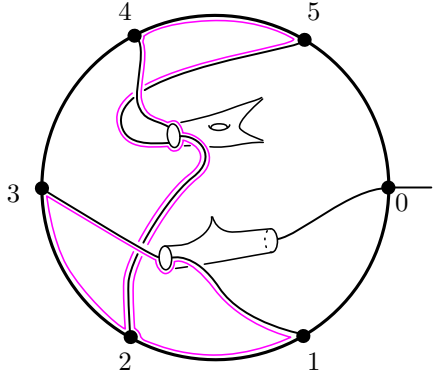}
    \caption{\label{figure:unparametrized}
      We interpret a combinatorial 1-Sullivan diagram $\Sigma = (\lambda, S_1, S_2)$ as cell of $\SDs_g^m$.
      It is $\lambda = (0)(1\ 3)(2\ 5\ 4)$, $\rho = (0\ 3\ 5)(1\ 4\ 2)$, $S_1 = (0,1,\{(0),(1\ 3)\})$ and $S_2 = (1,2, \{(2\ 5\ 4)\})$.
      We highlight one of the non-degenerate outgoing boundary curves of the thickened surface $S_\Sigma$.
    }
  \end{figure}
  In the unenumerated case, the equivalence relation gives that any two choices of enumeration of the leaves of $\Gamma_i$ are equivalent i.e., the leaves are unenumerated.
  See Figure \ref{figure:parametrized_and_enumerated} for an example of the parametrized, enumerated case.

  Condition \ref{data_partition_par_en} implies that there are no repeated markings on the boundary of the ghost surfaces and each vertex on the ground circle has exactly one surface attached to it.
  Condition \ref{data_attach_all_surfaces_par_en} implies that all surfaces must be attached.
  Condition \ref{data_suspension_disc_par_en} implies that a disk with exactly one marked point at the boundary can only be attached at the vertex $0$.
  This together with condition \ref{data_partition_par_en} imply that the graph constructed is essentially trivalent at the boundary.
  Condition \ref{data_conditions_leaves_par_en} ensures that each boundary cycle of the Sullivan diagram has exactly one leaf.
  
  The two constructions presented here are inverse to each other.
\end{proof}

\begin{remark}
  \label{remark:compare_combinatorial_1_SD_with_Rad}
  In \cite{egaskupers}, the second author and Kupers provide a cellular homotopy equivalence from the harmonic compactification of the moduli space of surfaces denoted $\ov{\Rad}_p(S)$ to the space of Sullivan diagrams denoted $\ph\SDs(S)$.
  Using their construction one can assemble the composite
  \[
    \Rad_p(S) \to \ov{\Rad}_p(S) \xr{\simeq} \ph\SDs(S) 
  \]
  where the first map is the inclusion of the space of radial slit configurations into the harmonic compactification given by B\"{o}digheimer.
  Since our definition of combinatorial 1-Sullivan diagrams is inspired by the combinatorial description of the cells of B\"{o}digheimer's model for moduli space, we can give a simpler description of this composition for the case $p=1$.
  To make this precise, let us first introduce some notation.
  
  The space of radial slit configurations $\Rad_1(S)$ is a relative bi-semisimplicial set i.e., a bi-semisimplicial set with some faces missing.
  A $\Delta^p \times \Delta^q$-cell of $\Rad_1(S)$ is given by a sequence of permutations
  \[
    \tau = (\tau_q | \ldots | \tau_1) \in \Symm([p])^{q}
  \]
  subject to certain combinatorial conditions (see \cite{bodigheimer}, \cite{abhau} and \cite{BoesHermann}).
  The cycles of a permutation $\alpha \in \Symm([p])$ are the equivalence classes of the equivalence relation $\sim_{cyc}$ generated by $j \sim_{cyc} \alpha(j)$.
  More generally, given a tuple $(\beta_q, \ldots, \beta_1)$ with $\alpha = \beta_q \cdots \beta_1$, consider the equivalence relation $\sim_{cl}$ generated by $j \sim_{cl} \beta_i(j)$ for some $i$.
  The equivalence classes of $\sim_{cl}$ are called \emph{clusters}.
  By definition $\sim_{cl}$ is coarser than the equivalence relation generated by $i \sim_{cyc} \alpha(i)$, i.e.\ $j \sim_{cyc} k$ implies $i \sim_{cl} k$.
  Consequently, each cluster is the union of cycles of $\alpha$.
  Thus, we say that two cycles of $\alpha$ \emph{are in the same cluster} if they belong to the same cluster.

  The \emph{radial projection} of a cell of $\Rad_1(S)$ of type $\tau = (\tau_q | \ldots | \tau_1)$ is the Sullivan diagram $\Sigma(\tau) = (\lambda, S_1, \ldots, S_k) \in \xh{1}\SDs(S)$,
  with $\lambda^{-1} = \tau_q \cdots \tau_1$ and 
  $S_i = (0,0,A_i)$ where $A_i$ is the collection of cycles of $\lambda$ that are in the same cluster (with respect to the decomposition $\lambda^{-1} = \tau_q \cdots \tau_1$).
  A straightforward check yields that the map $\iota \colon \Rad_1(S) \to {\ensuremath{1\hspace{1pt}\mhyphen}}\SDs(S)$ is given on the interior of the cells by
  \[
    \tau \times Int(\Delta^p \times \Delta^q) \to \Sigma(\tau) \times Int( \Delta^p )
  \] 
  induced by the projection $\Delta^p \times \Delta^q \to \Delta^p$.
\end{remark}

\begin{remark}[Top degree and Euler characteristic]
  \label{remark:top_degree_of_a_SD}
  Consider a Sullivan diagram $\Sigma \in \SDs_g^m$ and let ${S_{\Sigma}}$ be the surface with decorations obtained by fattening it as in Definition \ref{definition:fattening}.
  If $\Sigma$ is of top degree, then all ghost surfaces $S_i$ are either the suspension disk or disks with  exactly two attaching points at the boundary and no punctures or degenerate boundary.
  In other words, $\Sigma$ is a graph given by a system of chords attached on the ground circle where the vertex attached to the admissible leaf is of valence three.
  See Figure \ref{top_degree} (a).
  \begin{figure}[ht]
    \inkpic[.5\columnwidth]{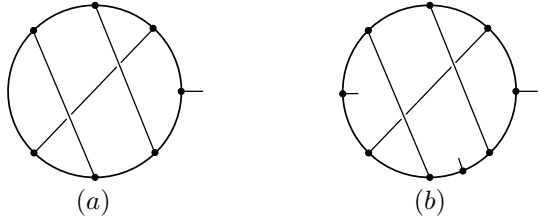}
    \caption{\label{top_degree} Examples of top dimensional cells.
      The unparametrized, unenumerated case is seen in (a) and the parametrized, unenumerated case is seen in (b).}
  \end{figure}
  Now, let $c$ be the number of chords of $\Sigma$.
  The Euler characteristic of the circle is $0$ and each time one adds a chord to the circle the Euler characteristic goes down by one.
  Therefore, we get 
  \[
    -c=\chi(\Sigma)=\chi(S_{\Sigma}).
  \]
  Thus, the degree of $\Sigma$ is 
  \[
    \deg(\Sigma)=-2\chi(S_{\Sigma}).
  \]

  On the other hand, if $\Sigma \in \SDs_{g,m}$ or $\betterwidetilde{\SDs}_{g,m}$ then the top degree cell is a diagram with chords and leaves on the admissible cycles,
  exactly one per boundary cycle of $\Sigma$.
  See Figure \ref{top_degree} (b).
  Therefore, if $\Sigma$ is of top degree we have
  \[
    \deg(\Sigma)=-2\chi(S_{\Sigma})+m.
  \]
\end{remark}

\subsection{The face map}
\label{section:sd_one_circle:face_map}
In Proposition \ref{proposition:combinatorial_1_SD}, we have seen that a Sullivan diagram $\Sigma$ is uniquely given by a combinatorial 1-Sullivan diagram.
After fixing some notation, we describe the face maps in this description in Discussion \ref{disc:faces}.

\begin{definition}
  Using the notation from the simplex category $\Delta$, let $[n] = \{ 0, \ldots, n \}$;
  the face maps are denoted by $d_i^\Delta \colon [ n-1 ] \to [n]$ and the degeneracy maps are denoted by $s_i^\Delta \colon [n+1] \to [n]$.
  Given a (possibly empty) finite set $L = \{ l_1, \ldots, l_m \}$, we extend $d_i$ and $s_i$ via $d_i(l_j) = l_j$ and $s_i(l_j) = l_j$.
\end{definition}

The upcoming maps were introduced in \cite{bodigheimerold}.

\begin{definition}
  \label{defi:D_i}
  Consider a finite set $L$ that is allowed to be empty.
  For a permutation $\alpha \in \Symm([n] \cup L)$ and $0 \le i \le n$, the \emph{$i$-th face map}
  \[
    D_i \colon \Symm([n] \cup L) \to \Symm([n-1] \cup L)
  \]
  removes $i$ from its cycle in its cycle decomposition and renormalizes the other symbols afterwards.
  More formally, denote the permutation exchanging $i$ and $\alpha^{-1}(i)$ by $(i\ \alpha^{-1}(i))$
  (it is a transposition if $i \neq \alpha^{-1}(i)$ and the identity otherwise).
  Now,
  \[
    D_i(\alpha) := s_i^\Delta \circ \alpha \circ (i\ \alpha^{-1}(i)) \circ d_i^\Delta
  \]
  where we write composition of permutations as composition of maps.
\end{definition}

\begin{definition}
  Consider a finite set $L$.
  The symmetric group $\Symm(L)$ operates on $\Symm([n] \cup L)$ by conjugation.
  The set of orbits is denoted by $\Tymm([n] \cup L)$.
  
  For $\sigma \in \Symm(L)$ and $\alpha \in \Symm([n] \cup L)$, we have $D_i(\sigma^{-1} \alpha \sigma) = \sigma^{-1} D_i(\alpha) \sigma$.
  The induced map of the set of orbits is also denoted by
  \[
    D_i \colon \Tymm([n] \cup L) \to \Tymm([n-1] \cup L) \,.
  \]
\end{definition}

The following Lemma is readily checked.

\begin{lemma}
  Consider a (possibly empty) finite set $L$.
  \begin{enumerate}
    \item We have two semisimplicial sets with $n$-simplices either $\Symm([n] \cup L)$ or $\Tymm([n] \cup L)$ and face maps $D_i$.
    \item The face maps $D_i$ commute with taking inverses i.e.\ $D_i(\alpha^{-1}) = D_i(\alpha)^{-1}$.
  \end{enumerate}
\end{lemma}

\begin{lemma}
  \label{lem:rho_face}
  Let $\Sigma\in\SDs_g^m$, $\betterwidetilde{\SDs}_g^m$, $\SDs_{g,m}$ or $\betterwidetilde{\SDs}_{g,m}$ be a cell of degree $n$ with
  a (possibly empty) set of labels $L$, with fat structure $\lambda$ and non-degenerate boundary $\rho$.
  Let $\tilde{\Sigma}:=d_i(\Sigma)$ i.e., $\tilde{\Sigma}$ is the Sullivan diagram obtained by collapsing the $i$-th edge of the admissible cycle of $\Sigma$,
  and let $\tilde{\lambda}$ and $\tilde{\rho}$ denote the fat structure and non-degenerate boundary of $\tilde{\Sigma}$.
  Then
  \[\tilde{\rho}=D_i(\rho) 
  \quad\quad\text{ and } \quad\quad
  \tilde \lambda = D_i( \lambda \circ (a\ i) )
  \]
  where $a = \rho(i)$.
\end{lemma}
\begin{proof}
  From the definition of $d_i$ and 
  since the cycles of $\rho$ are in bijection with the boundary components of the geometric realization of $\Sigma$ (see Remark \ref{remark:non-degenerate_boundary}), it is clear that $\tilde\rho = D_i(\rho)$.
  The other formula is a little computation.
  \begin{align*}
    \tilde\lambda^{-1}
      &= \tilde \rho \circ (n-1\ \ldots\ 0) \\
      &= s_i^\Delta \circ (a\ i ) \circ \rho \circ d_i^\Delta \circ (n-1\ \ldots\ 0) \\
      \intertext{The maps $d_i^\Delta \circ (n-1\ \ldots\ 0)$ and $(n\ \ldots\ i + 1\ i - 1\ \ldots\ 0) \circ d_i^\Delta$ are equal.}
      \tilde\lambda^{-1} &= s_i^\Delta \circ (a\ i) \circ \rho \circ (n\ \ldots\ i+1\ i-1\ \ldots\ 0) \circ d_i^\Delta\\
      &= s_i^\Delta \circ (a\ i) \circ \rho \circ (n\ \ldots\ i+1\ i\ i-1\ \ldots\ 0) \circ (i\ i+1) \circ d_i^\Delta \\
      &= s_i^\Delta \circ (a\ i) \circ \lambda^{-1} \circ (i\ i+1) \circ d_i^\Delta \,.
      \intertext{Since $(a\ i) \circ \lambda^{-1} = (a\ i) \circ \rho \circ (n\ \ldots\ i+1\ i\ i-1\ \ldots\ 0)$ maps $i+1$ to $i$ we end (by the definition of $D_i$) with the following.}
      \tilde\lambda^{-1}&= D_i( (a\ i) \circ \lambda^{-1} ) \,.
  \end{align*}
  The claim follows since $D_i$ commutes with taking the inverse.
\end{proof}

\begin{discussion}
  \label{disc:faces}
  Let $\Sigma=(\lambda,S_1,S_2,\ldots,S_k)$ be a cell of degree $n$, let $\rho$ be its non-degenerate boundary.
  We describe the $i$-th face of $\Sigma$ which is denoted by
  $\tilde{\Sigma} := d_i(\Sigma) = (\tilde{\lambda},\tilde{S}_1,\tilde{S}_2,\ldots,\tilde{S}_{\tilde{k}})$.
  Lemma \ref{lem:rho_face} gives us $\tilde{\lambda}$ and $\tilde{\rho}$ so it is only left to determine the ghost surfaces of $\tilde{\Sigma}$.

  Let $S_{j}$ denote the surface attached at $i$, $S_{t}$ denote the surface attached at $i+1$ and $a = \rho(i)$.
  Since $\rho=\lambda^{-1} (0\ \ldots\ n)$ then we have that
  \[
    \lambda(a)=\lambda(\rho(i))=i+1.
  \]
  Moreover Lemma \ref{lem:rho_face} gives that 
  \[
    \tilde{\lambda}=D_i(\underline{\lambda}) 
    \quad\quad \text{where} \quad\quad
    \underline{\lambda}=\lambda\circ(i\ a ).
  \]
  Let $\Lambda, \tilde{\Lambda}$ and $\underline{\Lambda}$ denote the sets of cycles of $\lambda$, $\tilde{\lambda}$ and $\underline{\lambda}$ respectively.
  Note that $D_i$ induces a partially defined map 
  \begin{align}
    \label{Di_lambda} D_i^\ast \colon \underline{\Lambda} \dashrightarrow \tilde{\Lambda}.
  \end{align}
  which is a bijection if $i$ is not a fixed point of $\underline{\lambda}$;
  otherwise is it a bijection after removing the cycle $(i)$.
  Therefore, given a partition of $\underline{\Lambda}$ there is an induced partition of $\tilde{\Lambda}$.
  We use this setup to study the effect of the differential on the ghost surfaces separately in different possible cases.
  \begin{description}
    \item[Case 1]
      \underline{Let $a = i$.}
      In this case, $\lambda(i) = i+1$ and so $S_j=S_{t}$, see Figure \ref{figure:case1}.
      Note in particular that this case can not happen if $\Sigma$ is a cell in either ${\SDs}_{g,m}$ or $\betterwidetilde{\SDs}_{g,m}$ (by Condition \ref{data_conditions_leaves_par_en} in Definition \ref{definition:combinatorial_1_SD_par_en}).
      Now, when taking the $i$-th face of $\Sigma$ the boundary cycle corresponding to $(i)$ in $\rho$ disappears and degenerates to a puncture of $S_j$.
      Thus, we have to increase the number of punctures of $S_j$ by one while leaving the other data fixed.
      More precisely,  $\underline{\lambda}=\lambda$.
      Thus, the partition of $\Lambda$ induces a partition of $\underline{\Lambda}$ and by (\ref{Di_lambda}) this induces a partition of $\tilde{\Lambda}$ into $k$ blocks which we denote $\{\tilde{A}_i\}_i$.
      Then, $\tilde{S_r}=(g_r, m_r, \tilde{A}_r)$ for any $r\neq j$ and $\tilde{S}_j=(g_j, m_j+1, \tilde{A}_j)$.
      \begin{figure}[ht]
        \inkpic[.75\columnwidth]{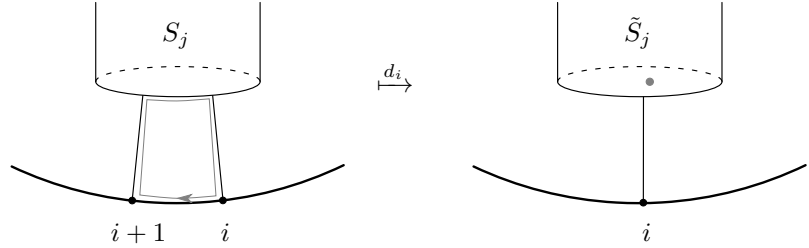}
        \caption{\label{figure:case1} Example of Case 1.
        The non degenerate boundary cycle $(i)$ of $\rho$ degenerates into a puncture of $S_j$ which we indicate by a gray dot.}
      \end{figure}
    \item[Case 2]
      \underline{Let $a \neq i$ with $a$ and $i$ in different cycles of $\lambda$.}
      In this case, $i+1$ and $i$ are in different cycles of $\lambda$.
      When taking the $i$-th face we merge the boundary of $S_{j}$ attached to $i$ to the boundary of $S_{t}$ attached to $i+1$.
      There are two possible sub-cases which we will call (a) and (b).
      We begin with the geometric meaning of the two sub-cases and provide the precise formulation afterwards.
      
      (a)  If $S_{j}\neq S_{t}$, then we merge both surfaces i.e., we add their genus and number of punctures and the remaining ghost surfaces remain unchanged, see Figure \ref{figure:case2_1}.
      \begin{figure}[ht]
        \inkpic[\columnwidth]{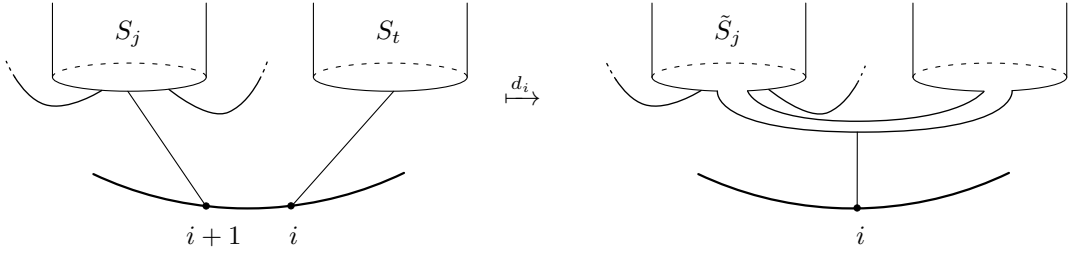}
        \caption{
          \label{figure:case2_1}
          Example of Case 2 (a).
          The face map glues together the boundaries of two ghost surfaces $S_j$ and $S_t$ resulting in a ghost surface $S_{\tilde j}$.
        }
      \end{figure}
      
      (b)  If $S_j= S_{t}$, then we merge both boundary components, which increases the genus of the surface by one, and the remaining ghost surfaces remain unchanged, see Figure \ref{figure:case2_2}.
      \begin{figure}[ht]
        \inkpic[\columnwidth]{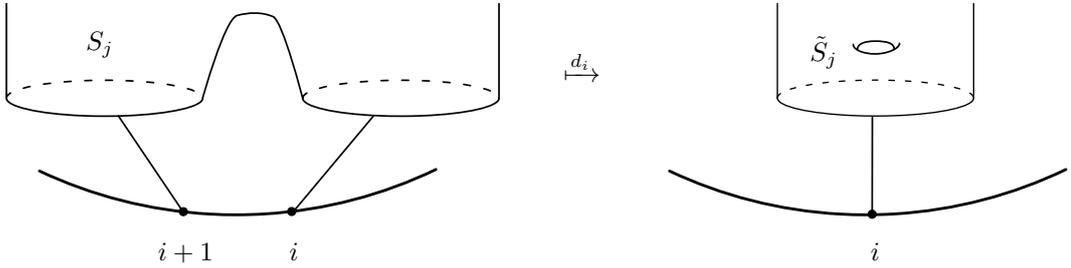}
        \caption{
          \label{figure:case2_2}
          Example of Case 2(b).
          The face map glues together two different boundary components of a ghost surface $S_j$ resulting in a ghost surface $S_{\tilde j}$
          having the same number of punctures while the genus is increased by one and the number of boundary components is decreased by one.
        }
      \end{figure}
      
      More precisely, $\underline{\lambda}$ is the permutation obtained from $\lambda$ by merging the cycles of $i$ and $i+1$.
      Therefore, by construction $\underline{\lambda}$ has one cycle less than $\lambda$.
      The two cases are as follows:
      
      (a) The cycles containing $i$ and $i+1$ belong to different blocks in the partition of $\Lambda$.
      Then, the partition of $\Lambda$ naturally induces a partition of $\underline{\Lambda}$ given by merging the blocks corresponding to the cycles containing $i$ and $i+1$.
      This partition induces a partition of $\tilde{\Lambda}$ into $k-1$ blocks which we denote $\{\tilde{A}_i\}_i$.
      Then we have  $\tilde{S_r}=(g_r, m_r, \tilde{A}_r)$ for any $r \neq j, t$ and
      $\tilde{S}_j=(g_{j}+g_{t}, m_{j}+m_{t}, \tilde{A}_j)$.
      
      (b) The cycles containing $i$ and $i+1$ belong to the same block in the partition of $\Lambda$.
      Then the partition of $\Lambda$ naturally induces a partition of $\underline{\Lambda}$ into $k$ blocks.
      This again induces a partition of $\tilde{\Lambda}$ into $k$ blocks which we denote $\{\tilde{A}_i\}_i$.
      Then, $\tilde{S_r}=(g_r, m_r, \tilde{A}_r)$ for any $r\neq j$ and $\tilde{S}_j=(g_j+1, m_j, \tilde{A}_j)$.
  \item[Case 3] 
    \underline{Let $a \neq i$ with $a$ and $i$ in the same cycle of $\lambda$.}
    In particular $i+1$ and $i$ are in the same cycle of $\lambda$ and $S_j= S_{t}$, see Figure \ref{figure:case3}.
    When taking the $i$-th face in this case, the attaching chord at $i$ moves onto the ghost surface and separates the attached boundary into two pieces.
    The other ghost surfaces remain unchanged.
    \begin{figure}[ht]
      \inkpic[.9\columnwidth]{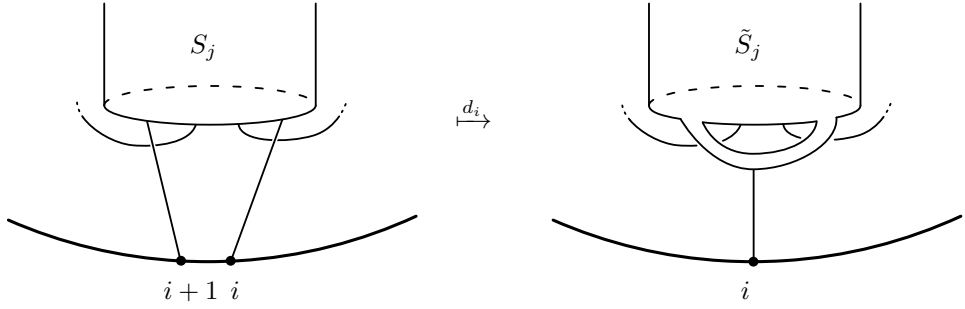}
      \caption{
        \label{figure:case3}
        Example of Case 3.
        The face map splits the boundary of a ghost surface into two parts.
      }
    \end{figure}
    More precisely, $\underline{\lambda}$ is obtained from $\lambda$ by splitting up the cycle containing $i$ and $i+1$ into two cycles.
    Then, the partition of $\Lambda$ naturally induces a partition of $\underline{\Lambda}$ where the two cycles obtained by splitting the cycle containing $i$ and $i+1$ belong to the same block.
    By (\ref{Di_lambda}) this induces a partition of $\tilde{\Lambda}$ into $k$ blocks which we denote $\{\tilde{A}_i\}_i$.
    Then, for any $r $ we have that $\tilde{S_r}=(g_r, m_r, \tilde{A}_r)$.
  \end{description}  
\end{discussion}

\subsection{The stabilization map}
\label{section:sd_one_circle:stabilization_map}
In this section we describe the stabilization map for Sullivan diagrams with respect to genus.
First, we briefly recall the stabilization for the mapping class groups.
Recall that an oriented cobordism $S$ is an oriented surface with parametrized boundary (or equivalently exactly one marked point on each boundary component) together with a partition of its boundary into incoming and outgoing.
Let $T$ be the cobordism with one incoming and one outgoing boundary component and genus one.
Let $S$ be an arbitrary cobordism with exactly one outgoing boundary component.
Then we can construct the composite cobordism $S\circ T$ obtained by ``sewing" the incoming boundary of $T$ to the outgoing boundary of $S$ using their parametrizations.
See Figure \ref{stabilization_map_mod_spc}.
\begin{figure}[ht]
  \centering
  \inkpic[.8\columnwidth]{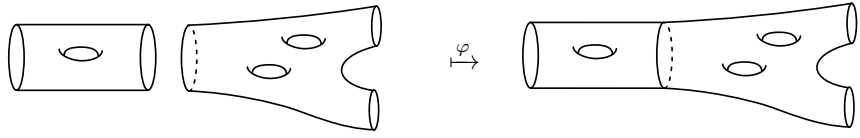}
  \caption{\label{stabilization_map_mod_spc}The stabilization map is induced from sewing surfaces.}
\end{figure}

This construction induces a homomorphism of groups 
and thus a continuous map of spaces 
\begin{equation}
  \label{stabilization_MCG}
  \varphi \colon B\Mod(S)\to B\Mod(S\circ T)
\end{equation}
which is called the stabilization map.
In \cite{Harer_stable} Harer showed that the induced map in homology is an isomorphism in a range.
In particular if $S$ is a cobordism with only one outgoing boundary component and no incoming boundary component then
the induced map in $H_*$ is an isomorphism for  $g \ge \frac{3\ast + 2}{2}$, see \cite{RandalWilliams_resolution_of_moduli_spaces}.
There is a similar picture in Sullivan diagrams.

\begin{definition}
  Let $\SDs(S)$ denote the component of the space of 1-Sullivan diagrams of topological type $S$ where we consider the boundary cycle of the admissible leaf to be the outgoing boundary.
  The \emph{stabilization map} is the map of spaces
  \[
    \varphi\colon \SDs(S)\cof \SDs(S\circ T)
  \]
  given by $\Sigma\mapsto \Sigma\circ \tau$ where $\Sigma\circ \tau$ is the Sullivan diagram obtained from $\Sigma$ by increasing the genus of the ghost surface attached at zero by one.
  We show the stabilization map in Figure \ref{stabilization_map_SD}.
  One readily checks that $\varphi$ is a semisimplicial map, thus it induces a degree zero chain map
  \[
    \varphi\colon \SD(S)\cof \SD(S\circ T)
  \]
  which, by abuse of language, we also call the stabilization map.
\end{definition}

\begin{figure}[ht]
  \centering
  \inkpic[.7\columnwidth]{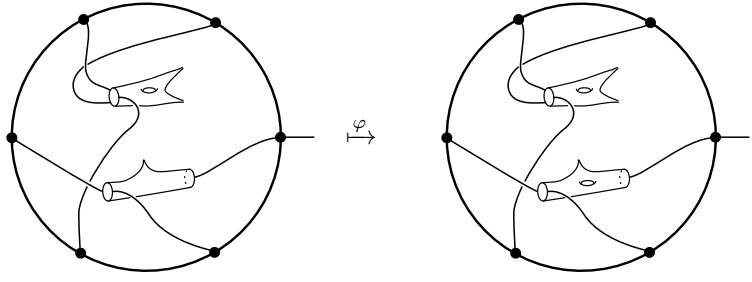}
  \caption{\label{stabilization_map_SD} The stabilization map increases the genus of the ghost surface attached at the vertex 0.}
\end{figure}

The next Lemma is readily checked.

\begin{lemma}
  \label{lemma:stabilization_map_as_inclusion_of_a_sub-complex}
  The stabilization map $\varphi \colon \SDs(S) \to \SDs(S\circ T)$ is an injective semisimplicial map.
  In particular, it identifies $\SD(S)$ with a sub-complex of $\SD(S\circ T)$.
\end{lemma}

\begin{remark}
  Our choice of notation is due to the fact that $\Sigma\circ \tau$ is the diagram obtained from $\Sigma$ by composition with the class $\tau$ shown in Figure \ref{stabilization_map_SD_prop} (b),
  using the PROP structure described in Definition \ref{def:composition}.
  Moreover, the stabilization map on Sullivan diagrams extends the stabilization map on the level of mapping class groups.
  
  To see this, recall that we have a quotient map of PROPs
  \[
    \ph BW_{\mathrm{graphs}}(S)\fib \ph\SD(S)
  \]
  where the left hand side is Costello's chain complex of black and white graphs, which models the two-dimensional cobordism category.
  See Section \ref{subsection:BW_graphs} for more details.
  The stabilization map on the level of mapping class groups (\ref{stabilization_MCG}) can be modeled in this setup as
  the composition with a degree zero class of the topological type of the cobordism $T$ of genus one with one incoming and one outgoing boundary component.
  
  \begin{figure}[ht]
    \centering
    \inkpic[.4\columnwidth]{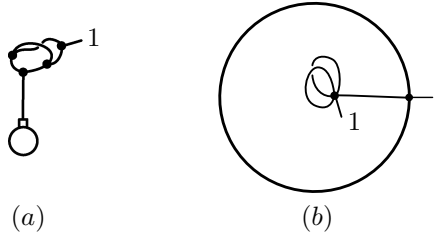}
    \caption{\label{stabilization_map_SD_prop} 
    (a) The black and white graph $\Upsilon$.
    Composition with this class represents the stabilization map on mapping class groups in this model.
    (b) The Sullivan diagram $\tau$ corresponding to the black and white graph $\Upsilon$.
    Composition with $\tau$ is the stabilization map on Sullivan diagrams.}
  \end{figure}
  
  In Figure \ref{stabilization_map_SD_prop} (a) one can see a degree 0 class in the chain complex of black and white graphs using the notation conventions of \cite{wahlwesterland}.
  We denote this class by $\Upsilon$.
  In Figure \ref{stabilization_map_SD_prop} (b) we show the Sullivan diagram corresponding to $\Upsilon$ under this quotient map.
  We call this Sullivan diagram $\tau$.
  One can readily check that the Sullivan diagram $\Sigma\circ \tau$ is the diagram obtained from $\Sigma$ by this PROP composition with the class $\tau$.
  This shows that the following diagram in chain complexes commutes.
  
  \begin{center}
    \begin{tikzpicture}[xscale=1.5,yscale=1.1]
      \path node   (c) at (0,0) {$\SD(S)$}
      node   (d) at (3,0) {$\SD(S \circ T)$}
      node   (a) at (0,1) {$\ph BW_{\mathrm{graphs}}(S)$} 
      node   (b) at (3,1) {$\ph BW_{\mathrm{graphs}}(S \circ T)$}
      node   (a0) at (-1.7,1) {$C_*(B\Mod(S))\cong$}
      node   (b0) at (5.05,1) {$\cong  C_*(B\Mod(S \circ T))$};
      \draw[->]   (a) -- node[above]{$\varphi$}    (b);
      \draw[->]   (c) -- node[above]{$\varphi$}    (d);
      \draw[->>]  (a) -- (c);
      \draw[->>]  (b) -- (d);
    \end{tikzpicture}
  \end{center}

\end{remark}

\section{The spaces of Sullivan diagrams and the stabilization maps are highly connected}
\label{section:vanishing}

We now state our vanishing results.

\begin{theorem}[Theorem A]\hspace{2mm}
  \label{theorem_a}%
  Let $g \ge 0$ and $m \ge 1$.
  The spaces $\betterwidetilde{\SDs}_{g,m}$, $\SDs_{g,m}$ $\betterwidetilde\SDs_g^m$ or $\SDs_g^m$ are highly connected, i.e.
  \begin{align*}
   \pi_\ast( \betterwidetilde{\SDs}_g^m ) = 0 \mspc{and}{20} \pi_\ast( \betterwidetilde{\SDs}_{g,m} ) = 0 \mspc{for}{20} \ast \le m-2
  \end{align*}
  and
  \begin{align*}
    \pi_\ast( \SDs_g^m ) = 0 \mspc{and}{20} \pi_\ast( \SDs_{g,m} ) = 0 \mspc{for}{20} \ast \le m'
  \end{align*}
  for $m'$ the largest even number smaller than $m$.
\end{theorem}

\begin{theorem}[Theorem B]\hspace{2mm}
  \label{theorem_b}%
  Let $g \ge 0$ and $m > 2$.
  The stabilization maps are highly connected.
  More precisely, the maps
  \begin{align*}
    \betterwidetilde{\SDs}_g^m \xr{\varphi} \betterwidetilde{\SDs}_{g+1}^m \mspc{and}{20} \betterwidetilde{\SDs}_{g,m} \xr{\varphi} \betterwidetilde{\SDs}_{g+1,m}
  \end{align*}
  are $(g + m - 2)$-connected and the maps
  \begin{align*}
    \SDs_g^m \xr{\varphi} \SDs_{g+1}^m \mspc{and}{20} \SDs_{g,m} \xr{\varphi} \SDs_{g+1,m}
  \end{align*}
  are $(g + m')$-connected where $m'$ is the largest even number smaller than $m$.
\end{theorem}

Moreover, the structure of the proofs show that a huge amount of cells do not contribute to the homology.
Let us make this precise.

\begin{proposition}
  \label{proposition:support_of_homology}
  Let $\SD(g,m) = \betterwidetilde{\SD}_{g,m}$, $\SD_{g,m}$ $\betterwidetilde\SD_g^m$ or $\SD_g^m$ and let $R$ be an arbitrary coefficient group.
  Denote the sub-complex of Sullivan diagrams with at least $k \ge 0$ degenerate boundary cycles by $B_k$.
  That is, the generators of $B_k$ are Sullivan diagrams whose ghost surfaces have in total at least $k$ punctures or degenerate boundary components.
  \begin{enumerate}[label={(\roman*)}]
    \item \label{proposition:support_of_homology:not_supported}
      Every homology class $x \in H_\ast(\SD; R)$ is represented by a chain $\sum \kappa_i c_i$ with $c_i$ a non-degenerate Sullivan diagram i.e., $c_i \notin B_1$.
    \item \label{proposition:support_of_homology:injection}
      Let $B \subseteq B_2$ be an arbitrary sub-complex.
      Then, the projection $p \colon \SD \to \SD / B$ induces a split-injective map in homology
      \[
        H_\ast( \SD; R ) \to H_\ast( \SD / B; R) \,.
      \]
    \item \label{proposition:support_of_homology:sharp_bound}
      The projection $p \colon \SD_0^m \to \SD_0^m / B_1$ is not injective in homology if $m=6,8$.
  \end{enumerate}
\end{proposition}

Let us say a few words on the proof and how this section is organized.
Our main technical tool is discrete Morse theory introduced by Forman in \cite{Forman}.
For convenience of the reader, we review the basic setup and results from discrete Morse theory in Subsection \ref{subsection:DMT}.
In Subsection \ref{subsection:vanishing_pi_and_geometric_idea} we show that the spaces in question are simply connected.
The computation carried out there is a simplified version of the general Morse flow to be constructed.
The Morse flow will be perfect in a certain range.
Then, our results are implied by Hurewicz Theorem.
The details are carried out in Subsections \ref{subsection:fans_and_proof_thm_A} and \ref{subsection:fences_and_proof_thm_B}.
The proofs of two technical Lemmas are postponed to Section \ref{section:proof_AB}.

\subsection{Discrete Morse theory review}
\label{subsection:DMT}

In this section we review the elements of discrete Morse theory which was introduced by Forman in \cite{Forman}.
We briefly recall the geometric motivation in a simplified setting.
Consider a finite, regular simplicial complex $X$ and
assume we have a top dimensional simplex $\Delta$ of $X$ with a free face $d_i\Delta$ (i.e.\ $d_i\Delta$ appears exactly once as a face of a top dimensional simplex).
Retracting $\Delta$ onto $\Lambda^i\Delta = \del \Delta - d_i\Delta$ results in a finite, regular simplicial complex $X' \simeq X$ having two simplicies fewer then $X$.
The cell $\Delta$ is called collapsible and $d_i\Delta$ redundant.
Repeating this process results in a sequence of pairs of collapsibles and redundants.
This set of (disjoint) pairs is called \emph{discrete Morse flow} and all non-paired simplices are called \emph{essential}.
Performing the collapses one gets a smaller simplicial complex $M \simeq X$ and $M$ has as many simplices as there are essentials in $X$.

We need a more general version where a non-regular cell complex $X$ and non-free faces are allowed.
We only describe this in the algebraic setting.

\begin{definition}[Cellular Graph]
  Let $R$ be a coefficient ring.
  A directed graph $C = (V,E,d,\theta)$ with \emph{vertices $V$}, \emph{edges $E \subseteq V \times V$}, \emph{degree $d \colon V \to \mathbb Z$} and
  \emph{coefficients} $\theta \colon E \to R$ is said to be \emph{cellular}
  if it has only finitely many vertices in each grading and every edge $(v,w) \in E$ decreases the grading by one and has non-vanishing coefficient, i.e.
  \[
    d(v) = d(w) + 1 \mspc{and}{20} \theta(v,w) \neq 0 \,.
  \]
  Let $K_\bullet$ be a chain complex of finite type that is free in each degree and has a distinguished choice of basis, e.g.\ the cellular complex of a CW complex of finite type.
  Its associated cellular graph has vertices given by the basis elements and edges given by the differential.
  The grading of the vertices is given by the degree of the basis elements in the chain complex and the coefficients of the edges by their coefficient in the differential.
\end{definition}

\begin{example}
  Consider the integral chain complex $K_\bullet$ concentrated in degrees $0$ and $1$ with $K_0 = \mathbb Z \langle e_1, e_2, e_3 \rangle$, $K_1 = \mathbb Z \langle f_1, f_2, f_3 \rangle$;
  the only non-trivial differential is
  \[
    \left( \begin{array}{rrr}
      1 & -1 & 0  \\
      2 & 2  & 1  \\
      0 & 5  & -1 
    \end{array} \right) \,.
  \]
  The cellular graph is
  \[
    \begin{tikzpicture}[scale=2.0]
      \node (e1) at (1,2)   {$e_1$};
      \node (e2) at (2.5,2) {$e_2$};
      \node (e3) at (4,2)   {$e_3$};
      \node (g0) at (5,2)   {$d = 1$};
      \node (f1) at (1,1)   {$f_1$};
      \node (f2) at (2.5,1) {$f_2$};
      \node (f3) at (4,1)   {$f_3$};
      \node (g1) at (5,1)   {$d = 0$};
      \node (pt) at (5.5,1.5) {.};
      
      \draw[->] (e1) to node [auto, swap] {$1$} (f1);
      \draw[->] (e1) to node [auto, very near start] {$2$} (f2);
      \draw[->] (e2) to node [auto, very near start, swap] {$-1$} (f1);
      \draw[->] (e2) to node [auto] {$2$} (f2);
      \draw[->] (e2) to node [auto, very near start] {$5$} (f3);
      \draw[->] (e3) to node [auto, very near start, swap] {$1$} (f2);
      \draw[->] (e3) to node [auto] {$-1$} (f3);
    \end{tikzpicture}
  \]
\end{example}

\begin{definition}[Matchings]
  Let $C = (V,E,d,\theta)$ be a cellular graph.
  A collection of edges $F \subseteq E$ is a \emph{matching} if any two edges in $F$ have disjoint vertices and
  if $\theta(v,w)$ is invertible for every edge $(v,w) \in F$.
  In this case, a vertex $v \in V$ is called
  \begin{cond}
    \emph{collapsible} & if there exists an edge $(v,w) \in F$ \\
    \emph{redundant}   & if there exists an edge $(u,v) \in F$ \\
    \emph{essential}   & else.
  \end{cond}
\end{definition}

\begin{definition}[$F$-inverted graph]
  Let $C = (V,E,d,\theta)$ be a cellular graph and $F \subseteq E$ a matching.
  The \emph{$F$-inverted graph} $\finv = (V, E_F, d, \theta_F)$ is obtained from $C$ by inverting the edges in $F$ i.e.\
  \begin{align}
    E_F &=  (E-F) \sqcup \{ (w,v) \mid (v,w) \in F \} \\
    \intertext{and inverting their coefficient in $R$ i.e.}
    \theta_F(x,y) &=
      \begin{cases}
        \theta(x,y) & (x,y) \in E - F \\
        - \theta(y,x)^{-1} & (y,x) \in F
      \end{cases}
  \end{align} 
  Every edge in the $F$-inverted graph either decreases or increases the grading by one and so we write $v \ad w$ if $(v,w) \in E - F$ and $v \au w$ else.
  
  A path in $\finv$ from $x$ to $y$ is denoted by $\gamma \colon x \leadsto y$.
  The set of paths in the $F$-inverted graph is
  \begin{align}
    P_F(y, x) = \{ \gamma \colon x \leadsto y \text{ in } \finv \} \,.
  \end{align}
  The \emph{coefficient} $\theta_F(\gamma)$ of a path $\gamma \in P_F(y, x)$ is the product of the coefficients $\theta_F$ of its edges.
\end{definition}

\begin{definition}[Discrete Morse flow]
  Let $C$ be a cellular graph.
  A matching $F$ is a \emph{discrete Morse flow} if the $F$-inverted graph $\finv$ is \emph{acyclic}, i.e.\ it does not contain oriented loops.
\end{definition}

The following Lemma turns out to be helpful to deduce the acyclicity of a matching.

\begin{lemma}
  \label{dmt_acyclicity}
  Let $C$ be a cellular graph and $F$ a matching.
  The $F$-inverted graph is acyclic if and only if there exists a partial order $\le$ on the collapsibles such that
  the existence of a path
  \[
    c_1 \ad r_2 \au c_2 \mspc{with}{20} \text{$c_i$ collapsible}
  \]
  implies $c_1 \lneqq c_2$.
\end{lemma}

\begin{proof}
  If $\finv$ is acyclic, then the relation
  \[
    c_1 \precneqq c_2 \;\; \text{if and only if there exists a non-empty path from $c_1$ to $c_2$}
  \]
  is a partial order.
  The converse is clear.
\end{proof}

\begin{definition}[Morse Complex]
  Let $C = (V,E,d,\theta)$ be a cellular graph with Morse flow $F$.
  The \emph{Morse complex} $M_\bullet= M(C, F)_\bullet$ of $C$ and $F$ is a chain complex freely generated by the essential vertices in each grading.
  The coefficient of an essential cell $y \in M_{n-1}$ in the boundary of an essential cell $x \in M_n$ is
  \[
    \del_{y, x} = \sum_{\gamma \in P_F(y, x)} \theta_F(\gamma) \,.
  \]
\end{definition}

\begin{theorem}[Forman]
  \label{thm:dmt_homology}
  Let $C$ be the cellular graph of a chain complex $K_\bullet$.
  Let $F$ be a discrete Morse flow on $C$.
  Then, the Morse complex $M(C,F)_\bullet$ is a homotopy retract of $K_\bullet$.
  The inclusion $\iota \colon M(C,F)_\bullet \to K_\bullet$ is given by sending essential cells $x$ to
  \[
    \iota(x) =  \sum_{\gamma \in P_x} \theta_F(\gamma) t(\gamma)
  \]
  where $P_x = \{ \gamma \in P_F(y, x) \mid \deg(y) = \deg(x), y \text{ not redundant} \}$ and
  $t(\gamma)$ denotes the endpoint of a path $\gamma$.
\end{theorem}

For a proof of Theorem \ref{thm:dmt_homology}, we refer the reader to \cite[Chapter 11.3]{KozlovComAlgTop}.

\subsection{Fundamental groups and the geometric idea behind the discrete Morse flow}
\label{subsection:vanishing_pi_and_geometric_idea}
We first show that most of the spaces in question are simply connected.
The proof of this fact is a simplified version of the proof of our vanishing results.
At the end of this section we informally describe the Morse flow we construct following these ideas.

\begin{lemma}
  \label{lemma:1_connectedness}
  Let $m > 2$ and $g \ge 0$.
  The spaces $\SDs_g^m$, $\betterwidetilde\SDs_g^m$, $\SDs_{g,m}$ and $\betterwidetilde{\SDs}_{g,m}$ are simply connected.
\end{lemma}

\begin{proof}
  Let us denote by $\SDs$ either $\SDs_g^m$, $\betterwidetilde\SDs_g^m$, $\SDs_{g,m}$ or $\betterwidetilde{\SDs}_{g,m}$.
  Observe that $\SDs$ has a single zero-cell in the unenumerated case and $m$ zero-cells else.
  Let us organize the one-dimensional cells $\Sigma$ of $\SDs$.
  We have three types of one-cells $\Sigma$.
  The first type, which we denote $\alpha$, has exactly two non-degenerate boundaries.
  Else, $\Sigma$ has exactly one non-degenerate boundary.
  In this case, $\Sigma$ either has a single ghost surface with two different boundary curves attached to the admissible circle or it has two different surfaces attached.
  In the former case it is of type $\beta$ and in the latter it is of type $\gamma$.
  See the upper half of Figure \ref{fig:generators_and_relations}.
  Observe in particular that suspended cells are of type $\gamma$ by definition.
  \begin{figure}[ht]
    \resizebox{.95\columnwidth}{!}{
    \inkpic[1.1\columnwidth]{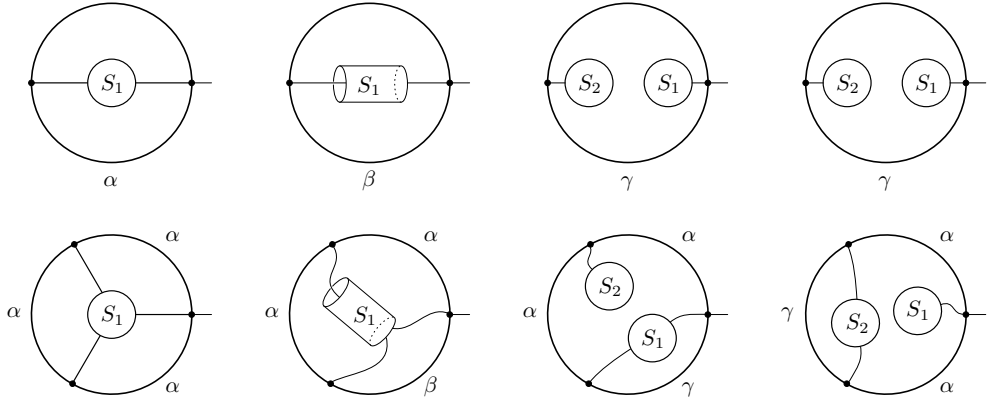}}
    \caption{\label{fig:generators_and_relations}
      The one and two-dimensional cells ordered by type.
      The first row shows the types of the one-dimensional cells.
      The second row shows the types of the two-dimensional cells.
      In the second row, we indicate the type of the face obtained by contracting one of the three edges of the admissible circle by writing the type next to the edge.
    }
  \end{figure}
  
  Now that the 1-skeleton $\SDs^{(1)}$ is understood, we attach to it certain two-cells and obtain a subspace $\SDs^{(1)} \subset C \subset \SDs^{(2)}$.
  We further construct smaller sub-complexes $A \subset B \subset C$
  and show that these are deformation retracts of $C$ and that $A$ is simply connected.
  This proves the Lemma.
  
  We first construct $A$.
  Consider the sub-complex consisting of all zero-cells and all one-cells $\Sigma$ of type $\alpha$.
  To this, we add all two-cells $\Sigma'$ having exactly three non-degenerate boundary cycles (see Figure \ref{fig:generators_and_relations}).
  The resulting space is called $A$.
  
  We now construct $B$.
  Let us denote by $m(S_i)$ the number of punctures or degenerate boundary components of the attached surface $S_i = (g_i, m_i, A_i)$.
  That is, in the unparametrized case $m(S_i) = m_i$ and in the parametrized case $m(S_i) = \#\{ (l_j) \in A_i \}$.
  
  Consider a cell $\Sigma$ of type $\beta$.
  Since $m > 2$, the unique ghost surface $S = S_1$ of $\Sigma$ has punctures or degenerate boundaries i.e.\ $m(S_1) > 0$.
  In the unenumerated case, there is a unique two-cell $\Sigma'$ with two non-degenerate boundaries such that $d_0(\Sigma') = \Sigma$ (cf.\ Figure \ref{fig:generators_and_relations}).
  In the enumerated case, there are $m(S_1)-1$ choices of $\Sigma'$ with two non-degenerate boundaries such that $d_0(\Sigma') = \Sigma$,
  because we can choose which of the enumerated degenerate boundaries (or punctures) is part of the first edge of the admissible cycle.
  Let us choose the enumerated boundary (or puncture) with the smallest label.
  To $A$, we attach all one-cells $\Sigma$ of type $\beta$ and (the uniquely chosen) cells $\Sigma'$ discussed above to obtain a space $B$.
  
  Finally we construct $C$.
  Consider a cell $\Sigma$ of type $\gamma$.
  Since $m > 2$, one of the two ghost surfaces $S_1$ and $S_2$ of $\Sigma$ has punctures or degenerate boundary i.e.\ $m(S_i) > 0$ for at least one $i$.
  Assume first that $m(S_1) > 0$.
  We proceed as before.
  In the unenumerated case, there is a unique two-cell $\Sigma'$ with three non-degenerate boundaries such that $d_0(\Sigma') = \Sigma$ (cf.\ Figure \ref{fig:generators_and_relations}).
  In the enumerated case, there are $m-1$ choices of $\Sigma'$ with two non-degenerate boundaries such that $d_0(\Sigma') = \Sigma$,
  because we can choose which of the enumerated degenerate boundaries (or punctures) is part of the first edge of the admissible cycle.
  Again, we choose the enumerated boundary (or puncture) with the smallest label.
  If $m(S_1) = 0$ then $m(S_2) > 0$ and one expands at $S_2$ as described above.
  To $B$, we attach all one-cells $\Sigma$ of type $\gamma$ and the (uniquely chosen) cells $\Sigma'$.
  This gives the sub-complex $\SDs^{(1)} \subset C \subset \SDs^{(2)}$.
  
  By construction, the cells of type $\beta$ and $\gamma$ are free faces in $C$, i.e.\ they appear only once as face of a simplex.
  Collapsing the unique two-cell they lie in provides deformation retractions 
  $C \stackrel{\simeq}{\to} B \stackrel{\simeq}{\to} A$.
  It remains to show, that $A$ is 1-connected.
  
  In the unenumerated case, this is fairly easy.
  Here we have a single zero-cell, a single one-cell (which is of type $\alpha$) and a single two-cell; the resulting space is a dunce-cap which is contractible.
  In the enumerated case, one has to work a little more.

  Observe that $A$ is the 2-skeleton of the semisimplicial set of ordered subsets of $\{ l_1, \ldots, l_m \}$:
  The $k$-simplices are named by all ordered subsets of size $k+1$ i.e.,\ 
  \[A^{(k)} = \{ (l_{i_0}, \ldots, l_{i_k}) \mid l_i \neq l_{i'} \text{ for } i \neq i' \}.\]
  The $i$-th face is given by forgetting the $i$-th entry.
  Observe further that $A$ is regular but does not have free faces.
  
  For better readability, we write $i$ instead of $l_i$
  and we denote a $k$-simplex by a tuple $(i_0, \ldots, i_k)$.
  Therefore, $A$ has $m$ zero-cells, $m (m-1)$ one-cells and $m (m-1) (m-2)$ two-cells.
  Let $A'\subset A$ denote the sub-complex given by the union of the one-cells $(1, k)$ together with their boundaries.
  Note that $A'$ is contractible sub-complex.
   We show that $A \simeq A / A'$ is simply connected.
  
  By construction, $A/ A'$ has a single zero-cell, $m(m-1) - (m-1) = (m-1)^2$ one-cells and the same number of two-cells as $A$.
  The fundamental group of $A/ A'$ is generated by the the one-cells $(a, b)$ with $a \neq 1$.
  Let us assume that $b \neq 1$.
  The two-cell $(1,a,b)$ in $A$ has the faces $(a,b), (1,a), (1,b)$.
  Therefore, $(a,b)$ with $b\neq 1$ is trivial in $\pi_1(A/A')$.
  It remains to study the cells $(a,1)$.
  The two-cell $(a,1,c)$ in $A$ has the faces $(1,c), (a,c), (a,1)$ with $c \neq 1$.
  Therefore, $(a,1)$ is also trivial in $\pi_1(A/A')$.
\end{proof}

\begin{remark}
  Repeating this argument in higher dimensions, the second author and Frank Lutz show that $\SDs_0^m$ is $1$,$2$,$3$ or $4$-connected if $m$ is at least $3$,$5$,$7$ or $9$ by providing a sequence of elementary homotopy collapses.
  Unfortunately, the argument cannot be easily extended to dimensions five or higher.
  This is due to the fact that for bigger and bigger $m$ the complex $\SDs_0^m$ contains an increasingly larger sub-complex which is contractible but not collapsible.
\end{remark}

The argument used to show triviality of the fundamental group is a toy version of the Morse flow showing that $\SDs$ is highly connected.
We now describe it informally before developing the details in the next subsections.
We start with the construction of a discrete Morse flow on the chain complex of Sullivan diagrams $\SDgm$.
In order to decide whether a given Sullivan diagram $\Sigma$ is collapsible, we look for a surface whose attaching strips resemble the shape of a fan as shown on Figure \ref{fig:fan_example}.
A cell with a fan is locally a generalization of the co-face of a cell of type $\alpha$ given in the bottom left corner of Figure \ref{fig:generators_and_relations}.
A generalization of the argument of the proof of Lemma \ref{lemma:1_connectedness} will show that the faces 
\[d_i(\Sigma)=d_{i+1}(\Sigma)=\ldots =d_{i+l-1}(\Sigma)\]
are all equal, where $i$ is the position where the fan starts.
However, they are different from $d_j(\Sigma)$ for $j \notin \{i, \ldots, i+l-1\}$.
We call $l$ the length of the fan and if it is odd, then $d_i(\Sigma)$ appears exactly once in the total boundary map of the chain complex of $\SDgm$.
In this case, we declare $\Sigma$ to be collapsible with redundant partner $d_i(\Sigma)$.
From counting dimensions it will be clear that the associated Morse complex $M$ does not have cells in low dimensions except for a single zero-cell.
But the chain complex $\SDgm$ and $M$ have the same homotopy type so the result follows.

The discrete Morse flow described above is compatible with the stabilization map i.e.,\ we have an induced discrete Morse flow on the quotient $\SD(g+1,m) / \SDgm$.
We extend this flow in order to show the connectivity increases with respect to genus.
As before, in order to decide if a cell is collapsible we look for a surface whose attaching strips resemble the shape of a fence.
A fence is a local condition analogous to that of a fan but in terms of genus instead of in terms of boundary.
Figure \ref{fig:fence} gives a schematic picture of a fence.
As before, the consecutive faces at a fence will coincide but will be different from all other faces.
A Sullivan diagram is collapsible if an attached surface has a fence of odd length and its redundant partner will be the diagram obtained by collapsing an edge on the fence.
As before a counting argument shows that the associated Morse complex has no cells in low dimensions so Theorem \ref{theorem_b} follows.

\subsection{Fans and the proof of Theorem A}
\label{subsection:fans_and_proof_thm_A}
Let $\SDgm$ be one of the four models.
In this section we construct the discrete Morse flow used to show that $\SD(g, m)$ has trivial homology in degrees smaller than $m-2$ which implies Theorem \ref{theorem_a}.

\begin{definition}[Foot-point]
  Let $\Sigma$ be a cell of degree $n$ and let $S$ be a surface of $\Sigma$.
  The \emph{foot-point} $ft(S) \in \{0 , \ldots, n \}$ is the first position on the ground circle, the surface $S$ is attached to.
\end{definition}

\begin{definition}[Fan]
  A cell $\Sigma$ of $\SD_g^m$ has a \emph{fan of length $l$ with foot-point $i$} if
  \begin{align*}
    \text{$0 \le i, \ldots, i+l-1 \le n$ are fixed points of $\rho$}
  \end{align*}
  and this condition is maximal i.e.\ it cannot be extended to $i-1$ or $l+1$.
  
  Let $\Sigma$ be a cell of $\SD_{g,m}$ and denote its set of leaves by $L$.
  We say that $\Sigma$ has a \emph{fan of length $l$ with foot-point $i$} if
  \begin{align*}
    \text{$0 \le i, \ldots, i+l-1 \le n$ are fixed points of $\rho^2$ and $\rho(i), \ldots, \rho(i+l-1) \in L$}
  \end{align*}
  and this condition is maximal i.e.\ it cannot be extended to $i-1$ or $l+1$.
  See Figure \ref{fig:fan_example} for a picture of a fan.
  \begin{figure}[ht]
    \inkpic[\columnwidth]{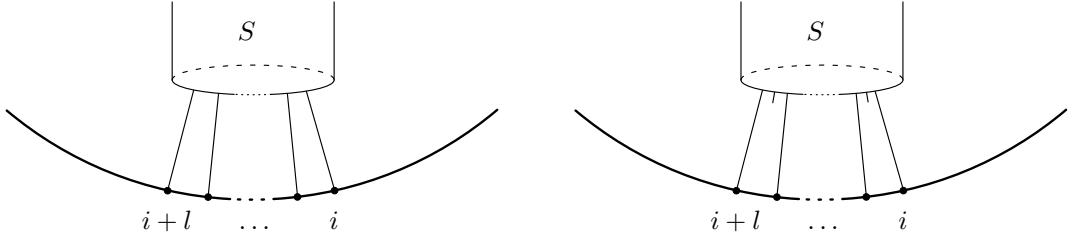}
    \caption{\label{fig:fan_example}
      The left hand side shows a fan in a cell $\Sigma$ of $\SD_g^m$ and the right hand side shows a fan in a cell $\Sigma$ of $\SD_{g,m}$.
    }
  \end{figure}
\end{definition}

\begin{lemma}
  \label{lem:faces_fan}
  Let $\Sigma$ be an $n$-cell of $\SD_g^m$ or $\SD_{g,m}$ with a fan of length $l$ with foot-point $i$.
  Then the following hold:
  \begin{enumerate}[label={(\roman*)}]
    \item \label{lemma:faces_fan:faces} $d_i(\Sigma) = \ldots = d_{i+l-1}(\Sigma) \neq d_k(\Sigma) \quad \text{for } k \notin \{ i, \ldots, i+l-1 \}.$
    \item \label{lemma:faces_fan:coefficient}  Let $\tilde \Sigma = d_i(\Sigma)$.
      Then, the coefficient of $\tilde \Sigma$ in the boundary of $\Sigma$ is
      \begin{align*}
        \del_{\tilde \Sigma, \Sigma} &=
          \begin{cases}
            (-1)^i & \text{$l$ is odd} \\
            0    & \text{$l$ is even}
          \end{cases}
      \end{align*}
  \end{enumerate}
\end{lemma}

\begin{proof}
  Let us proof \ref{lemma:faces_fan:faces} and \ref{lemma:faces_fan:coefficient} for $\Sigma \in \SD_g^m$.
  Note that the numbers $i, \ldots, i+l-1$ are consecutive fixed points of $\rho$ so $D_i(\rho) = \ldots = D_{i+l-1}(\rho)$ and
  all these faces degenerate an outgoing boundary of the same surface (see Discussion \ref{disc:faces}).
  Therefore, these faces give the same cell.
  For $k \neq i, \ldots, i+l-1$ it follows $D_i(\rho) \neq D_k(\rho)$ so the faces $d_i(\Sigma)$ and $d_k(\Sigma)$ have to be different.
  Finally, \ref{lemma:faces_fan:coefficient} is an immediate consequence of \ref{lemma:faces_fan:faces}.
  
  The proof for $\Sigma \in \SD_{g,m}$ is similar.
\end{proof}

\begin{definition}
  Let $\Sigma$ be a cell of $\SD_g^m$ or $\SD_{g,m}$ and let $S$ be a surface of $\Sigma$.
  We say that \emph{$S$ starts with a fan of length $l$} (or \emph{with an odd fan}) if
  $\Sigma$ has a fan of length $l$ (or of odd length) with foot-point $ft(S)$.
  In this case, we set $l_S = l$ and $l_S = 0$ otherwise.
\end{definition}

We are going to define a Morse flow on $\SD_g^m$ and $\SD_{g,m}$ such that all cells with punctures or degenerate boundary are either redundant or collapsible.
We say that the redundants and collapsibles are of type $0$ since we will later enlarge the Morse flow and the new redundant and collapsible cells will be denoted of type $1$.

\begin{definition}[Collapsibles and redundants of type 0]
  \label{defi:collapsible_0}
  A cell $\Sigma \in \SD_g^m$ or $\SD_{g,m}$ of positive degree is \emph{collapsible of type $0$} if it has a surface $S$ starting with an odd fan such that
  every other surface $\tilde S$ of $\Sigma$ with $ft(S) > ft(\tilde S)$ does not have punctures or degenerate boundary.
  
  Among the surfaces of $\Sigma$ fulfilling the above condition, we pick $S$ with minimal foot-point i.e.\ $i = ft(S)$ is minimal, and say that
  \emph{$S$ witnesses the collapsibility of $\Sigma$}.
  In this case, $d_i(\Sigma)$ is the \emph{redundant partner of type $0$}.
\end{definition}

\begin{lemma}
  \label{lem:redundant_parter_unenumerated}
  A cell $\Sigma \in \SD_g^m$ or $\SD_{g,m}$ is the redundant partner of some other cell {$\tilde\Sigma$ if and only $\tilde\Sigma$}
  has a surface $S$ with punctures or degenerate boundary not starting with an odd fan such that
  every other surface $\tilde S$ with $ft(S) > ft(\tilde S)$ does not have punctures or degenerate boundary nor does it start with an odd fan.
\end{lemma}

\begin{proof}
  This is a immediate consequence from the definitions and Discussion \ref{disc:faces}.
\end{proof}

\begin{remark}
  \label{remark:suspended_cells_and_collapsibility_A}
  A cell $\Sigma = (\lambda, S_1, \ldots, S_k)$ is called \emph{suspended}, if the surface $S$ with foot-point $0$ is the suspension disk $S = (0,0, \{ (0) \})$.
  Observe that a suspended cell is collapsible (resp.\ redundant) type $0$ if and only if $d_0(\Sigma)$ is collapsible (resp.\ redundant) of type $0$.
\end{remark}

\begin{lemma}[Lemma A]
  \label{lemma_a_unenumerated}
  The collapsibles and their redundant partners define a discrete Morse flow on $\SD_g^m$ and $\SD_{g,m}$.
  A cell (of positive degree) with punctures or degenerate boundary is either collapsible or redundant.
\end{lemma}

By construction, the definition of collapsibles and their redundant partners define a matching on the cellular graph $\SDs_g^m$ or $\SDs_{g,m}$.
Therefore, it is only left to show that this matching is acyclic.
We postpone the rather technical proof of this fact to Section \ref{subsection:proofs:lemma_a}.

In the enumerated case we do a matching which is in spirit similar to the unenumerated case in the sense that we will pair cells by collapsing a chamber of a fan and thus adding a puncture.
However, the structure of the matching is different, because instead of using lengths of fans we use the labels of the chambers to determine which cells are collapsible and which are redundant.
This is a technical description which we give in detail in Subsection \ref{subsection:proofs:setup_labaled_case}.
We obtain the following result.

\begin{lemma}[Lemma A]
  \label{lemma_a_enumerated}
  The cellular graphs of $\betterwidetilde\SDs_g^m$ and $\betterwidetilde\SDs_{g,m}$ admit a discrete Morse flow such that every cell with punctures or degenerate boundary is either redundant or collapsible 
  (except for a single cell of degree zero).
\end{lemma}

\begin{proof}[Proof of Theorem \ref{theorem_a} ] 
  Given a cell $\Sigma$, denote the number of punctures or degenerate boundary components by $m_\Sigma$.
  By Lemmas \ref{lemma_a_unenumerated} and \ref{lem:redundant_parter_unenumerated} we have a discrete Morse flow such that every cell $\Sigma$ with $m_\Sigma > 0$ is either redundant or collapsible.
  By Discussion \ref{disc:faces} taking a co-face of $\Sigma$, the number $m_\Sigma$ is either constant or drops by one and the unique essential cell $\Sigma$ in degree zero fulfills $m_\Sigma = m-1$.
  Thus, there are no essential cells in degrees smaller than $m-1$ (besides a single cell in degree zero) so $\SD(g,m)$ has vanishing reduced homology in degree smaller than $m-1$.
  By Lemma \ref{lemma:1_connectedness}, $\SD(g,m)$ is simply connected.
  We obtain the first part of Theorem \ref{theorem_a} --- claiming $\pi_\ast(\SDgm) = 0$ for $\ast \le m-2$ --- from Hurewicz's Theorem.
  
  In order to deduce the second part of Theorem \ref{theorem_a} let $m$ be odd (in the even case, there is nothing to show).
  Using Hurewicz's Theorem again, it is left to show that $\SDs_{g,m}$ and $\SDs_g^m$ do not have essential cells in degree $m' = m-1$.
  Since cells with punctures or degenerate boundary are not essential, there is a single cell of degree $m'$ to be considered.
  But this cell has just an odd fan of length $m'$ i.e.\ hence is not essential either.
\end{proof}

From this it follows that there is no homology supported in cells with punctures or degenerate boundary as we show below.

\begin{proof}[Proof of Proposition \ref{proposition:support_of_homology}]
  By construction of the discrete Morse flow, there are no paths (in the $F$-inverted graph) from collapsible cells with punctures or degenerate boundary to essential cells.
  The Morse complex is identified with a particular sub-complex of $\SD$ (cf.\ Theorem \ref{thm:dmt_homology}).
  It consists of cells without punctures or degenerate boundary and it is a homotopy retract of $\SD$.
  This implies \ref{proposition:support_of_homology:not_supported}.
  
  To show \ref{proposition:support_of_homology:injection}, it suffices to show that $p \colon \SD \to \SD/ B_2$ is split injective.
  To this end, consider the restriction of the discrete Morse flow on $\SD$ to $\SD /B_2$ i.e.\  a pair of cells in $\SD/B_2$ is paired if and only if it is paired in $\SD$.
  Let $e \in \SD$ be without punctures or degenerate boundary.
  Then it is either essential or collapsible and so is $p(e)$ in $\SD/B_2$ (because all faces of $e$ have at most one puncture or degenerate boundary).
  Therefore, essential cells $e \in \SD$ and $p(e) \in \SD/B_2$ have the same Morse differential in $\SD$ and $\SD/B_2$.
  In order to construct the splitting, let $f \in \SD / B_2$ be an essential cell not in the image of $p$.
  Then, $f$ has to have at least one puncture or degenerate boundary and its Morse differential in $\SD/B_2$ is always a sum of cells, each of them having at least one puncture or degenerate boundary.
  Therefore, we have a splitting by discarding all cells not in the image of $p$.
  
  From \cite{Klamt_comfrob} it follows that $H_\ast( \SD_0^m / B_1; \Z)$ is a suspension of the homology of the $m$-th braid group $Br_m$.
  Comparing Betti numbers (c.f.\ Appendix \ref{appendix:computations}), we cannot get an injection from $H_\ast( \SD_0^m; \Q)$ to any suspension of $H_\ast( Br_m ; \Q)$ for $m=6,8$.
  This proves \ref{proposition:support_of_homology:sharp_bound}.
\end{proof}

\subsection{Fences and the proof of Theorem B}
\label{subsection:fences_and_proof_thm_B}
Let $\SDgm$ be one of the four models.
In Section \ref{section:sd_one_circle:stabilization_map}, we introduced the stabilization map $\varphi \colon \SDgm \to \SD(g+1,m)$.
It identifies $\SDgm$ with a sub-complex of $\SD(g+1,m)$ by Lemma \ref{lemma:stabilization_map_as_inclusion_of_a_sub-complex}.
This inclusion is compatible with the discrete Morse flows on both complexes.
In this section, we show that the quotient $\SD(g+1, m) / \SDgm$ has trivial reduced homology in degrees smaller than $g+m-2$, by extending the discrete Morse flow.
This implies Theorem \ref{theorem_b}.

\begin{definition}[Endpoint]
  \label{defi:endpoint}
  Let $\Sigma$ be an $n$-cell of $\SD_g^m$ or $\SD_{g,m}$ and let $S$ be a surface of $\Sigma$.
  The \emph{endpoint} $end(S) \in \{ 0, \ldots, n \}$ is the last position on the ground circle the surface $S$ is attached to.
\end{definition}

\begin{definition}[Fence]
  \label{defi:fence}
  Let $\Sigma$ be a cell of $\SD_g^m$ or $\SD_{g,m}$ with at least two surfaces attached.
  A cell $\Sigma$ has a \emph{fence of length $L$ with endpoint $i$} if
  there exists a surface $S$ with foot-point $ft(S) > 0$ such that
  \begin{align}
    \label{defi:fence:fixed_pts_lambda}
    & \text{$n \ge i, \ldots, i-L+1 \ge 0$ are fixed points of $\lambda$}
  \intertext{or equivalently}
    \tag{\ref*{defi:fence:fixed_pts_lambda}'}
    \label{defi:fence:cond_rho}
    & \rho(i-1) = i, \ldots, \rho(i-L) = i-L+1
  \intertext{and}
     \label{defi:fence:attaching_pts}
    & \text{$n \ge i, \ldots, i-L \ge 0$ are attached to pairwise different boundary components of $S$} 
  \end{align}
  Furthermore, these conditions are maximal i.e.\ conditions \eqref{defi:fence:attaching_pts} and \eqref{defi:fence:fixed_pts_lambda} cannot both simultaneously be extended to $i+1$ or $L+1$.
  We allow one of the conditions to be extensible to $i-1$ or $L+1$ only if the other condition is not.
  See Figure \ref{fig:fence} for a picture of a fence.
  \begin{figure}[ht]
    \inkpic[.5\columnwidth]{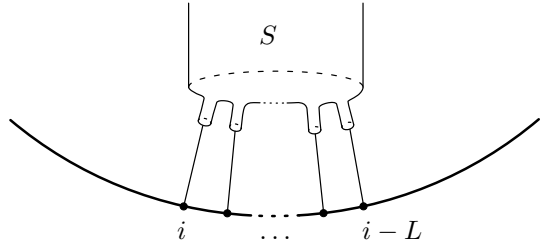}
    \caption{\label{fig:fence}A surface $S$ that ends with at fence.
      The surface $S$ is attached to the ground circle with $L$ tubes.
      Each tube has a single attaching point and these are attached to $i,\ldots, i-L+1$.
      Moreover, the surface $S$ is also attached to position $i-L$ but we do not require that this boundary component of $S$ is also a tube with a single attaching point.
      However, in this picture it is the case.
    }
  \end{figure}
\end{definition}

\begin{lemma}
  \label{lem:faces_fence}
  Let $\Sigma$ be an $n$-cell of $\SD_g^m$ or $\SD_{g,m}$ having a fence of length $L$ with endpoint $i$.
  Then the following hold:
  \begin{enumerate}[label={(\roman*)}]
    \item \label{lemma:faces_fence:faces} $d_{i-1}(\Sigma) = \ldots = d_{i-L}(\Sigma) \neq d_k(\Sigma) \quad \text{for } k \notin \{ i-1, \ldots, i-L \}$ 
    \item \label{lemma:faces_fence:coefficient} Let $\tilde \Sigma = d_{i-1}(\Sigma)$.
      The coefficient of $\tilde \Sigma$ in the boundary of $\Sigma$ is
    \begin{align*}
      \del_{\tilde \Sigma, \Sigma} &=
        \begin{cases}
          (-1)^{i-1} & \text{$L$ is odd} \\
          0    & \text{$L$ is even}
        \end{cases}
    \end{align*}
  \end{enumerate}
\end{lemma}

\begin{proof}
  To see \ref{lemma:faces_fence:faces}, note that the numbers $i, \ldots, i-L+1$ are consecutive fixed points of $\lambda = \lambda(\Sigma)$ or equivalently $\rho(i-1) = i, \ldots, \rho(i-L) = i-L+1$.
  Therefore $D_{i-1}(\rho ) = \ldots = D_{i-L}(\rho)$ and
  all these faces merge two boundary cycles of the same surface $S$ to increase the genus of $S$ by one (see Discussion \ref{disc:faces}).
  The outcome is a cell that has a fence of length exactly $L-1$ with endpoint $i-1$.
  We showed $d_{i-1}(\Sigma) = \ldots = d_{i-L}(\Sigma)$.
  
  Let us assume, we have $d_{i-1}(\Sigma) = d_k(\Sigma)$ for some $k$.
  The face $d_{i-1}$ merges two boundaries of the same surface $S$, therefore $d_k$ has to merge two boundaries of the same surface $\tilde S$.
  This and the definition of a fence imply that $k$ must not be $i$, because $d_i$ has to merge two different surfaces.
  Observe that for $k \le i-L-1$ or $k \ge i+1$, the cell $d_k(\Sigma)$ has a fence of length at least $L$ ending at $i-1$ or $i$.
  However, $d_{i-1}$ has a fence of length exactly $L-1$ with endpoint $i-1$.
  Therefore $k$ has to be one of $i-1, \ldots, i-L$.
  This proofs \ref{lemma:faces_fence:faces}.
  
  Part \ref{lemma:faces_fence:coefficient} is an immediate consequence of \ref{lemma:faces_fence:faces}.
\end{proof}

\begin{definition}
  Let $S$ be a surface of a cell $\Sigma \in \SD_g^m$ or $\SD_{g,m}$.
  We say that \emph{$S$ ends with a fence of length $L$} (or \emph{with an odd fence}) if
  $\Sigma$ has a fence of length $L$ (or of odd length) with endpoint $end(S)$.
  In this case, we set $L_S = L$ and $L_S = 0$ otherwise.
\end{definition}

We now extend our discrete Morse flow given by fans, by using fences.

\begin{definition}[Collapsibles and redundants of type 1]
  \label{defi:collapsible_1}
  A cell $\Sigma \in \SD_g^m$ or $\SD_{g,m}$ is \emph{collapsible of type $1$} if the following holds:
  \begin{enumerate}
   \item It is neither collapsible nor redundant of type $0$;
   \item The cell $\Sigma$ has a surface $S$ ending with an odd fence;
   \item Every other surface $\tilde S$ of $\Sigma$ with $ft(\tilde S) \neq 0$ and $end(\tilde S) > end(S) > 0$ has genus zero.
  \end{enumerate}
  Among the surface of $\Sigma$ fulfilling the above conditions we pick $S$ with maximal endpoint i.e.\ $i = end(S)$ maximal.
  We say say that \emph{$S$ witnesses the collapsiblity of $\Sigma$}.
  In this case, $d_{i-1}(\Sigma)$ is the \emph{redundant partner of type $1$}.
\end{definition}

\begin{remark}
  \label{remark:suspended_cells_and_collapsibility_B}
    Note that in the definition of a fence we are only interested in surfaces that are not attached to $0$.
  This restriction ensures that the Morse flow behaves well with the stabilization map, which increases the genus of the surface with foot-point $0$.
  Therefore, in contrast to Remark \ref{remark:suspended_cells_and_collapsibility_A}, if a suspended cell $\Sigma = (\lambda, S_1, \ldots, S_k)$ is collapsible of type $1$, its face $d_0(\Sigma)$ might not be collapsible of type $1$.
\end{remark}

The definition above extends our matching of type 0 in the unenumerated case.
Furthermore, after defining the matching of type 0 in the enumerated case, the description above will also extend to this setting.
This will give the following lemma.
The proof of this result as well as the Morse matching in the enumerated case, both of which are technical in nature, are deferred to Section \ref{section:proof_AB}.

\begin{lemma}[Lemma B]
  \label{lemma_b_unenumerated}
  The collapsibles and their redundant partners of type $1$ extend the discrete Morse flow (of type $0$) on $\SD_g^m$, $\SD_{g,m}$, $\betterwidetilde\SD_g^m$ and $\betterwidetilde\SD_{g,m}$.
  Every cell $\Sigma$ is redundant or collapsible if it has punctures or degenerate boundary or if it has a surface $S$ of positive genus $g_S > 0$ not attached at $0$ i.e., with  $ft(S) > 0$.
  
  Moreover, the stabilization map respects this discrete Morse flow, i.e.\ (1) it sends essentials to essentials, redundants to redundants and collapsibles to collapsibles and
  (2) a redundant or collapsible cell is in the sub-complex if and only if its partner is.
\end{lemma}

With this lemma we proceed to proof our stability result.

\begin{proof}[Proof of Theorem \ref{theorem_b}]
  By $\SD(g,m)$ we denote one of $\SD_g^m$, $\SD_{g,m}$, $\betterwidetilde\SD_g^m$ or $\betterwidetilde\SD_{g,m}$.
  The stabilization map $\SD(g,m) \to \SD(g+1,m)$ is seen as the inclusion of a sub-complex by Lemma \ref{lemma:stabilization_map_as_inclusion_of_a_sub-complex}.
  In particular a cell $\Sigma \in \SD(g+1,m)$ is in the sub-complex $\SD(g,m)$ if and only if the surface $S$ with $ft(S) = 0$ has positive genus.
  
  Using the same argument as in the proof of of Theorem \ref{theorem_a}, we see that the essential cells not in $\SD(g,m)$ are of degree at least $m+g-1$.
  In particular, the Morse complex of the quotient $\SD(g+1,m) / \SD(g,m)$ has no cells in degree below $m+g-1$.
  By Theorem \ref{theorem_a} and Lemma \ref{lemma:1_connectedness}, the spaces in question are at least $1$-connected if $m > 2$.
  Applying the relative Hurewicz theorem yields Theorem \ref{theorem_b}.
\end{proof}

\section{Detecting non-trivial families}
\label{section:detecting_homology}
In this section, we present our methods to detect non-trivial homology classes.
We perform explicit computations in Subsection \ref{subsection:explicit_computations}.
In Subsection \ref{subsection:detecting_homology_via_string_topology}, we use String topology to interpret cycles in the space of Sullivan diagrams as operations on the Hochschild homology of Frobenius algebras.
It follows that our cycles are non-trivial because the associated operations are non-trivial.
Having non-trivial classes at hand, we cook up new non-trivial classes using PROP-compositions and transfers in Subsection \ref{subsection:detecting_homology_recipe}.

\subsection{Explicit computations}
\label{subsection:explicit_computations}
We first compute the fundamental group in a special case.

\begin{proposition}
  \label{proposition:fundamental_group_two_punctures}
  Consider the space of Sullivan diagrams $\SDs_g^2$.
  \begin{enumerate}[label={(\roman*)}]
    \item \label{prop_c:i}
      The fundamental group of $\SDs_g^2$ is
      \begin{align*}
        \pi_1( \SDs_g^2 ) \cong 
        \begin{cases}
          \mathbb Z\langle \alpha_0 \rangle & g=0 \\
          \mathbb Z / 2\mathbb Z \langle \alpha_g \rangle  & g > 0
        \end{cases}
      \end{align*}
      with $\alpha_g$ shown in Figure \ref{fig:generator_m_2}.
      The map induced by the stabilization map on fundamental groups sends $\alpha_g$ to $\alpha_{g+1}$.
    \item \label{prop_c:iii}
      Furthermore, $\alpha_g$ is in the image of the homomorphism induced by the map of spaces.
      \[
        B\Mod(S_{g,1}^2) \to \SDs_g^2
      \]
      The source is the classifying space of the mapping class group of the surface $S_{g,1}^2$ with one boundary component, genus $g$ and two punctures;
      and the target is the space of Sullivan diagrams $\SDs_g^2$.
      After taking fundamental groups, the generator $\alpha_g$ is in the image of the Dehn (half-)twist that exchanges the two punctures inside a small disk.
  \end{enumerate}
\end{proposition}

\begin{figure}[ht]
    \inkpic[.5\columnwidth]{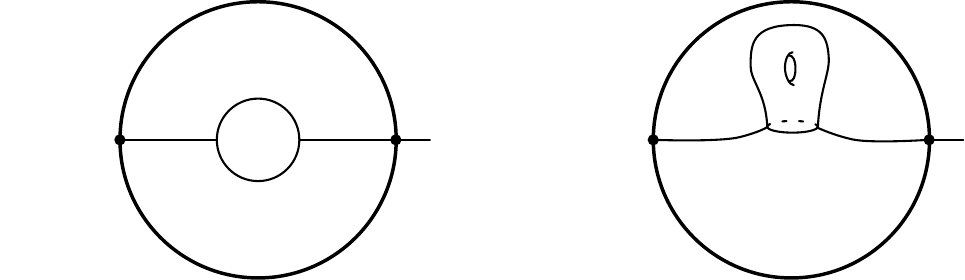}
    \caption{\label{fig:generator_m_2}The generators of $\pi_1(\SDs_g^2)$.}
\end{figure}

\begin{remark}
  \label{remark:vanishing_result_ww}
   In \cite{wahlwesterland} the authors describe a quotient map
   \[
    C_*(B\Mod(S_{g,p+q}^m)) \simeq \ph BW_{\mathrm{graphs}}(S_{g,p+q}^m)\fib \ph\SD(S_{g,p+q}^m),
  \]
  see Subsection \ref{subsection:BW_graphs}.
  In Proposition 2.14, they show that if $q\geq 1$ all stable classes vanish when passing to the quotient.
  Our Proposition \ref{proposition:fundamental_group_two_punctures} shows their bound $q \ge 1$ is sharp.
\end{remark}

\begin{proof}[Proof of Proposition \ref{proposition:fundamental_group_two_punctures}]
  Part \ref{prop_c:i} is a direct computation using generators and relations.
  The one-cells of $\SDs_g^2$ are of type $\alpha$, $\beta$ and $\gamma$ as in the proof of Lemma \ref{lemma:1_connectedness} (see also Figure \ref{fig:generators_and_relations}).
  The cells of type $\beta$ and $\gamma$ are trivial for the same reason as in the proof of Lemma \ref{lemma:1_connectedness}.
  One can see that $\pi_1(\SDs_0^2) \cong \Z$ is generated by the cell $\alpha_0$ seen on the left in Figure \ref{fig:generator_m_2}.
  Similarly, for $g > 0$, the fundamental group $\pi_1(\SDs_g^2) \cong \Z/2\Z$ is generated by the cell $\alpha_g$ seen on the right in Figure \ref{fig:generator_m_2}.
  However, there is the additional relation of the form $\alpha \beta \alpha$ and $\beta$ vanishes in $\pi_1(\SDs_g^2)$.
  By construction the stabilization map sends $\alpha_g \mapsto \alpha_{g+1}$.

  To see part \ref{prop_c:iii} observe first that $\pi_1(B\Mod(S_{0,1}^2)) = \Mod(S_{0,1}^2)$ is the second braid group.
  In particular, it is an infinite cyclic group generated the Dehn (half)twist that exchanges the two punctures.
  Now consider the quotient map described in Subsection \ref{subsection:BW_graphs}
  \[
    C_*(B\Mod(S_{0,1}^2))\simeq \ensuremath{1\hspace{1pt}\mhyphen}BW_{\mathrm{graphs}}(S_{g,1}^2)\fib \SD_{g,1}^2
  \]
  where the left hand side is Costello's fat graph model of the mapping class group and the quotient map is described in \cite[Theorem 2.9]{wahlwesterland}.
  One can see that the class $\alpha_g$ is hit by the class $\Lambda_g$ in black and white graphs shown in Figure \ref{fig:class_Lambda_g}.
  The class $\Lambda_g$ represents the element in $\Mod(S_{g,1}^2)$ that exchanges the two particles by a Dehn twist.
  \begin{figure}[ht]
    \inkpic[.3\columnwidth]{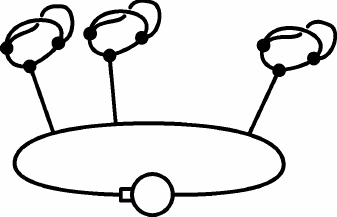}
    \caption{\label{fig:class_Lambda_g}The class $\Lambda_g \in BW_{\mathrm{graphs}}(S_{g,1}^2)$.}
  \end{figure}
  Another way to see part \ref{prop_c:iii} uses B{\"o}digheimer's radial model \cite{bodigheimer}.
  Here, the homology class in question is represented by a pair of antipodal radial slits in an annulus that exchange their positions using a rotation by the angle $\pi$.
\end{proof}

In contrast to our vanishing results, the homology of $\SDs$ is highly non-trivial.
Indeed, for small genus and number of boundary components we computed the homology of $\SDs_g^m$ and $\SDs_{g,m}$ by hand and with the help of a computer program.
See Appendix \ref{appendix:computations}.
In particular, for $\SDs_0^m$ we found that in these cases the first non-vanishing homology group is free abelian of rank one and it sits in degree $m'+1$ with $m'$ the largest even integer strictly smaller than $m$.
The Morse flow shows this is true in general for any $m$.

\begin{definition}
  \label{definition:class_zeta_m}
  For $m > 0$ let $\lambda^{\zeta}= (0\ 1\ \ldots\ m-1)$ and denote the unique cycle of $\lambda^{\zeta}$ by $\Lambda^\zeta$.
  We define $\zeta^m:=(\lambda^{\zeta},S_1)$ to be the Sullivan diagram where we attach a single surface $S_1 = (0,0, \Lambda^\zeta)$.
  Similarly, let $\lambda^\eta = (0)(1\ 2\ \ldots\ m)$ and denote the cycle corresponding to the fixed point $0$ by $\Lambda^0$ and the other cycle by $\Lambda^1$.
  We define $\eta^m:=(\lambda^{\eta},S_1,S_2)$ to be the Sullivan diagram where we attach the suspension disk $S_1 = (0,0, \Lambda^0)$ and $S_2 = (0,0, \Lambda^1)$.
  See Figure \ref{fig:class_zeta_m} for an example.
  \begin{figure}[ht]
    \inkpic[.8\columnwidth]{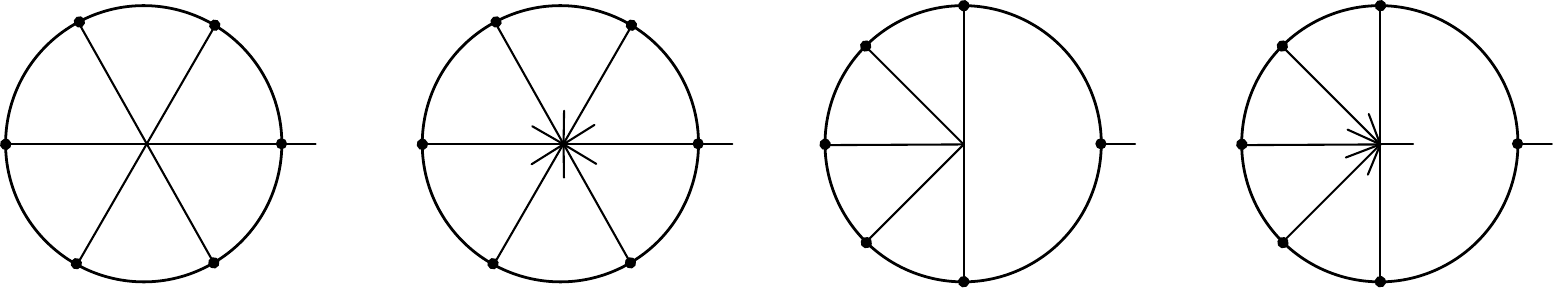}
    \caption{\label{fig:class_zeta_m}From left to right: the Sullivan diagrams $\zeta^6$, $\zeta_6$, $\eta^5$ and $\eta_5$.}
  \end{figure}
  
  In the parametrized unenumerated case, we obtain similar cells $\zeta_m$ and $\eta_m$ by putting $m$ leaves to the surface $S_1$ and $S_2$ respectively.
  More precisely, consider the permutation $\lambda_\zeta = (0\ l_1\ 1\ l_2\ \ldots\ m-1\ l_m)$ and let $\Lambda$ denote its unique cycle.
  We define $\zeta_m:=(\lambda_{\zeta},S_1)$ to be the Sullivan diagram where we attach a single surface $S_1 = (0,0, \Lambda_\zeta)$.
  Similarly consider the permutation $\lambda_\eta = (0)(l_1\ 1\ l_2\ 2\ l_3\ \ldots\ m)$ and denote by $\Lambda_0$ the cycle corresponding to the fixed point $0$ and by $\Lambda_1$ the other cycle.
  We define $\eta_m:=(\lambda_{\eta},S_1,S_2)$ to be the Sullivan diagram where we attach the suspension disk $S_1 = (0,0, \Lambda_0)$ and $S_2 = (0,0, \Lambda_1)$.
  
  Note in particular that $\zeta_m$ and $\zeta^m$ are cycles if and only if $m$ is even and $\eta^m$ and $\eta_m$ are cycles if and only if $m$ is odd.
\end{definition}

\begin{proposition}
  In the unenumerated case, the first non-vanishing reduced homology groups are as follows.
  For $m$ even they are given by
  \[
      H_{m-1}( \SDs_0^m; \Z ) \cong \mathbb \Z \langle \zeta^m \rangle \mspc{and}{20} H_{m-1}( \SDs_{0,m}; \Z ) \cong \mathbb \Z\langle \zeta_m \rangle
  \]
and for $m$ odd they are given by
  \[
      H_{m}( \SDs_0^m; \Z ) \cong \mathbb \Z \supset \Z \langle \eta^m \rangle \mspc{and}{20} H_{m}( \SDs_{0,m}; \Z ) \cong \mathbb \Z \supset \Z\langle \eta_m \rangle  \,.
  \]
\end{proposition}

\begin{proof}
  By Theorem \ref{theorem_a}, the spaces $\SDs_0^m$ and $\SDs_{0,m}$ are simply connected.
  In our Morse flow, cells with punctures or degenerate boundary are either redundant or collapsible.
  For even $m$, it is easy to see that $\zeta^m$ respectively $\zeta_m$ is the unique essential cell in degree $m-1$.
  After listing all essential cells of degree $m$, a straightforward calculation implies that
  $\zeta^m$ respectively $\zeta_m$ generate the homology of the Morse complex in this degree, which is free abelian of rank one.
  In this particular case, the inclusion of the Morse complex into $\SDs_0^m$ respectively $\SDs_{0,m}$ sends $\zeta^m$ to $\zeta^m$ respectively $\zeta_m$ to $\zeta_m$.
  
  For odd $m$, there are several essential cells in degree $m$ and $m+1$.
  However, one can work out, that the homology of the Morse complex in this degree is free abelian of rank one.
  Moreover, the class $\eta^m$ respectively $\eta_m$ are not boundaries in the Morse complex and
  the inclusion of the Morse complex into $\SDs_0^m$ respectively $\SDs_{0,m}$ sends $\eta^m$ to $\eta^m$ respectively $\eta_m$ to $\eta_m$.
\end{proof}

\subsection{String Topology}
\label{subsection:detecting_homology_via_string_topology}
We now use string topology to further detect non-trivial classes in the homology of Sullivan diagrams following the construction of \cite{wahluniversal}.
We give a brief sketch of this idea.
For any Frobenius algebra $A$ we study operations of the form 
\[
  CC_*(A,A)^{\otimes q} \longrightarrow CC_*(A,A)^{\otimes p}
\]
where $CC_*(A,A)$ denotes the reduced Hochschild chains of $A$.
Tradler and Zeinalian in \cite{TradlerZeinalian}, describe an action of a Sullivan diagram on the Hochschild Homology of any finite dimensional, unital Frobenius algebra $A$.
In \cite{wahlwesterland}, Wahl and Westerland give a recipe of how to read a Sullivan diagram as an operation on the reduced Hochschild chains of the algebra $A$.

All natural operations on the Hochschild homology of Frobenius algebras form a chain complex and in \cite{wahluniversal} Wahl introduces a chain complex $Nat(q,p)$ of formal operations,
which are an approximation of the chain complex of natural operations.
Let $\SD(q,p)$ denote the chain complex of Sullivan diagrams with $p$ admissible circles and $p+q$ enumerated leaves, exactly one in each boundary cycle.
Wahl shows that there is an inclusion 
\[
 \SD(q,p)\cof Nat(q,p)
\]
and this inclusion is a split quasi-isomorphism, see \cite[Theorem 2.9]{wahluniversal}.
Therefore, a cycle in $\betterwidetilde{\SD}_{g,m} \subset \SD(m,1)$  that induces a non-trivial operation $HH_\ast(A,A)^{\otimes m} \to HH_\ast(A,A)^{\otimes 1}$ for some $A$ is a non-trivial class in $H_*( \betterwidetilde{\SD}_{g,m}; \mathbb Z)$.

Using this approach, we construct two infinite families of generators of $ \betterwidetilde{\SD}_{g,m}$.
We show these are cycles by direct computation and then prove they represent non-trivial homology classes by showing that they induce non-trivial operations in the  reduced Hochschild chains of a certain Frobenius algebra $A$.
We now construct these generators.

\begin{definition}
  We begin with the basic building blocks.
  See Figure \ref{figure:building_blocks} for examples for small $m$.
  \begin{enumerate}
  \item For $m>0$ consider the permutation $\widetilde{\lambda}_\zeta = (0\ l_1\ 1\ l_2\ \ldots\ l_{m-1}\ m-1\ l_m)$ and let $\Lambda_\zeta$ denote its unique cycle.
    We define the chain $\widetilde{\zeta}_m:=(\widetilde{\lambda}_{\zeta},S_1)\in \widetilde{\SD}_{0,m}$ to be the Sullivan diagram where we attach a single surface $S_1 = (0,0, \Lambda_\zeta)$.
  \item
    For $m>0$ we define the chain $\widetilde{\mu}_m:=(\lambda,S_0,S_1,\ldots, S_m) \in \widetilde{\SD}_{0,m}$ to be given by the following data:
    \begin{align*}
      \lambda_0 &= (0\ 2\ 4\ \dots \ 2m-2) \\
      \lambda_i &= (2i-1\ l_i) \text{ for } 1 \le i \le m \\
      \lambda &= \lambda_0 \lambda_1 \cdots \lambda_m \\
      S_i & =(0,0,\lambda_i) \text{ for } 0 \leq i \leq m
    \end{align*}
  \item
    For $m>1$ we define the chain $\widetilde{\omega}_m:=\widetilde{\omega}_{m,1}-\widetilde{\omega}_{m,2}\in \widetilde{\SD}_{0,m}$, where 
    $\widetilde{\omega}_{m,1}=(\lambda_1,S_{0,1},S_{1,1},S_2,\ldots, S_{m})$ and 
    $\widetilde{\omega}_{m,2}=(\lambda_2,S_{0,2},S_{1,2},S_2,\ldots, S_{m})$ 
    are given by the following data:
    \begin{align*}
      \lambda_{0,1} &=(0) \text{ and } \lambda_{0,2} =(0\ l_1)\\
      \lambda_{1,1} &=(1\ 3\ \ldots \ 2m-1 \ l_1) \text{ and } \lambda_{1,2} =(1\ 3\ \ldots \ 2m-1) \\
      \lambda_i &=( 2i - 2\ l_i) \text{ for } 2 \le i \le m \\
      \lambda_1 &= \lambda_{0,1} \lambda_{1,1} \lambda_2 \cdots \lambda_m \\
      \lambda_2 &= \lambda_{0,2} \lambda_{1,2} \lambda_2 \cdots \lambda_m \\
      S_{i,j} &= (0,0,\lambda_{i,j}) \text{ for } i =0,1 \text{ and } j = 1,2 \\
      S_i &= (0,0,\lambda_i) \text{ for } 2\leq i \leq m
    \end{align*}
  \item 
    We define the chain $\widetilde{\gamma}:=\widetilde{\gamma}_1+\widetilde{\gamma}_2 -\widetilde{\gamma}_3\in \widetilde{\SD}_{1,1}$ where for $1\leq i \leq 3$, 
    $\widetilde{\gamma}_i=(\lambda_i,S_{1,i}, S_{2,i})$ and these are given by the following data 
    \begin{align*}
      \lambda_{1,j} &= (0) \text{ for } j=1,2\\
      \lambda_{1,3} &= (0 \ l_1) \\
      \lambda_{2,1} &= (l_1 \ 1 \ 3 \ 2)\\
      \lambda_{2,2} &= (1 \ 3 \ l_1 \ 2)\\
      \lambda_{2,3} &= (1 \ 3 \ 2)\\
      \lambda_{j} &= \lambda_{1,j} \lambda_{2,j} \text{ for } j=1,2,3 \\
      S_{i,j} &= (0,0,\lambda_{i,j}) \text{ for } i = 1,2 \text { and } j=1,2,3
    \end{align*}
  \end{enumerate} 
\end{definition}

\begin{figure}[ht]
  \centering
  \inkpic[.8\columnwidth]{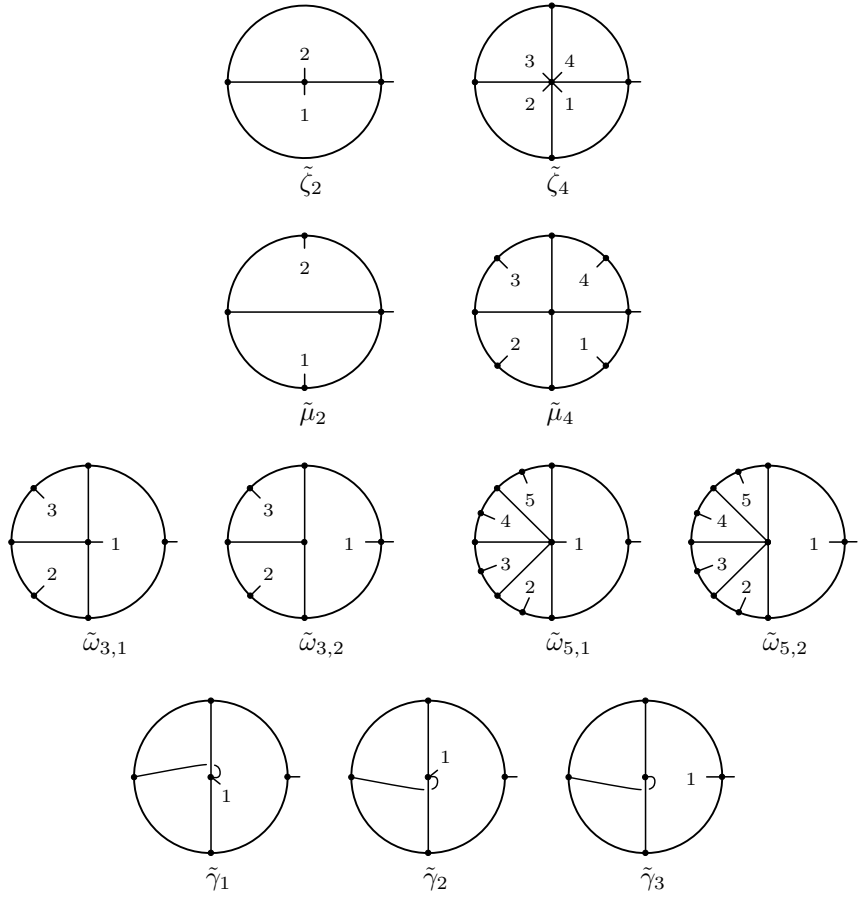}
  \caption{The building blocks.}
  \label{figure:building_blocks}
\end{figure}

\begin{lemma}
  \label{lem:cycles_building_blocks}
  For any $m>0$ the chains $\widetilde{\mu}_m, \ \widetilde{\omega}_{m}$ and $\widetilde{\gamma}$ are cycles in the chain complex of Sullivan diagrams.
  The chain $\widetilde{\zeta}_m$ is not a cycle.
\end{lemma}

\begin{proof}
  A short computation shows that
  \begin{align*}
    d_i(\widetilde{\mu}_m) & =d_{i+1}(\widetilde{\mu}_m) \text{ for all } i \text{ even}\\
    d_i(\widetilde{\omega}_{m,j}) & =d_{i+1}(\widetilde{\omega}_{m,j}) \text{ for } j=1,2 \text{ and } 1\leq i \leq 2m-3 \text{ odd}\\
    d_i(\widetilde{\omega}_{m,1}) & =d_{i}(\widetilde{\omega}_{m,2}) \text{ for } i=0,2m-1\\ 
    d_1(\widetilde{\gamma}_{i}) & =d_{2}(\widetilde{\gamma}_{i}) \text{ for } i=1,2,3\\
    d_0(\widetilde{\gamma}_{1}) & =d_{0}(\widetilde{\gamma}_{3})\\
    d_3(\widetilde{\gamma}_{2}) & =d_{3}(\widetilde{\gamma}_{3})\\
    d_0(\widetilde{\gamma}_{2}) & =d_{3}(\widetilde{\gamma}_{1}).
  \end{align*}
  The result follows.
  To see that $\widetilde{\zeta}_m$ is not a cycle, notice that for $i\neq j$ we have that $d_i(\widetilde{\zeta}_m)\neq d_j(\widetilde{\zeta}_m)$.
\end{proof}

Using these building blocks we now construct other chains which we describe using the PROP composition as described in Definition \ref{def:composition}.

\begin{definition}
  Let $m > 0 $ and let $(c_1, \ldots , c_m)$ be a tuple of integers with $c_i>1$ for $1\leq i \leq m$ and set $c:=\sum_{i=1}^m c_i$.
  We define the chain $\widetilde{\Omega}_{(c_1, \ldots, c_m)}\in \widetilde{\SD}_{0,c}$ to be
  \[\widetilde{\Omega}_{(c_1, \ldots, c_m)}:=\widetilde{\zeta}_m\circ
    (\widetilde{\omega}_{c_1}\otimes
    \widetilde{\omega}_{c_2}\otimes
    \ldots \otimes
    \widetilde{\omega}_{c_m}).
  \] 
  Note in particular that $\widetilde{\Omega}_{(c)} = \widetilde{\zeta}_1\circ \widetilde{\omega}_c = \widetilde{\omega}_c$.
  See Figure \ref{figure:generator_Omega} for another example.
  
  Similarly, for $m> 0$ we define the chain $\widetilde{\Gamma}_{m}\in \widetilde{\SD}_{m,m}$ to be
  \[\widetilde{\Gamma}_{m}:=\widetilde{\zeta}_m\circ
    (\widetilde{\gamma}^{\otimes m}).
  \]
  Note in particular that $\widetilde{\Gamma}_1=\widetilde{\gamma}$.
\end{definition}

\begin{figure}[ht]
  \centering
  \inkpic[.8\columnwidth]{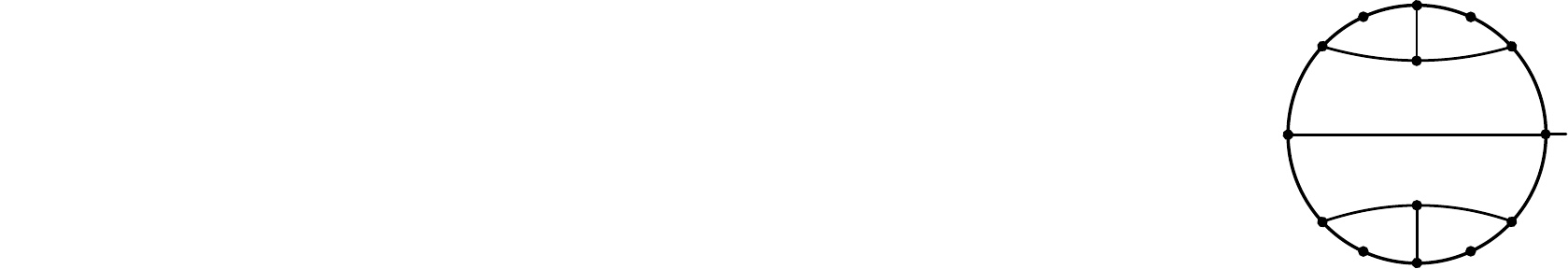}
  \caption{\label{figure:generator_Omega}The homology class $\widetilde{\Omega}_{(3,3)}$.}
\end{figure}

\begin{lemma}
  \label{lem:cycles_composed_classes}
  For any $m>0$ and any sequence of integers $(c_1, \ldots , c_m)$ with $c_i>1$ for $1\leq i \leq m$, the chains $\widetilde{\Omega}_{(c_1,\ldots,c_m)}$ and $\widetilde{\Gamma}_m$ are cycles in the chain complex of Sullivan diagrams.
\end{lemma}

\begin{proof}
  This follows from noticing the boundaries of these generators behave locally as for the basic building blocks.
  We make this precise for $\widetilde{\Gamma}_m$.
  Let $1\leq j \leq m$, $1\leq k \leq 3$ and let 
  \[
    \widetilde{\gamma}^{(j, k)}:=\sum_{1\leq i_s\leq 3} (-1)^\sigma
    \widetilde{\gamma}_{i_1}\otimes \widetilde{\gamma}_{i_2}\otimes \ldots
    \otimes \widetilde{\gamma}_{i_{j-1}}\otimes \widetilde{\gamma}_{k}\otimes \widetilde{\gamma}_{i_{j+1}}\otimes \ldots
    \otimes \widetilde{\gamma}_{i_m}
  \]
  and 
  \[
    \widetilde{\Gamma}_m^{(j,k)}:=\widetilde{\zeta}_m\circ
    (\widetilde{\gamma}^{(j,k)})
  \]
  where 
  \[
    \sigma (i_1, i_2, \ldots, i_m) = \#\{i_s=3 \mid 1\leq s \leq m, s \neq k\}.
  \]
  Notice that for any $1\leq j \leq m$
  \[
    \widetilde{\gamma}^{\otimes m}= \widetilde{\gamma}^{(j, 1)}+\widetilde{\gamma}^{(j, 2)}-\widetilde{\gamma}^{(j, 3)}
  \]
  and thus 
  \[\widetilde{\Gamma}_m=\widetilde{\Gamma}_m^{(j, 1)}+\widetilde{\Gamma}_m^{(j, 2)}-\widetilde{\Gamma}_m^{(j, 3)}.\]
  Now let $l=4(j-1)$ and note that
  \begin{align*}
    d_{l+1}(\widetilde{\Gamma}_m^{(j,k)}) & =d_{l+2}(\widetilde{\Gamma}_m^{(j,k)}) \text{ for } k=1,2,3\\
    d_{l}(\widetilde{\Gamma}_m^{(j,1)}) & =d_{l}(\widetilde{\Gamma}_m^{(j,3)})\\
    d_{l+3}(\widetilde{\Gamma}_m^{(j,2)}) & =d_{l+3}(\widetilde{\Gamma}_m^{(j,3)})\\
    d_l(\widetilde{\Gamma}_m^{(j,2)}) & =d_{l+3}(\widetilde{\Gamma}_m^{(j,1)})
  \end{align*}
  which shows that $\widetilde{\Gamma}_m$ is a cycle.
  The case of $\widetilde{\Omega}_{(c_1,\ldots,c_m)}$ follows similarly.
\end{proof}

In \cite[Section 4.2]{wahluniversal}, Wahl shows that for $m>0$ the $\widetilde{\mu}_m$'s are non-trivial classes of the homology of $\SD$ by showing that these cycles induce non-trivial operations.
We extend these results to all the infinite families of cycles constructed above.

\begin{remark}
  \label{remark:relating_mu_and_Omega}
  The classes found by Wahl $\widetilde{\mu}_m$ and the classes $\widetilde{\Omega}_{(c_1, \ldots, c_k)}$ with $\sum_{i=1}^k c_i = m$ have the same degree and topological type.
  A simple argument shows that $\widetilde{\mu}_m-\widetilde{\omega}_{m}$ is a boundary and thus
  $\widetilde{\mu}_m$ and $\widetilde{\omega}_{m}=\widetilde{\Omega}_{(m)}$ are homologous.
  We show this graphically in Figure \ref{homologous}.
  The general case, follows in exactly the same way.
  \begin{figure}[ht]
    \centering
    \inkpic[.7\columnwidth]{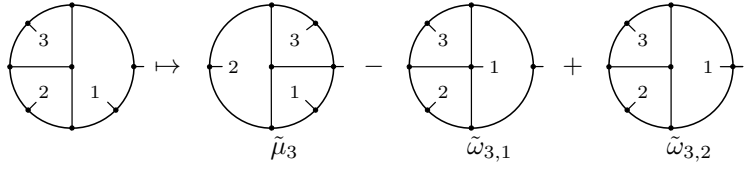}
    \caption{A Sullivan diagram whose boundary is $\tilde\mu_3 - \tilde\omega_3$}
    \label{homologous}
  \end{figure}
  
  However, if $(c_1, \ldots, c_k)$ is a sequence of length two or more, we do not know if $\widetilde{\mu}_m$ and $\widetilde{\Omega}_{(c_1, \ldots, c_k)}$ are homologous.
  Furthermore, the classes $\widetilde{\Gamma}_m$ are certainly not homologous to the ones found by Wahl, since they are of a different topological type.
\end{remark}

\begin{proposition}
  \label{proposition:non_trivial_op}
  For any $m>0$ and any sequence of integers $(c_1, \ldots , c_m)$ with $c_i>1$ for all $1\leq i \leq m$ the cycles $ \widetilde{\Omega}_{(c_1, \ldots , c_m)}$ and $\widetilde{\Gamma}_m$ represent
  non-trivial classes of infinite order in the homology of the complex of Sullivan diagrams.
\end{proposition}

We will prove the proposition by considering the Frobenius algebra $A=\Z[x]/(x^2)$ with $\vert x \vert = 1$, where the coproduct is given by $\nu(1)=1\otimes x + x\otimes 1$ and $\nu(x)=x\otimes x$.
For the convenience of the reader, let us recall the Hochschild homology of $A$ for $n>0$.
\begin{align}
  \label{HH_A_over_Z}
  HH_n( A, A ) \cong
  \begin{cases}
    \Z \langle x^{\otimes n+1} \rangle & \text{$n$ even}\\
    \Z \langle 1 \otimes x^{\otimes n} \rangle \oplus \Z/2\Z \langle x^{\otimes n+1} \rangle & \text{$n$ odd}
  \end{cases}
\end{align}

\begin{remark}
  The Frobenius algebra $A$ is actually the cohomology algebra of $S^1$ i.e.\ $A:=H^*(S^1)$.
  Moreover, up to a degree shift and signs it is also the cohomology algebra of $S^n$ for $n\geq 2$.
  Since signs and degrees do not play a role in the proof of Proposition \ref{proposition:non_trivial_op}, the same argument of \cite[p.~29]{wahluniversal} shows that
  $\tilde\Omega_{(c_1, \ldots , c_m)}$ and $\tilde\Gamma_m$ give non-trivial string operations on  $H_*(LS^n)$, where $LS^n$ is the free loop space of $S^n$.
\end{remark}

\begin{proof}[Proof of Proposition \ref{proposition:non_trivial_op}]
These chains are cycles by Lemma \ref{lem:cycles_composed_classes}.
  In order to see these cycles represent non-trivial classes in homology, we read them as operations on the Hochschild homology of $A$.
  Let $\xi\in \widetilde{\SD}_{g,m}$ denote any of these cycles.
  Following the recipe given in \cite[Section 6.2]{wahlwesterland}, the cycle $\xi$ induces an operation
  \[
    \xi_{*}:HH_*(A)^{\otimes m} \to HH_{*+|\xi|}(A)
  \]
  and we test it on the non-trivial element $x^{\otimes m} \in HH_0(A)^{\otimes m}$.
  We briefly describe how to graphically compute $\xi_{*}(x^{\otimes m})$.
  For a general description and further details we refer the reader to \cite{wahluniversal, wahlwesterland}.

  Assume first that $\xi$ is a single Sullivan diagram.
  Choose an essentially trivalent representative of $\xi$ and remove the edges from the admissible cycle and its admissible leaf.
  This gives a trivalent fat graph $\Gamma_\xi$ with $m$ enumerated leaves which we think of as ``incoming" and $1 + |\xi|$ enumerated vertices of valence one which we think of as ``outgoing".
  The enumerated vertices correspond to the vertices in the admissible cycle in the cyclic order in which they occur starting with the vertex to which the admissible leaf was attached.
  See Figure \ref{figure:example_map_to_nat} for an example.
  \begin{figure}[ht]
    \begin{center}
      \begin{minipage}{1.5cm}
        \centering
        $(a)$
      \end{minipage}
      \begin{minipage}{.15\columnwidth}
        \centering
        \inkpic[\columnwidth]{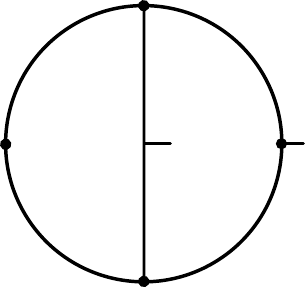}
      \end{minipage}
      \begin{minipage}{2cm}
        \centering
        $\leadsto$
      \end{minipage}
      \begin{minipage}{.20\columnwidth}
      \tikz[baseline={([yshift=-3.5pt]current bounding box.center)}, x=1ex, y=1ex]
      {
        \draw[line width=.8pt] (0,0) to (6,0) to (8,2) to (10,2) node[shape=circle, fill=black, inner sep=1pt] {};
        \draw[line width=2pt, color=white]                (6,2) to (8,0);
        \draw[line width=.8pt] (0,3) to (2,3) to (4,2) to (6,2) to (8,0) to (10,0) node[shape=circle, fill=black, inner sep=1pt] {};
        \draw[line width=.8pt]          (2,3) to (4,4) to (10,4) node[shape=circle, fill=black, inner sep=1pt] {};
        \draw node[shape=circle, fill=black, inner sep=1pt] at (10,6) {};
        \draw node at (-1,0) {$2$};
        \draw node at (-1,3) {$1$};
        \draw node at (12,0) {$e_3$};
        \draw node at (12,2) {$e_2$};
        \draw node at (12,4) {$e_1$};
        \draw node at (12,6) {$e_0$};
      }
      \end{minipage}
    \end{center}
    
    \begin{center}
      \begin{minipage}{1.5cm}
        \centering
        $(b)$
      \end{minipage}
      \begin{minipage}{.15\columnwidth}
        \centering
        \inkpic[\columnwidth]{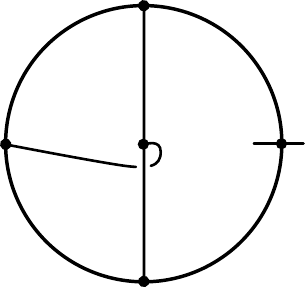}
      \end{minipage}
      \begin{minipage}{2cm}
        \centering
        $\leadsto$
      \end{minipage}
      \begin{minipage}{.20\columnwidth}
        \tikz[baseline={([yshift=-3.5pt]current bounding box.center)}, x=1ex, y=1ex]
        {
          \draw[line width=.8pt] (2,3) to[out=180, in=90] (0,2) to[out=270, in=180] (2,1);
          \draw[line width=.8pt] (2,1) to (4,1) to (6,0) to (8,0) to (10,2) to (12,2) node[shape=circle, fill=black, inner sep=1pt] {};
          \draw[line width=2pt, color=white] (6,2) to (8,2) to (10,0) to (12,0) node[shape=circle, fill=black, inner sep=1pt] {};
          \draw[line width=.8pt]          (4,3) to (6,2) to (8,2) to (10,0) to (12,0) node[shape=circle, fill=black, inner sep=1pt] {};
          \draw[line width=.8pt] (2,3) to (4,3) to (6,4) to (12,4) node[shape=circle, fill=black, inner sep=1pt] {};
          \draw[line width=.8pt] (0,6) to (12,6) node[shape=circle, fill=black, inner sep=1pt] {};
          \draw node[shape=circle, fill=black, inner sep=1pt] at (12,6) {};
          \draw node at (-1,6) {$1$};
          \draw node at (14,0) {$e_3$};
          \draw node at (14,2) {$e_2$};
          \draw node at (14,4) {$e_1$};
          \draw node at (14,6) {$e_0$};
        }
      \end{minipage}
    \end{center}

    \caption{\label{figure:example_map_to_nat}
      Two Sullivan diagrams and their corresponding trivalent graphs.
      Diagram (a) induces the operation:
      $x\otimes x \mapsto 1\otimes x\otimes x \otimes x \neq 0.$
      Diagram (b) induces the operation:
      $x\mapsto  x \otimes 1\otimes x\otimes x + x \otimes x\otimes 1\otimes x + x \otimes x\otimes x \otimes 1=0$.
    }
  \end{figure}

  We place an $x$ in each of the incoming leaves and proceed to ``read" the graph as a composition of basic operations in $A$ as detailed in Figure \ref{figure:atoms_fat_graphs}.
  The result of $\xi_*(x^{\otimes m})$ is the tensor product of the entries at the $1+|\xi|$ outgoing vertices.
  See Figure \ref{figure:example_map_to_nat} for examples.
  \begin{figure}[ht]
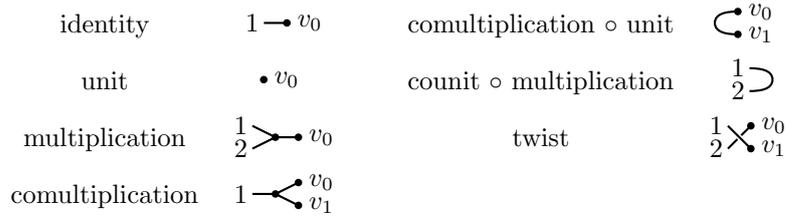

    \begin{tabular}{ccccc}
      identity & $\bbid$ & \hspace{1ex} & comultiplication $\circ$ unit & $\bbcomultunit$ \\
      unit & $\bbunit$ && counit $\circ$ multiplication & $\bbcounitmult$ \\
      multiplication & $\bbmult$ && twist & $\bbtwist$ \\
      comultiplication & $\bbcomult$ \\
    \end{tabular}
    \caption{\label{figure:atoms_fat_graphs}Graphical representation of basic operations of $A$.}
  \end{figure}
  
  If $\xi$ is not a single diagram but a sum of such, then we compute the operation for each entry and sum up the results.
  
  \medskip
  
Assume $\xi$ is a single Sullivan diagram.
From the above, one can proof the following facts:
  
  \begin{itemize}
  \item[(a)] We have that
  \[
    \xi_*(x^{\otimes m})=\sum y_1 \otimes y_2\otimes  \cdots \otimes y_k
  \]
  where $y_i\in \{1,x\}$, $k=m+|\xi|+1$ if $\xi$ is suspended and $k=m+|\xi|$ otherwise.
  
  \item[(b)] If $\xi$ is suspended, then
  \[
    \xi_*(x^{\otimes m})=\sum 1 \otimes y_2\otimes  \cdots \otimes y_k \,.
  \]
  
  \item[(c)] If $\xi$ is unsuspended and has a ghost surface $S$ that does not have a leaf in its boundary, then
  for each entry of the sum there is at least one $i>1$ for which $y_i=1$.
  From this it follows that
  \[
    \xi_*(x^{\otimes m}) = 0 \,.
  \]
    \end{itemize}

  Using these facts, we deduce that
  \[\widetilde{\Omega}_{(c_1,\ldots,c_m)*}(x^{\otimes c})=
    (\widetilde{\zeta}_m\circ
    (\widetilde{\omega}_{c_1,1}\otimes 
    \widetilde{\omega}_{c_2,1}\otimes
    \ldots \otimes
    \widetilde{\omega}_{c_m,1}))_*(x^{\otimes c})=1\otimes x^{\otimes 2c -1}\neq 0
  \]
  where the last computation is obtained by following the recipe on the graph.
  Similarly, 
  \begin{align*}
    \widetilde{\Gamma}_{m*}(x^{\otimes m})&=
    \sum_{i_j=1,2}  
    (\widetilde{\zeta}_m\circ
    (\widetilde{\gamma}_{i_1}\otimes 
    \widetilde{\gamma}_{i_2}\otimes
    \ldots \otimes
    \widetilde{\gamma}_{i_m}))_*(x^{\otimes m})\\
    &=
    \sum_{\substack{i_j=1,2 \\ 1\leq j\leq m}}
    (1\otimes x^{\otimes 4m-1})\\
    &= 2^{m}(1\otimes x^{\otimes 4m-1})\neq 0
  \end{align*}   
  where the middle step is obtained by following the recipe on the graph.
  Therefore, since $\widetilde{\Omega}_{(c_1,\ldots,c_m)*}$ and $\widetilde{\Gamma}_{m*}$ are non-trivial operations, their corresponding chains are non-trivial classes.
  They are of infinite order since $(t\cdot \widetilde{\Omega}_{(c_1,\ldots,c_m)*})(x^{\otimes c}) = t \cdot (\widetilde{\Omega}_{(c_1,\ldots,c_m)*}(x^{\otimes c})) \neq 0$ and similarly for $\widetilde{\Gamma}_{m*}$.
\end{proof}

\subsection{More non-trivial families by transfer}
\label{subsection:detecting_homology_recipe}
On the level of spaces, we have the following diagram of forgetful maps.
\begin{align}
  \begin{tikzcd}[ampersand replacement=\&]
                                                                                        \& \betterwidetilde{\SDs}_g^m \arrow{dr}{\omega}  \& \\
      \betterwidetilde{\SDs}_{g,m} \arrow{ur}{\tilde\vartheta} \arrow{dr}{\hat\omega}   \&                                                \& \SDs_g^m \\
                                                                                        \& \SDs_{g,m} \arrow{ur}{\vartheta}               \&
    \end{tikzcd}
\end{align}
They are extensions of maps of moduli spaces.
The maps $\hat\omega$ and $\omega$, which forget the enumeration, are extensions of covering maps.
The maps $\tilde\vartheta$ and $\vartheta$ , which forget the $m$ leaves which are not the admissible leaf,  are extensions of fibrations.
Unfortunately, because the homotopy type of the fibers are not constant, these extensions are not coverings respectively fibrations.
See Remark \ref{remark:not_a_fibration} for details.
Nevertheless, using our Morse flow we show that the maps $\hat\omega$ and $\omega$ behave like coverings in homology.

Using our description of cells, given in Proposition \ref{proposition:combinatorial_1_SD}, and our discussion of the face map, see Lemma \ref{lem:rho_face} and Discussion \ref{disc:faces},
we assemble the corresponding forgetful maps of chain complexes.
It is a straight forward verification that forgetting the enumeration of the boundaries or punctures defines chain map versions of $\hat \omega$ and $\omega$.

Let us describe the forgetful map $\vartheta \colon \SD_{g,m} \to \SD_g^m$.
Consider the sub-complex $\SD_{g,m}' \subseteq \SD_{g,m}$ consisting of all cells $\Sigma$ that do not have a leaf on the admissible circle at a position other than $e_0$.
The map that forgets the leaves is clearly a chain map $\vartheta' \colon \SD_{g,m}' \to \SD_g^m$.
As a map of graded modules, let $\vartheta = \vartheta' \circ p$ with $p \colon \SD_{g,m} \to \SD_{g,m}'$ the orthogonal projection,
i.e.\ we define $p$ by setting $p(\Sigma) = \Sigma$ for a basis element $\Sigma \in \SD'_{g,m}$ and $p(\Sigma) = 0$ for the other basis elements.
It is left to show that $\vartheta$ commutes with the differential if $\vartheta(\Sigma) = 0$.
In this case, $\Sigma$ has a leaf on the admissible circle at a position $e_i \neq e_0$.
Now $\vartheta(d_j(\Sigma) ) = 0$ for $j \neq i, i-1$ because $d_j(\Sigma)$ has a leaf on the admissible circle not sitting at $e_0$.
Moreover, $\vartheta(d_i(\Sigma)) = \vartheta(d_{i-1}(\Sigma))$ so these two cells cancel each other in the boundary.
We showed $\vartheta(d(\Sigma)) = 0$ if $\vartheta(\Sigma) = 0$.

The forgetful map $\tilde \vartheta \colon \betterwidetilde\SD_{g,m} \to \betterwidetilde\SD_g^m$ is constructed and treated analogously.
We have proven the following proposition.

\begin{proposition}
  \label{proposition:forgetful_chainmaps}
  There are forgetful maps of chain complexes.
  \begin{align}
    \begin{tikzcd}[ampersand replacement=\&]
                                                                                        \& \betterwidetilde{\SD}_g^m \arrow{dr}{\omega} \& \\
      \betterwidetilde{\SD}_{g,m} \arrow{ur}{\tilde\vartheta} \arrow{dr}{\hat\omega}    \&                                              \& \SD_g^m \\
                                                                                        \& \SD_{g,m} \arrow{ur}{\vartheta}              \&
    \end{tikzcd}
  \end{align}
\end{proposition}

We want to think of $\hat\omega$ and $\omega$ as covering maps.
On the cells of top dimension, the maps behave like an $m!$-sheeted covering but on Sullivan diagrams with at least two degenerate boundary cycles the number of sheets drops.
Fortunately, by Proposition \ref{proposition:support_of_homology}, the homology is not supported on Sullivan diagrams of this type.
This allows us to show that, in homology, $\hat\omega_\ast$ and $\omega_\ast$ behave like covering maps.
To make this precise, consider the sub-complex $B$ of $\SD = \SD_{g,m}$ or $\SD_g^m$ of diagrams having at least two degenerate boundary cycles.
Denote its counterpart in $\betterwidetilde\SD = \betterwidetilde\SD_{g,m}$ or $\betterwidetilde\SD_g^m$ by $\betterwidetilde B$.
Clearly,the forgetful maps $\hat\omega$ and $\omega$ descent to the quotients
\[
  \hat P \colon \betterwidetilde{\SD}_{g,m} / \betterwidetilde{B}_{g,m} \to \SD_{g,m} / B_{g,m} \;\;\;\;\text{ or }\;\;\;\; P \colon \betterwidetilde{\SD}_g^m / \betterwidetilde{B}_g^m \to \SD_g^m / B_g^m \,.
\]
These are covering maps with $m!$ sheets:
Each cell $\Sigma$ in $\SD_{g,m} / B_{g,m}$ or $\SD_g^m / B_g^m$ has exactly $m!$ preimages under $\hat P$ and $P$.
Moreover, sending $\Sigma$ to the sum of its preimages defines a transfer map
\begin{align}
  \label{transfer_maps_on quotients}
  tr \colon \SD_{g,m} / B_{g,m} \to \betterwidetilde{\SD}_{g,m} / \betterwidetilde{B}_{g,m} \;\;\;\;\text{ and }\;\;\;\; tr \colon \SD_g^m / B_g^m \to \betterwidetilde{\SD}_g^m / \betterwidetilde{B}_g^m \,.
\end{align}

\begin{proposition}
  \label{proposition:transfer_maps}
  Let $\SD = \SD_{g,m}$ or $\SD_g^m$.
  In homology, the forgetful maps $\hat\omega$ respectively $\omega \colon \betterwidetilde\SD \to \SD$ behave like $m!$-sheeted covering maps, i.e.\  we have transfer maps
  \begin{align}
    \label{transfer_maps}
    tr \colon H_\ast(\SD_{g,m}; \Z ) \to H_\ast( \betterwidetilde{\SD}_{g,m}; \Z ) \;\;\;\;\text{ and }\;\;\;\; tr \colon H_\ast( \SD_g^m; \Z ) \to H_\ast( \betterwidetilde{\SD}_g^m; \Z )
  \end{align}
  and $\hat\omega_\ast \circ tr$ respectively $\omega_\ast \circ tr$ is the multiplication by $m!$.
\end{proposition}

\begin{proof}
  By Proposition \ref{proposition:support_of_homology}, the homology of $\SD$ injects into the homology of $\SD/B$.
  Therefore, the trace map \eqref{transfer_maps_on quotients} lifts.
  \[
    \begin{tikzcd}[ampersand replacement=\&]
      H_\ast( \betterwidetilde\SD; \Z ) \arrow[hook]{r}         \& H_\ast( \betterwidetilde\SD / \betterwidetilde B; \Z ) \\
      H_\ast( \SD; \Z ) \arrow[hook]{r} \arrow[dotted]{u}{tr}   \& H_\ast( \SD / B; \Z ) \arrow{u}{tr}
    \end{tikzcd}
  \]
\end{proof}

We use these transfers to construct more non-trivial families.

\begin{proposition}
  \label{proposition:generators_by_operations}
  For any $m > 0$ and any sequence of integers $(c_1, \ldots , c_m)$ with $c_i>1$ the classes $\hat\omega( \tilde\Omega_{(c_1, \ldots , c_m)} )$ and  $\hat\omega( \tilde\Gamma_m)$ are of infinite order.
  For $m$ even, the classes $\hat\omega( \tilde\zeta_m )$ are of infinite order.
\end{proposition}

\begin{proof}
  Let $\xi = \tilde\Omega_{(c_1, \ldots , c_m)}$ or $\tilde\Gamma_m$.
  Note that $tr \hat\omega(\xi) = \sum_{\sigma \in \Symm(L)} \sigma.\xi$ where $\Symm(L)$ acts on the cells by permuting the labels.
  Observe that in the proof of Proposition \ref{proposition:non_trivial_op} the non-trivial operation $\xi_*$ associated to $\xi$ and $\sigma.\xi$ agree.
  Therefore $tr \hat\omega(\xi)$ induce the operation $m! \cdot \xi_*$ which is also of infinite order.
  Thus, $\hat \omega(\xi)$ must be a non-trivial homology class of infinite order.
  
  Observe that $\hat\omega(\tilde\zeta_m)$ is a cycle (provided $m$ is even) and observe further that
  \[tr \hat\omega( \tilde \mu_m )=
  (tr \hat\omega( \tilde\zeta_m ))\circ (\tilde\mu_1 \otimes \ldots \otimes \tilde \mu_1)\]
  where $\circ$ denotes the PROP composition.
In \cite{wahluniversal}, Wahl shows that $\tilde\mu_m$ is non-trivial of infinite order.
  Alternatively, the same follows from Remark \ref{remark:relating_mu_and_Omega} stating that $\tilde\mu_m$ is homologous to $\Omega_{(m)}$ which is of infinite order by Proposition \ref{proposition:non_trivial_op}.
  This proofs the claim.
\end{proof}

\section{Proof of Lemmas A and B}
\label{section:proof_AB}
In this section, we prove our technical lemmas.
Let us treat the unenumerated case first.

\subsection{Proof of Lemma A (unenumerated case)}
\label{subsection:proofs:lemma_a}

We restate Lemma \ref{lemma_a_unenumerated} for the convenience of the reader.
\begin{lemmawithfixednumber}{\ref{lemma_a_unenumerated}}
  The collapsibles and their redundant partners define a discrete Morse flow on $\SD_g^m$ and $\SD_{g,m}$.
  A cell (of positive degree) with punctures or degenerate boundary is either collapsible or redundant.
\end{lemmawithfixednumber}

\begin{proof}[Proof of Lemma \ref{lemma_a_unenumerated}]
  Let $\SDgm$ be either $\SD_g^m$ or $\SD_{g,m}$.
  Our definitions of redundant and collapsible give a matching $F$ on the cellular complex of $\SDgm$.
  It is left to show that it is acyclic.
  For every cell $\Sigma$, we introduce the following notation.
  \begin{itemize}
    \item The number of attached surfaces is $s_\Sigma$.
    \item The number of punctures or degenerate boundary of a surface $S \in \Sigma$ is $m_S$ and
          the number of punctures or degenerate boundary of the cell $\Sigma$ is $m_\Sigma = \sum_{S \in \Sigma} m_S$.
    \item If a surface $S \in \Sigma$ starts with a fan of length $l$ we define $l_S = l$ else we set $l_S = 0$.
          The total length of fans of the cell $\Sigma$ is $l_\Sigma = \sum_{S \in \Sigma} l_S$.
  \end{itemize}
  Moreover, we assign to every cell $\Sigma$ a \emph{degree of degeneracy}
  \begin{align}
      \label{lemma_a:deg_of_degeneracy}
      \degen(\Sigma) = (\dim(\Sigma), s_\Sigma, m_\Sigma, -l_\Sigma) \,.
  \end{align}
  Now, we show that the lexicographical order of this sequence decreases along every path $c \ad r \au \tilde c$ in the $F$-inverted graph.
  Then $F$ is acyclic by Lemma \ref{dmt_acyclicity}.
  
  Consider a path $c \ad r \au \tilde c$ with $c$ and $\tilde c$ collapsible and $r$ redundant.
  Clearly $\dim(c) = \dim(\tilde c)$ and $s_{c} \ge s_{r} = s_{\tilde c}$.
  We assume that $s_{c} = s_{\tilde c}$ and study how the punctures or degenerate boundary change.
  The transition from $r$ to $\tilde c$ reduces $m_r$ by one and increases $l_r$ by one; in formulas: $m_r = m_{\tilde c} + 1$ and $l_r = l_{\tilde c} - 1$.
  The cell $r$ is a face of $c$ so we have $m_c \le m_r \le m_{\tilde c} + 1$ (by Discussion \ref{disc:faces}).
  If $m_c = m_r = m_{\tilde c} + 1$ we are done, because $m_c > m_{\tilde c}$.
  
  It remains to study the case $\dim(c) = \dim(\tilde c)$, $s_{c} = s_{\tilde c}$ and $m_{c} = m_{\tilde c}$.
  Denote the unique surface that witnesses the collapsiblity of $\tilde c$ by $\tilde S$.
  The redundant partner $r$ is obtained from $\tilde c$ by the face maps at $ft(\tilde S), \ldots, ft(\tilde S) + l_{\tilde S} - 1$.
  Therefore, we see that the surfaces of $r$ that are attached in front of $ft(\tilde S)$ do not have punctures or degenerate boundary and start with even fans (compare Lemma \ref{lem:redundant_parter_unenumerated}).
  Moreover, the surface at $ft(\tilde S)$ has at least one puncture or degenerate boundary component and starts with an even fan.
  Since $m_c = m_{\tilde c} = m_r - 1$, the transition from $c$ to $r$ uses a face that is part of an odd fan in some surface $S$ of $c$.
  
  Assume $ft(S) > ft(\tilde S)$.
  By the reasoning in the last paragraph, we see that $c$ and $r$ have a surface with foot-point $ft(\tilde S)$ that has punctures or degenerate boundary and starts with an even fan and
  the surfaces that are attached in front of $ft(\tilde S)$ also start with even fans.
  But then, $c$ cannot be collapsible by Lemma \ref{lem:redundant_parter_unenumerated}.
  This leads to a contradiction.
  
  If $ft(S) = ft(\tilde S)$, then the surface $T$ witnessing the collapsibility of $c$ has to be attached after $ft(S)$.
  Moreover, the face used in the transition from $c$ to $r$ is not taken at the fan of $S$ (compare Figure \ref{fig:pic1_lemma_a_unenumerated}).
  \begin{figure}[ht]
      \inkpic[.75\columnwidth]{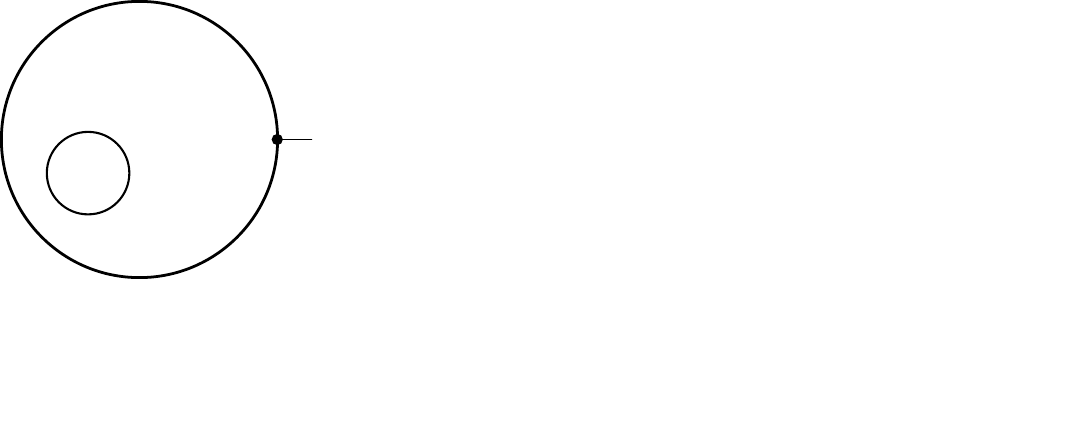}
      \caption{\label{fig:pic1_lemma_a_unenumerated} The transition $c \ad r \au \tilde c$ for $ft(S) = ft(\tilde S)$.}
  \end{figure}
  We conclude $l_c = l_r = l_{\tilde c} - 1$.
  
  If $ft(S) < ft(\tilde S)$,
  we are not taking a face at the fan of $S$ (because $S$ does not start with an odd fan by the above discussion).
  But then, the surface at $ft(S)$ of $r$ witnesses the redundancy of $r$ (all surface that are attached before start with even fans and do not have punctures or degenerate boundary).
  This contradicts the assumption, that $\tilde c$ is the partner of $r$.
\end{proof}
\subsection{Proof of Lemma B (unenumerated case)}
\label{subsection:proofs:lemma_b}
We restate Lemma \ref{lemma_b_unenumerated}, in the unenumerated case, for the convenience of the reader.
The proof is similar to the proof of Lemma \ref{lemma_a_unenumerated} above.

\begin{lemmawithfixednumber}{\ref{lemma_b_unenumerated}}[unenumerated case]
  The collapsibles and their redundant partners of type $1$ extend the discrete Morse flow (of type $0$) on $\SD_g^m$ and $\SD_{g,m}$.
  Every cell $\Sigma$ is redundant or collapsible if it has punctures or degenerate boundary or if it has a surface $S$ of positive genus $g_S > 0$ not attached at $0$ i.e., with  $ft(S) > 0$.
  
  Moreover, the stabilization map respects this discrete Morse flow, i.e.\ (1) it sends essentials to essentials, redundants to redundants and collapsibles to collapsibles and
  (2) a redundant or collapsible cell is in the sub-complex if and only if its partner is.
\end{lemmawithfixednumber}

\begin{proof}[Proof of Lemma \ref{lemma_b_unenumerated}]
  For every cell $\Sigma \in \SDgm$, we introduce the following notation.
  \begin{itemize}
    \item The number of attached surfaces is $s_\Sigma$.
    \item The number of punctures or degenerate boundary of a surface $S \in \Sigma$ is $m_S$ and
      the number of punctures or degenerate boundary of the cell $\Sigma$ is $m_\Sigma = \sum_{S \in \Sigma} m_S$.
    \item If a surface $S \in \Sigma$ starts with a fan of length $l$ we define $l_S = l$ and set $l_S = 0$ otherwise.
      The total length of fans of the cell $\Sigma$ is $l_\Sigma = \sum_{S \in \Sigma} l_S$.
    \item The genus of a surface $S \in \Sigma$ is $g_S$ and
      the reduced genus of the cell $\Sigma$ is $g_\Sigma = \sum_{S \in \Sigma, ft(S) > 0} g_S$.
    \item If a surface $S \in \Sigma$ ends with a fence of length $L$ we define $L_S = L$ and set $L_S = 0$ otherwise.
      The total length of fences of the cell $\Sigma$ is $L_\Sigma = \sum_{S \in \Sigma, ft(S) > 0} L_S$.
  \end{itemize}
  We extend the degree of degeneracy as follows.
  \begin{align}
    \degen(\Sigma) = ( \dim(\Sigma), s_\Sigma, m_\Sigma, - l_\Sigma, g_\Sigma, -L_\Sigma )
  \end{align}
  As before, we show that the lexicographical order of this sequence decreases along every path $c \ad r \au \tilde c$ in the $F$-inverted graph.
  Then, $F$ is acyclic by Lemma \ref{dmt_acyclicity}.
  
  Let us first assume that $\tilde c$ is collapsible of type $0$.
  The case where $c$ is also collapsible of type $0$ is treated in the proof of Lemma \ref{lemma_a_unenumerated}.
  Let $c$ be collapsible of type $1$.
  As before we have that $s_c \geq s_r = s_{\tilde c}$, so it is enough to study the case $s_c = s_r = s_{\tilde c}$.
  Since $c$ is of type 1, we have $m_c = 0$ and since $\tilde c$ is collapsible of type 0 we have $m_{\tilde c} = m_r - 1$.
  We may assume $m_c = 0 = m_{\tilde c}$ and therefore $m_r = 1$.
  Moreover $c$ does not contain any surface starting with an odd fan (because it is collapsible of type $1$).
  Therefore (using $s_c = s_r = s_{\tilde c}$) we get that $l_c \le l_r = l_{\tilde c} - 1$.
  We conclude $\degen(c) < \degen(\tilde c)$ if $\tilde c$ is collapsible of type $0$.
  
  We now assume that $\tilde c$ is collapsible of type $1$ and since $s_c \ge s_r = s_{\tilde c}$ we can assume further that $s_c = s_{\tilde c}$.
  Combining this with the fact that $r \au \tilde c$ does not change the number of punctures or degenerate boundary, we conclude $m_{\tilde c} = m_r = 0$ and therefore $m_c = 0$.
  Then observe that the transition $c \ad r$ cannot reduce the number $l_c$ since $s_c = s_r$ and $m_r = 0$.
  Moreover $l_r = l_{\tilde c}$, so we get $l_c \le l_r = l_{\tilde c}$.
  
  It remains to study the case where $\tilde c$ is collapsible of type $1$ and that
  \[
    \dim_c = \dim_{\tilde c}, \quad s_c = s_r = s_{\tilde c}, \quad m_c = m_r = m_{\tilde c} = 0 \quad\text{and}\quad  l_c = l_r = l_{\tilde c} \,.
  \]
  In particular, the three cells have the same sequence of initial fans.
  By Discussion \ref{disc:faces}, we either have $g_c = g_r = g_{\tilde c} + 1$ or $g_c + 1 = g_r = g_{\tilde c} + 1$.
  Thus we study the case where $c \ad r$ merges two boundaries of the same surface $S$ increasing the genus of $S$ by one.
  Similarly to the reasoning in the proof of Lemma \ref{lemma_a_unenumerated}, we see that $L$ increases by one.
\end{proof}
\subsection{Setup for the enumerated case}
\label{subsection:proofs:setup_labaled_case}

We now construct a discrete Morse flow for the enumerated case, which is in spirit similar to the unenumerated case in the sense that we will pair cells by collapsing a chamber of a fan and thus adding a puncture.
However, the structure of the matching is different, because instead of using lengths of fans we use the labels of the chambers to determine which are collapsible and which are redundant.

\begin{definition}[Unnormalized sentences]
  The symbols $\{ 1, 2, \ldots \}$ are called \emph{letters} and the symbol $\epsilon$ is a \emph{placeholder}.
  A \emph{word} is a finite sequence of letters and placeholders and we allow a word to be empty.
  We read words from right to left.
  The empty word is denoted by $()$.
  
  Let $\Sigma$ be a cell of $\betterwidetilde\SD_g^m$ or $\betterwidetilde\SD_{g,m}$.
  Denote the footpoints of the surfaces $S_k, \ldots, S_1$ by $i_k > \ldots > i_1 = 0$.
  If $S_j$ starts with a fan of length $l_j > 0$ then the enumerated boundary cycles at $i_j+l_j-1, \ldots, i_j$ define a non-empty word $w_j$.
  If $S_j$ does not start with a fan the word $w_j = ()$ is empty.
  The punctures or degenerate boundary of $S_j$ define a set of so called \emph{free letters over} $w_j$.
  This way, we obtain the \emph{unnormalized sentence} $\tilde T_\Sigma = w_k, \ldots, w_1$ and a set of free letters $Z$.
\end{definition}

\begin{example}
  \label{example:sentence_by_diagram_unnormalized}
  Consider the Sullivan diagram $\Sigma$, shown in Figure \ref{fig:sentence_by_diagram_unnormalized},
  described by the following combinatorial data:  $\lambda = (0\ l_4\ 1\ l_3\ 2) (l_2) (3\ l_1)$, $S_1 = (0,0, \{(0\ l_4\ 1\ l_3\ 2), (l_2)\})$ and $S_2 = (0,0,\{ (3\ l_1)\})$.
  The unnormalized sentence consists of two words, one for the non-trivial fan in $S_1$ and one for the trivial fan in $S_2$.
  It is $\tilde T_\Sigma = (), (3\ 4)$.
  There is a single free letter $2$, it is over the word $(3\ 4)$.
  \begin{figure}[ht]
      \inkpic[.3\columnwidth]{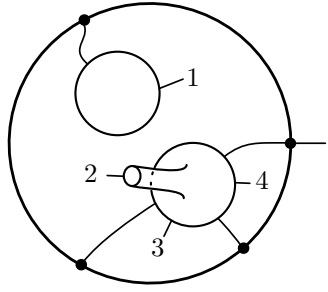}
      \caption{
        \label{fig:sentence_by_diagram_unnormalized}The Sullivan diagram $\Sigma$ given by $\lambda = (0\ l_4\ 1\ l_3\ 2) (l_2) (3\ l_1)$, $S_1 = (0,0, \{(0\ l_4\ 1\ l_3\ 2), (l_2)\})$ and $S_2 = (0,0,\{ (3\ l_1)\})$.
        Its corresponding unnormalized sentence is $\tilde T_\Sigma = (), (3\ 4)$ and its corresponding normalized sentence is $T_\Sigma = (), (\epsilon)$.}
  \end{figure}
\end{example}

The definition of the collapsibles and their redundant partners depend on an ordering of the set of the above words and the free letters.
It turns out that letters larger than the smallest free letter (if there is any) do not play a role in the discussion.
We replace them by what we denote placeholders and consecutive placeholders are seen as a single placeholder.
In order to make the presentation of the most technical step cleaner, we assume that the letters appearing in the sentence are exactly the numbers $1$ to $n$ and the free letter (if existing) is $n+1 \le m$.

\begin{definition}[Normalized sentences]
  On the set of words, we introduce the equivalence relation generated by the fusing two consecutive placeholders into one
  \[
    \ldots \epsilon \epsilon \ldots \sim \ldots \epsilon \ldots \,.
  \]
  A word is in \emph{reduced form} if it does not have two consecutive placeholders.
  Each word is equivalent to a unique reduced word.
  Abusing notation, we identify words with their reduced form.
  The \emph{length} of a word is the number of letters and placeholders of its reduced form.
  
  The set of words of the same length are linearly ordered as follows.
  We extend the canonical order on the natural numbers by $\epsilon > k \in \mathbb N$ and order the words lexicographically by their letters (read from right to left).
  
  Let $A_{n,k}$ be the set of sentences of $k$ words in the letters $1, \ldots, n$ and the placeholder $\epsilon$ such that every letter occurs exactly once in the sentence.
  We have no assumptions on the number of placeholders.
  If $k$ is fixed or known from the context, we simply write $A_n$.
  
  Let $\Sigma$ be a cell of $\betterwidetilde\SD_g^m$ or $\betterwidetilde\SD_{g,m}$ with unnormalized sentence $\tilde T_\Sigma = w_k, \ldots, w_1$ and its set of free letters $Z$.
  If $Z$ is empty, we renormalize the sentence $\tilde T_\Sigma$ such that the used letters are $1, \ldots, n$ for some $n$.
  Otherwise, there is smallest free letter, say $n+1$.
  We discard all other free letters, replace all letters larger than $n+1$ by the placeholder $\epsilon$ and renormalize the remaining letters and the free letter $n+1$ if necessary.
  In both cases we obtain the \emph{normalized sentence} $T_\Sigma \in A_n$.
\end{definition}

\begin{example}
  The Sullivan diagram discussed in Example \ref{example:sentence_by_diagram_unnormalized} has the unnormalized sentence $\tilde T_\Sigma = (), (3\ 4)$ and
  there is a single free letter $2$, it is over the first word $(3\ 4)$.
  Therefore, the normalized sentence is $T_\Sigma = (), (\epsilon)$ and the (normalized) free letter is $1$ over the first word $(\epsilon)$.
\end{example}

The Morse flow will couple cells with the same number $k$ of ghost surfaces (as in the unenumerated case).
Therefore we keep the letter $k$ fixed in what comes next.

\begin{definition}
  Consider $k$ to be fixed.
  For $n > m$ let
  \[
    \alpha_{m,n} \colon A_n \to A_m
  \]
  be the map induced by replacing the letters larger than $m$ with $\epsilon$.
  Set $\alpha_{n,n} = \id$.
  We define $B_n$, $I_n$, $J_n$ and $f_n^i$ recursively.
  Let
  \begin{align*}
    I_0 &:= \emptyset \subseteq A_0 \;, \\
    J_0 &:= \emptyset \subseteq A_0 \;, \\
    B_0 &:= A_0 - J_0 = A_0\\
  \intertext{and}
    f_0^i &\colon B_0 \xhr{} A_1 \text{ for } 1 \le i \le k
  \end{align*}
  be the map that is induced by placing the letter $1$ at the beginning of the $i^{\text{th}}$ word (i.e.\ $1$ is the right most letter afterwards).
  
We define 
  \begin{align*}
    I_n &:= \bigsqcup_{i=1}^k \imag(f_{n-1}^i) \subseteq A_n \;, \\
    J_n &:= \bigcup_{j=1}^n \alpha_{j,n}^{-1}( I_j ) \subseteq A_n \;, \\
    B_n &:= A_n - J_n = A_n - \bigcup_{\substack{1 \le j \le n\\ 1 \le i \le k}} \alpha_{j,n}^{-1}( \imag(f_{j-1}^i) ) \\
  \intertext{and}
    f_n^i &\colon B_n \xhr{} A_{n+1}
  \end{align*}
  to be the map that is induced by placing the letter $n+1$ into the $i^{\text{th}}$ word such that the resulting word has minimal order type.
\end{definition}

\begin{example}
  We provide an example for $k = 1$ and $n \le 2$.
  \label{example:def_labaled_case}
  \begin{align*}
    A_0 = \{& (), (\epsilon) \} \\
    B_0 = \{& (), (\epsilon) \} \\
    f_0^1 = \{& () \mapsto (1), \; (\epsilon) \mapsto (\epsilon 1) \} \\
    A_1 = \{& (1), (\epsilon 1), (1 \epsilon), (\epsilon 1 \epsilon) \} \\
    I_1 = \{& (1), (\epsilon 1) \} \\
    J_1 = \{& (1), (\epsilon 1) \} \\
    B_1 = \{& (1 \epsilon), (\epsilon 1\epsilon) \} \\
    f_1^1 = \{& (1 \epsilon) \mapsto (1 \epsilon 2), \; (\epsilon 1\epsilon) \mapsto (\epsilon 1\epsilon 2) \} \\
    A_2 = \{& (2 1), (\epsilon 2 1), (2 \epsilon 1), (\epsilon 2 \epsilon 1), (1 2), (\epsilon 1 2), (1 \epsilon 2), (\epsilon 1 \epsilon 2), (2 1 \epsilon), \\
            & (\epsilon 2 1 \epsilon), (2 \epsilon 1 \epsilon), (\epsilon 2 \epsilon 1 \epsilon), (1 2 \epsilon), (\epsilon 1 2 \epsilon), (1 \epsilon 2 \epsilon), (\epsilon 1 \epsilon 2 \epsilon) \} \\
    I_2 = \{& (1 \epsilon 2), (\epsilon 1 \epsilon 2) \} \\
    J_2 = \{& (2 1), (\epsilon 2 1), (2 \epsilon 1), (\epsilon 2 \epsilon 1), (1 \epsilon 2), (\epsilon 1 \epsilon 2) \} \\
    B_2 = \{& (1 2), (\epsilon 1 2), (2 1 \epsilon), (\epsilon 2 1 \epsilon), (2 \epsilon 1 \epsilon), (\epsilon 2 \epsilon 1 \epsilon), (1 2 \epsilon), (\epsilon 1 2 \epsilon), (1 \epsilon 2 \epsilon), (\epsilon 1 \epsilon 2 \epsilon) \}
  \end{align*}
  Observe that $J_2$ is the disjoint union of $\alpha_{1,2}^{-1}( I_1 ) = \{ (2 1), (\epsilon 2 1), (2 \epsilon 1), (\epsilon 2 \epsilon 1) \}$ and $\alpha_{2,2}^{-1}( I_2 ) = \{ (1 \epsilon 2), (\epsilon 1 \epsilon 2) \}$.
  This is true for all $n$ and $k$ by the next proposition.
\end{example}

\begin{proposition}
  \label{prop:set_of_sentences_decomposes}
  Let the number of sentences $k$ be fixed.
  The union $J_n = \cup_{j=1}^n \alpha_{j,n}^{-1}( I_j )$ is a disjoint union.
  In particular, the set of sentences $A_n$ decomposes as
  \begin{align}
    A_n = B_n \sqcup \bigsqcup_{\substack{1 \le j \le n\\ 1 \le i \le k}} \alpha_{j,n}^{-1}( \imag(f_{j-1}^i) )
  \end{align}
  i.e.\ for every sentence $T \in A_n - B_n$ there are unique $i$, $j$ and $T' \in B_{j-1}$ with $\alpha_{j,n}(T) = f_{j-1}^i(T')$.
\end{proposition}

This will follow from the upcoming lemmas.

\begin{lemma}
  The $j^{\text{th}}$ word of $\alpha_{n,n+1} \circ f_n^i( w_k, \ldots, w_1 )$ is $w_j$ for $j \neq i$ and $w_i$ or $\epsilon w_i$ for $j=i$.
\end{lemma}

\begin{proof}
  This is immediate from the definitions:
  The map $f_n^i$ inserts the letter $n+1$ right of the first $\epsilon$ in $w_i$ if there is any, else it is put at the left end.
  The map $\alpha_{n, n+1}$ replaces $n+1$ by $\epsilon$, i.e.\ it is either canceled by the other $\epsilon$ or it becomes an $\epsilon$ at the left end of $w_i$.
\end{proof}

\begin{lemma}
  Let $S := v_k, \ldots, v_i, \ldots, v_1$ and $S' := v_k, \ldots, \epsilon v_i, \ldots, v_1$.
  We have $S \in J_n$ if and only if $S' \in J_n$.
\end{lemma}

\begin{proof}
  We prove the claim by induction on $n$.
  The base case $n=1$ is clear and we assume the claim is true for all $n \le N$.
  By the definition of $J_{N+1}$ we assume $\alpha_{j,N+1}(S) \in \imag(f_{j-1}^l) \subset J_j$ for some $j$ and $l$.
  If $j < N+1$, then (by the induction hypothesis) this is equivalent to $\alpha_{j,N+1}(S') \in J_j$ and this implies $S' \in J_{N+1}$ by the definition of $J_{N+1}$.
  It remains to study the case $j = N+1$, i.e.\ $S = f_N^i(v_k, \ldots, v_1)$.
  Equivalently, $v_k, \ldots, v_1 \notin J_N$ and by the induction hypothesis $v_k, \ldots, \epsilon v_i, \ldots, v_1 \notin J_N$.
  But this is equivalent to $S' = f_N^i(v_k, \ldots, \epsilon v_i, \ldots, v_1)$ by the definition of $f_N^i$.
\end{proof}

\begin{lemma}
  \label{lemma:disjoint_decomposition}
  If $S \in \imag(f_{n-1}^i)$ for some $i$ and $n$, then $S \notin \imag(f_{n-1}^l)$ for any $l \neq i$.
  Moreover we have $\alpha_{n-1,n}(S) \notin J_{n-1}$.
\end{lemma}

\begin{proof}
  The first part is clear from the definition, the last part is a consequence of the two Lemmas above.
\end{proof}

\begin{proof}[Proof of Proposition \ref{prop:set_of_sentences_decomposes}]
  The proof is an immediate consequence of Lemma \ref{lemma:disjoint_decomposition}:
  Assume we have $S \in \alpha^{-1}_{j,n}( \imag(f_{n-1}^i) ) \cap \alpha^{-1}_{k,n}( \imag(f_{n-1}^l) )$.
  We are safe to assume that $j \ge k$.
  Considering $\alpha_{j,n}(S)$, it suffices to study the case $j=n$ because $\alpha_{n,n} = \id$.
  Therefore, we are left with
  \[
    S \in \imag(f_{n-1}^i) \cap \alpha^{-1}_{k,n}( \imag(f_{n-1}^l) ) \,.
  \]
  By Lemma \ref{lemma:disjoint_decomposition} and the definition of $\alpha_{k,n}$ and $J_k$, we deduce $k=n$ and therefore $l=i$.
\end{proof}

Now we define our Morse flow.

\begin{definition}[Collapsible or redundant of type $0$]
  Let $T_\Sigma$ be the normalized sentence of $\Sigma$.
  If $T_\Sigma \in J_n$ then there is a unique $j$ and $i$ with $\alpha_{j, n}( T_\Sigma ) \in \imag( f_{j-1}^i )$.
  The cell $\Sigma$ is \emph{collapsible of type $0$} and the \emph{redundant partner of type $0$} is $d_k(\Sigma)$ with $k$ corresponding to the position $i$.
  
  If $T_\Sigma \notin J_n$ and $n+1$ is a free letter then $\Sigma$ is redundant and the collapsible partner is the co-face of $\Sigma$ corresponding to the sentence $f_n^i(T_\Sigma)$ with $i$ the surface providing the free letter $n+1$.
\end{definition}

The upcoming Lemma refines the proof of Lemma \ref{lemma_a_unenumerated}.
To keep its presentation short, we assume the reader is familiar with the argumentation of the proof of Lemma \ref{lemma_a_unenumerated}.

\begin{lemmawithfixednumber}{\ref{lemma_a_enumerated}}
  The cellular graphs of $\betterwidetilde\SDs_g^m$ and $\betterwidetilde\SDs_{g,m}$ admit a discrete Morse flow such that every cell with punctures or degenerate boundary is either redundant or collapsible (except for a single cell of degree zero).
\end{lemmawithfixednumber}

\begin{proof}
  The collapsible and redundant cells of type $0$ define a matching by Proposition \ref{prop:set_of_sentences_decomposes}.
  It is clear that all cells with punctures or degenerate boundary is either redundant or collapsible (except for a single cell of degree zero which corresponds to the sentence $T_\Sigma = (1)$).
  In order to show that it is acyclic, let us assume the converse and consider a subpath $c \ad r \au \tilde c$ of a fixed loop.
  
  Using the same argument as in the proof of Lemma \ref{lemma_a_unenumerated}, the degree of degeneracy
  \[
    \degen(\Sigma) = (\dim(\Sigma), s_\Sigma, m_\Sigma, -l_\Sigma)
  \]
  is weakly decreasing among every path in the $F$-inverted graph.
  So we are safe to assume that $\degen(c) = \degen(\tilde c)$.
  
  Since $s_c = s_{\tilde c}$, both unnormalized sentences $\tilde T_c$ and $\tilde T_{\tilde c}$ have the same number of words.
  Moreover, the transition from $c$ to $\tilde r$ is by degenerating an enumerated boundary (which is part of a fan since $l_c = l_{\tilde c}$) to a degenerate boundary.
  The transition from $r$ to $\tilde c$ takes a degenerate boundary (which is a free letter) and creates a non-degenerate boundary such that the order type of the word corresponding to the fan is minimal.
  The means, we remove a letter $X$ from $\tilde T_c$ and inserting a letter $\tilde X \le X$ afterwards.
  So we are safe to assume that $X = \tilde X$.
  
  Summing up, the transition from $c$ to $\tilde c$ is by removing a letter $X$ from a word $w_i$ and inserts it afterwards (into the same word but at a different position).
  By construction of the Morse flow, the order type of the word $w_i$ is strictly decreased while all other words are constant.
  This contradicts the existence of a loop.
\end{proof}

In order to prove the stability result with respect to genus, we proceed as in the unenumerated case.
The definitions can be copied almost verbatim from Section \ref{subsection:fences_and_proof_thm_B}.
The enumerated version of Lemma B reads as follows and is proven analogously to the unenumerated case.
This completes the proof of Theorem \ref{theorem_b}.

\begin{lemmawithfixednumber}{Lemma \ref{lemma_b_unenumerated}}
  The discrete Morse flows (of type $0$) on $\betterwidetilde\SD_g^m$ and $\betterwidetilde\SD_{g,m}$ can be extended such that
  every cell $\Sigma$ is redundant or collapsible if it has punctures or degenerate boundary or if it has a surface $S$ of positive genus $g_S > 0$ that is attached at $ft(S) > 0$.
  
  Moreover, the stabilization map respects the Morse flow, i.e.\ (1) it sends essentials to essentials, redundants to redundants and collapsibles to collapsibles and
  (2) a redundant or collapsible cell is in the sub-complex if and only if its partner is.
\end{lemmawithfixednumber}

\appendix
\section{Computational results}
\label{appendix:computations}
We obtain the integral homology of $\SDs_g^m$ and $\SDs_{g,m}$ using computer software.
\begin{proposition}
  The integral homology of $\SDs_g^m$, for small parameters $2g+m$, is given by the tables \ref{result:genus_0}, \ref{result:genus_1}, \ref{result:genus_2} and \ref{result:genus_3}.
\end{proposition}

\begin{proposition}
  The integral homology of $\SDs_{g,m}$, for small parameters $2g+m$, is given by the tables \ref{result:para_genus_0}, \ref{result:para_genus_1} and \ref{result:para_genus_2}.
\end{proposition}

Our program is mainly written in \emph{Python 2.7}.
Given parameters $g$ and $m$ it produces the integral chain complexes $\SD_g^m$ and $\SD_{g,m}$.
Computing its integral homology is by far the most time consuming task.
There is a zoo of programs and libraries for this purpose.
We use a modified version of \emph{The Original CHomP Software} by Pawe\l{} Pilarczyk \cite{chomp_orig}.
Even small parameter $2g+m$ lead to integer overflows in the current version of CHomP and 
we work around this issue by forcing CHomP to use the \emph{GNU Multiple Precision Arithmetic Library} \cite{GMP}.
All results were produced on an \emph{Intel i7-2670QM} and \emph{Intel i5-4570} running \emph{Debian Sid} and \emph{Debian Jessie} with \emph{Linux Kernel 4.2.0-1-amd64}.

There is a separate program by the second author that computes the integral homology of the genus zero case using the fact that every Sullivan diagram $\Sigma \in \SD_0^m$ is the same as a weighted, non-crossing partition.
This is implemented in \emph{Magma} \cite{Magma} and our results agree with these computations.
\begin{table}[ht]
    \centering
    \begin{tabular}{c>{\raggedleft $}m{2ex}<{$}>{\raggedleft $}m{5ex}<{$}>{\raggedleft $}m{5ex}<{$}>{\raggedleft $}m{5ex}<{$}>{\raggedleft $}m{5ex}<{$}>{\raggedleft $}m{5ex}<{$}>{\raggedleft $}m{5ex}<{$}>{\raggedleft $}m{5ex}<{$}>{\raggedleft $}m{5ex}<{$}>{\raggedleft $}m{5ex}<{$}>{\raggedleft\let\newline\\\arraybackslash $}m{5ex}<{$}}
        \hline
        $m$ & H_0 & H_1 & H_2 & H_3 & H_4 & H_5 & H_6 & H_7 & H_8 & H_9 & H_{10}\\ \hline
        $1$ & \Z &  &  &  &  &  &  &  &  &  &  \\
        $2$ & \Z & \Z  &  &  &  &  &  &  &  &  &  \\ 
        $3$ & \Z &  &  & \Z &  &  &  &  &  &  &  \\ 
        $4$ & \Z &  &  & \Z &  &  &  &  &  &  &  \\ 
        $5$ & \Z &  &  &  &  & \Z &  &  &  &  &  \\ 
        $6$ & \Z &  &  &  &  & \Z &  & \Z & \Z  &  & \\ 
        $7$ & \Z &  &  &  &  &  &  & \Z &  &  & \\
        $8$ & \Z &  &  &  &  &  &  & \Z &  & \Z & \Z \\ \hline
    \end{tabular}
    \caption{The integral homology $H_\ast( \mathscr{SD}_0^m; \Z )$.}
    \label{result:genus_0}
\end{table}

\begin{landscape}
\begin{table}[ht]
    \centering
    \begin{tabular}{c>{\raggedleft $}m{2ex}<{$}>{\raggedleft $}m{9ex}<{$}>{\raggedleft $}m{9ex}<{$}>{\raggedleft $}m{9ex}<{$}>{\raggedleft $}m{9ex}<{$}>{\raggedleft $}m{9ex}<{$}>{\raggedleft $}m{9ex}<{$}>{\raggedleft $}m{9ex}<{$}>{\raggedleft $}m{9ex}<{$}>{\raggedleft\let\newline\\\arraybackslash $}m{9ex}<{$}}
        \hline
        $m$ & H_0 & H_1 & H_2 & H_3 & H_4 & H_5 & H_6 & H_7 & H_8 & H_9\\ \hline
        $1$ & \Z &  &  & \Z &  &  &  &  &  &  \\ 
        $2$ & \Z & C_2  &  & \Z &  &  &  &  &  &   \\ 
        $3$ & \Z &  &  & C_3 &  & \Z^2 & \Z &  &  &   \\ 
        $4$ & \Z &  &  & C_2 &  & \Z\oplus C_2 & C_2 & \Z^2 & \Z^2 &   \\ 
        $5$ & \Z &  &  &  &  &  & \Z & \Z^5 & \Z^3 & C_2 \\ \hline
    \end{tabular}
    \caption{The integral homology $H_\ast( \mathscr{SD}_1^m; \Z )$.}
    \label{result:genus_1}
\end{table}

\begin{table}[ht]
    \centering
    \begin{tabular}{c>{\raggedleft $}m{2ex}<{$}>{\raggedleft $}m{9ex}<{$}>{\raggedleft $}m{9ex}<{$}>{\raggedleft $}m{9ex}<{$}>{\raggedleft $}m{9ex}<{$}>{\raggedleft $}m{9ex}<{$}>{\raggedleft $}m{9ex}<{$}>{\raggedleft $}m{9ex}<{$}>{\raggedleft $}m{9ex}<{$}>{\raggedleft $}m{9ex}<{$}>{\raggedleft\let\newline\\\arraybackslash $}m{9ex}<{$}}
        \hline
        $m$ & H_0 & H_1 & H_2 & H_3 & H_4 & H_5 & H_6 & H_7 & H_8 & H_9 & H_{10} \\ \hline
        $1$ & \Z &  & \Z & C_5 &  & \Z^2 & C_3 &  &  &  &  \\ 
        $2$ & \Z & C_2  &  & C_2 &  & \Z \oplus C_2 & \Z\oplus C_2 & \Z^2 & \Z\oplus C_2 & C_2 &  \\ 
        $3$ & \Z &  &  & C_3 & C_2  &  & \Z^4 & \Z^9\oplus C_2 & \Z^4\oplus C_2 \oplus C_3^2 & \Z\oplus C_2 & \Z \\ \hline
    \end{tabular}
    \caption{The integral homology $H_\ast( \mathscr{SD}_2^m; \Z )$.}
    \label{result:genus_2}
\end{table}

\begin{table}[ht!]
    \centering
    \begin{tabular}{c>{\raggedleft $}m{2ex}<{$}>{\raggedleft $}m{9ex}<{$}>{\raggedleft $}m{9ex}<{$}>{\raggedleft $}m{9ex}<{$}>{\raggedleft $}m{9ex}<{$}>{\raggedleft $}m{9ex}<{$}>{\raggedleft $}m{9ex}<{$}>{\raggedleft $}m{9ex}<{$}>{\raggedleft $}m{9ex}<{$}>{\raggedleft $}m{9ex}<{$}>{\raggedleft $}m{9ex}<{$}>{\raggedleft\let\newline\\\arraybackslash $}m{9ex}<{$}}
        \hline 
        $m$ & H_0 & H_1 & H_2 & H_3 & H_4 & H_5 & H_6 & H_7 & H_8 & H_9 & H_{10} & H_{11}\\ \hline
        $1$ & \Z &  & \Z &  & \Z & C_{35} & \Z & \Z^5 & \Z\oplus C_2^2 \oplus C_3 &  & C_2 & C_2 \\ \hline
    \end{tabular}
    \caption{The integral homology $H_\ast( \mathscr{SD}_3^m; \Z )$.}
    \label{result:genus_3}
\end{table}

\begin{table}[ht]
    \centering
    \begin{tabular}{c>{\raggedleft $}m{2ex}<{$}>{\raggedleft $}m{5ex}<{$}>{\raggedleft $}m{5ex}<{$}>{\raggedleft $}m{5ex}<{$}>{\raggedleft $}m{5ex}<{$}>{\raggedleft $}m{5ex}<{$}>{\raggedleft $}m{5ex}<{$}>{\raggedleft $}m{5ex}<{$}>{\raggedleft $}m{5ex}<{$}>{\raggedleft $}m{5ex}<{$}>{\raggedleft\let\newline\\\arraybackslash $}m{5ex}<{$}}
        \hline
        $m$ & H_0 & H_1 & H_2 & H_3 & H_4 & H_5 & H_6 & H_7 & H_8 & H_9 & H_{10}\\ \hline
        $1$ & \Z & \Z &  &  &  &  &  &  &  &  &  \\
        $2$ & \Z & \Z  & \Z & \Z &  &  &  &  &  &  &  \\ 
        $3$ & \Z &  &  & \Z^3 & \Z^2 & \Z & \Z &  &  &  &  \\ 
        $4$ & \Z &  &  & \Z & \Z & \Z^6 & \Z^5 & \Z^2 & \Z^2 &  &  \\ 
        $5$ & \Z &  &  &  &  & \Z^7 & \Z^{10} & \Z^{13} & \Z^{11} & \Z^5 & \Z^3 \\ \hline
    \end{tabular}
    \caption{The integral homology $H_\ast( \SD_{0,m}; \Z )$.}
    \label{result:para_genus_0}
\end{table}

\begin{table}[ht]
    \centering
    \begin{tabular}{c>{\raggedleft $}m{2ex}<{$}>{\raggedleft $}m{5ex}<{$}>{\raggedleft $}m{5ex}<{$}>{\raggedleft $}m{5ex}<{$}>{\raggedleft $}m{5ex}<{$}>{\raggedleft $}m{5ex}<{$}>{\raggedleft $}m{5ex}<{$}>{\raggedleft $}m{5ex}<{$}>{\raggedleft $}m{5ex}<{$}>{\raggedleft $}m{5ex}<{$}>{\raggedleft\let\newline\\\arraybackslash $}m{5ex}<{$}}
        \hline
        $m$ & H_0 & H_1 & H_2 & H_3 & H_4 & H_5 & H_6 & H_7 & H_8 & H_9 & H_{10}\\ \hline
        $1$ & \Z & \Z &  & \Z & \Z &  &  &  &  &  &  \\
        $2$ & \Z &  &  & \Z^2 &  & \Z^3 & \Z^6 & \Z^2 &  &  &  \\ 
        $3$ & \Z &  &  &  & \Z^2 & \Z^{12} & \Z^{11} & \Z^9 & \Z^{14} & \Z^8 & \Z \\ \hline
    \end{tabular}
    \caption{The integral homology $H_\ast( \SD_{1,m}; \Z )$.}
    \label{result:para_genus_1}
\end{table}

\begin{table}[ht]
    \centering
    \begin{tabular}{c>{\raggedleft $}m{2ex}<{$}>{\raggedleft $}m{9ex}<{$}>{\raggedleft $}m{9ex}<{$}>{\raggedleft $}m{9ex}<{$}>{\raggedleft $}m{9ex}<{$}>{\raggedleft $}m{9ex}<{$}>{\raggedleft $}m{9ex}<{$}>{\raggedleft $}m{9ex}<{$}>{\raggedleft $}m{9ex}<{$}>{\raggedleft $}m{9ex}<{$}>{\raggedleft $}m{9ex}<{$}>{\raggedleft\let\newline\\\arraybackslash $}m{9ex}<{$}}
        \hline
        $m$ & H_0 & H_1 & H_2 & H_3 & H_4 & H_5 & H_6 & H_7 & H_8 & H_9 & H_{10} & H_{11} \\ \hline
        $1$ & \Z &  &  & \Z &  & \Z^2 & \Z^2 & C_3 &  &  &  & \\ \hline
        $2$ & \Z &  &  &  & \Z & \Z^3 & \Z^2 & \Z^{12} & \Z^{18} & \Z^{13} \oplus C_3^2 & \Z^{10} & \Z^4 \\ \hline
    \end{tabular}
    \caption{The integral homology $H_\ast( \SD_{2,m}; \Z )$.}
    \label{result:para_genus_2}
\end{table}
\end{landscape}

\clearpage

% Bibliography
% We use our own bibstyle which is essential alpha extended by the fields 'eprint', 'archivePrefix' and 'primaryClass'.
%This is recommended by the arXiv (see https://arxiv.org/hypertex/bibstyles/)
\bibliographystyle{alphaplus}
\bibliography{stringtopology}

\end{document}